\documentclass[11pt, a4paper,reqno]{amsart}
\pdfoutput=1
\usepackage{tikz,enumitem,rotating,latexsym,bm,stmaryrd,caption,}
\usetikzlibrary{positioning,intersections,decorations,shadings}
\usetikzlibrary{shapes}
\usetikzlibrary{arrows}
\usetikzlibrary{patterns}
\usepackage{tkz-euclide}
\usetikzlibrary{calc}
\usepackage{enumitem}   
\usetikzlibrary{shapes.geometric}
\tikzset{snake it/.style={decorate, decoration=snake}}
\tikzset{zigzag/.style={decorate, decoration=zigzag}}
\usepackage{verbatim}
\usepackage{breqn}

\definecolor{ao(english)}{rgb}{0.0, 0.5, 0.0}
\definecolor{darkgreen}{rgb}{0.0, 0.5, 0.0}

\pgfdeclareverticalshading{rainbow}{100bp}
{color(0bp)=(orange);color(25bp)=(orange);
color(50bp)=(magenta); color(75bp)=(cyan);
 color(100bp)=(cyan)}

\usetikzlibrary{decorations.pathmorphing}

	\definecolor{eng}{rgb}{0.0, 0.5, 0.0}
\definecolor{apple}{rgb}{0.55, 0.71, 0.0}
\definecolor{cadmium}{rgb}{0.0, 0.42, 0.24}
\definecolor{darkspringgreen}{rgb}{0.09, 0.45, 0.27}
\definecolor{amethyst}{rgb}{0.6, 0.4, 0.8}
\definecolor{ao}{rgb}{0.0, 0.0, 1.0}
\definecolor{atomictangerine}{rgb}{1.0, 0.6, 0.4}
\definecolor{carmine}{rgb}{0.59, 0.0, 0.09}

\definecolor{toggle}{rgb}{1.0, 0.94, 0.96}

\usepackage[normalem]{ulem}

 \newcommand{\csigma}{{{{\color{magenta}\bm\sigma}}}}
 \newcommand{\ctau}{{{{\color{cyan}\bm\tau}}}}

 \newcommand{\al}{{{{\color{magenta}\bm \sigma}}}}
 \newcommand{\crho}{{{{ \color{green!80!black}\bm\rho}}}}

\newcommand{\w}{{\underline{w}}}
\newcommand{\vvv}{{\underline{v} }}

\newcommand{\y}{{\underline{y}}}
\newcommand{\x}{{\underline{x}}}

\def\down{\vee}
\def\up{\wedge}

\usepackage{standalone}

\usepackage{etex}
 
\tikzset{
  variable line width/.style={
    every variable line width/.append style={#1},
    to path={%
      \pgfextra{%
        \draw[every variable line width/.try,line width=\pgfkeysvalueof{/tikz/thickness}] (\tikztostart) -- (\tikztotarget);
      }%
      (\tikztotarget)
    },
  },
  thickness/.initial=0.6pt,
  every variable line width/.style={line cap=round, line join=round},
}

\usepackage{todonotes}

\usepackage{tikz}
\usetikzlibrary{matrix,  intersections, calc, decorations.pathreplacing} 

\usepackage{tikz-cd}

\newlength{\superthick}
\newlength{\cornerradius}
\tikzstyle{corner}=[rounded corners=\cornerradius]
\tikzstyle{dot}=[circle, inner sep=0pt, minimum size=4.8pt]
\tikzstyle{string}=[line width=\superthick]
\tikzstyle{std}=[string,dash pattern=on 0.9pt off 0.9pt]
\definecolor{realcyan}{rgb}{0,1,1}

 \usepackage{amsmath,amsthm,amsfonts,amssymb,mathrsfs,pb-diagram}
\usepackage[
bookmarks=true,colorlinks=true,linktoc=page,citecolor=green!80!black,linkcolor=green!80!black,urlcolor=green!80!black]{hyperref}
\usepackage{caption}
\usepackage{lipsum,wasysym}
\usepackage{mathtools}
\usepackage[a4paper,margin=1in]{geometry}
\usepackage{cleveref}
 \usepackage{amsmath}
\mathchardef\mhyphen="2D
\usepackage{color}
\usepackage{xcolor}
\usepackage{ifthen}
\usepackage{sidecap}   
\definecolor{mediumblue}{rgb}{0.0, 0.0, 0.8}

\newcommand{\mptn}{{\mathscr{P}_{n}}}

\renewcommand{\geq}{\geqslant}
\renewcommand{\leq}{\leqslant}

\tikzset{wei/.style= 
{red,double=red,double
distance=0.5pt}}

\newcommand{\cp}{{\color{cyan}p}}
\newcommand{\ps}{{\color{pink}s}}
\newcommand{\vu}{{\color{violet}u}}
\newcommand{\gr}{{\color{darkgreen}r}}
\newcommand{\mq}{{\color{magenta}q}}
\newcommand{\ot}{{\color{orange}t}}
\newcommand{\rs}{{\color{red}s}}
\newcommand{\CP}{{\color{cyan}P}}
\newcommand{\MQ}{{\color{magenta}Q}}

 \newcommand{\bet}{{{\color{cyan}\boldsymbol \tau}}}

\tikzset{wei2/.style={red,double=red,double
distance=0.5pt}}

\allowdisplaybreaks
\numberwithin{equation}{section}
\usepackage{scalefnt}

\newtheorem{thm}{Theorem}[section]
\newtheorem{cor}[thm]{Corollary}

\newtheorem{lem}[thm]{Lemma}

\newtheorem{prop}[thm]{Proposition}

\newtheorem*{prop*}{Proposition}
\newtheorem*{thmA}{Theorem A}

\newtheorem*{thmB}{Theorem B}

\newtheorem*{cor*}{Corollary}

\newtheorem*{conj*}{Conjecture D}

\newtheorem*{conj1*}{Conjecture B}
\newtheorem*{Acknowledgements*}{Acknowledgements}

\newtheorem{rmk}[thm]{Remark}
\newtheorem{defn}[thm]{Definition}
\newtheorem{eg}[thm]{Example}

\newcommand{\Std}{{\rm Std}}

\newcommand{\la}{\lambda}

\newcommand{\sts}{\mathsf{s}}  
\newcommand{\stt}{\mathsf{t}}  
\newcommand{\stu}{\mathsf{u}}  

\newcommand{\ZZ}{{\mathbb Z}}

\let\<=\langle
\let\>=\rangle

\newcommand{\DP}{\underline{\mu}\la}

\tikzset{
ultra thin/.style= {line width=0.05pt},
very thin/.style=  {line width=0.2pt},
thin/.style=       {line width=0.1pt},
semithick/.style=  {line width=0.6pt},
thick/.style=      {line width=0.8pt},
very thick/.style= {line width=1.2pt},
ultra thick/.style={line width=1.6pt}
}

\crefname{ques}{Question}{Questions}
\crefname{defn}{Definition}{Definitions}
\crefname{thm}{Theorem}{Theorems}
\crefname{prop}{Proposition}{Propositions}
\crefname{lem}{Lemma}{Lemmas}
\crefname{cor}{Corollary}{Corollaries}
\crefname{conj}{Conjecture}{Conjectures}
\crefname{section}{Section}{Sections}
\crefname{subsection}{Subsection}{Subsections}
\crefname{eg}{Example}{Examples}
\crefname{figure}{Figure}{Figures}
\crefname{rem}{Remark}{Remarks}
\crefname{rmk}{Remark}{Remarks}
\crefname{equation}{equation}{equation}

\Crefname{ques}{Question}{Questions}
\Crefname{defn}{Definition}{Definitions}
\Crefname{thm}{Theorem}{Theorems}
\Crefname{prop}{Proposition}{Propositions}
\Crefname{lem}{Lemma}{Lemmas}
\Crefname{cor}{Corollary}{Corollaries}
\Crefname{conj}{Conjecture}{Conjectures}
\Crefname{section}{Section}{Sections}
\Crefname{subsection}{Subsection}{Subsections}
\Crefname{eg}{Example}{Examples}
\Crefname{figure}{Figure}{Figures}
\Crefname{rem}{Remark}{Remarks}
\Crefname{rmk}{Remark}{Remarks}

\usepackage[hang,flushmargin]{footmisc}

\begin{document}
   
 \title[Isotropic meta Kazhdan--Lusztig combinatorics I]{Isotropic meta Kazhdan--Lusztig combinatorics I: Ext-quiver presentation for the Hecke
category}

 \author{Ben Mills}
       \address{Department of Mathematics, 
University of York, Heslington, York,  UK}
\email{ben.mills@york.ac.uk}

\begin{abstract}
  	We provide an ${\rm Ext}$-quiver and relations presentation for the anti-spherical Hecke categories
	of isotropic Grassmannians in terms of     Temperley--Lieb diagrammatics.
\end{abstract}

 \maketitle

 \vspace{-0.5cm}
 \section{Introduction}

The
($p$)-Kazhdan--Lusztig polynomials arise in a wide range of contexts within representation theory, geometry, and Lie theory. Their ubiquity naturally leads to the following question: can the (often strictly combinatorial) methods used to compute Kazhdan--Lusztig polynomials be enriched so as to
shed additional light on the algebraic and geometric structures associated with them? An affirmative answer to this question in turn invites a further one: if two a priori different structures are described by the same Kazhdan--Lusztig polynomials, does this indicate the existence of a deeper relationship or even an equivalence between these  categorical structures?

The most combinatorially well-understood (parabolic) Kazhdan--Lusztig polynomials are those of Hermitian symmetric pairs \cite{lascoux1981polynomes,MR840583,Boe1988-va,brenti2009parabolic,Brundan2010-mb,MR2813567,bowman2023orientedtemperleyliebalgebrascombinatorial}.
 These polynomials control the  categories  of perverse sheaves on the (isotropic) Grassmannians  of type $A_n$, $D_n$ (and type $B_{n-1}$) \cite{Strop1,Brundan2011-ye,TypeDKhov}, 
 singular blocks of category $\mathcal{O}$ \cite{BOEANDCO}, algebraic supergroups \cite{MR1937204},
 the anti-spherical Hecke categories of types $(A_n , A_{n-k} \times A_k)$ and $(D_n, A_{n-1})$ \cite{elias2014hodgetheorysoergelbimodules,antiLW}, and the (generalised) Khovanov arc algebras \cite{MR2918294,TypeDKhov}.
These (generalised) Khovanov arc algebras are of particular interest as they serve as a conduit for moving geometric and categorical insight into the study of knot theory  \cite{Khov,TypeDKhov,Ehrig_Stroppel_2016,ehrig2017category,heidersdorf2024khovanovalgebrastypeb}. 

In this paper and its forthcoming companion, we address both of the above questions for the anti-spherical Hecke categories of isotropic Grassmannians, building on the work of \cite{ChrisDyckPaper} for the classical Grassmannians. In this paper, we answer the first question by providing an Ext-quiver and relations presentation of their basic algebra, $\mathcal{H}_{(D_n,A_{n-1})}$. 
We upgrade the existing type $(D_n,A_{n-1})$ decorated cup-cap combinatorics \cite{MR2813567,LEJCZYK_STROPPEL_2013} into a meta Kazhdan--Lusztig combinatorics rich enough to encode the full representation theoretic structure.
 The main theorem of the companion paper will utilise the Ext-quiver and relations presentation established here to answer the second question affirmatively in the case of $\mathcal{H}_{(D_n, A_{n-1})}$ and the generalised Khovanov arc algebras in type $D$ (using results of \cite{TypeDKhov} to complete the geometric picture); we do this by constructing an explicit $\ZZ$-graded isomorphism between these (basic) algebras. 
 The isomorphism itself is strikingly elementary; this simplicity serves to further emphasise the effectiveness of the aforementioned combinatorial methods 
 within Kazhdan--Lusztig theory. 

This paper adapts the methods of \cite{ChrisDyckPaper} to the isotropic case in order to allow us to express the generators of $\mathcal{H}_{(D_n, A_{n-1})}$ in terms of decorated oriented cup-cap  diagrams. This is achieved by employing the Temperley–Lieb tilings introduced in \cite{bowman2023orientedtemperleyliebalgebrascombinatorial}. We then state and prove the main theorem of this paper:

\begin{thmA}
\label{theorem A}
 For any integral domain $\Bbbk$ containing $i\in\Bbbk$ such that $i^2=-1$, the $\Bbbk$-algebra $\mathcal{H}_{(D_n, A_{n-1})}$ admits a quadratic presentation, phrased in terms of decorated oriented cup-cap combinatorics, with generators strictly in degrees $0$ and $1$, subject to the relations \eqref{rel1}--\eqref{adjacentcup}. 
If $\Bbbk$ is a field, this constitutes the presentation of $\mathcal{H}_{(D_n, A_{n-1})}$ by its Ext-quiver and relations.
\end{thmA}

We hope that this Ext-quiver and relations presentation will advance our understanding of 
related symmetric algebras  and standard extension algebras in the manner of \cite{bowman2024quiverpresentationsschurweylduality,Strop-Eber}.  
This presentation is particularly powerful because it not only enables us to understand the graded composition factors of standard modules (which are governed by the Kazhdan--Lusztig polynomials), but it also allows us to explicitly determine the extensions between composition factors, entirely within our  meta Kazhdan--Lusztig combinatorial language:

\begin{thmB}
Over a field $\Bbbk$, the submodule structure of a cell module of $\mathcal{H}_{(D_n, A_{n-1})}$ can be calculated in terms of  decorated oriented cup-cap combinatorics. Additionally, the structure is independent of the characteristic of $\Bbbk$.
\end{thmB}

An example of this module structure can be found in \cref{submodule}.

\subsection{Structure of this paper.} The goal of this paper is to adapt the methods and results of \cite{ChrisDyckPaper} into parabolic type $(D_n,A_{n-1})$; hence we will follow a similar structure to that paper and, where possible, our notation has been chosen to coincide. In \cref{Type D Combinatorics} we review the combinatorics of parabolic Coxeter systems for signed permutation groups in terms of tilings and oriented Temperley--Lieb diagrams from \cite{MR2813567,LEJCZYK_STROPPEL_2013,bowman2023orientedtemperleyliebalgebrascombinatorial}. 
In \cref{Cup combinatorics}, we develop a meta Kazhdan--Lusztig version of the decorated oriented cup-cap combinatorics.
 In \cref{Hecke} we recall the definition of the  path-algebra of $\mathcal{H}_{(D_n, A_{n-1})}$ via diagrammatic generators and relations.
 In \cref{gens} we lift the    combinatorics of 
  \cref{Type D Combinatorics,Cup combinatorics} 
  to give us a basis and generators of 
  $\mathcal{H}_{(D_n, A_{n-1})}$ 
   before providing the exact statement of \cref{theorem A}. 
 In \cref{Contraction Sect} we take a detour to  discuss a method of contracting larger
      oriented cup-cap diagrams and Hecke categories to smaller rank, which will significantly  simplify  \cref{proof} in which we provide the proof of \cref{theorem A}.
       Finally, in \cref{structure} we detail how to use the presentation from {Theorem A} to recover   the 
       full submodule structure of the cell modules of $\mathcal{H}_{(D_n, A_{n-1})}$.

 \smallskip
  \section{Type D Kazhdan--Lusztig Combinatorics}
  \label{Type D Combinatorics}
     To keep this paper as self-contained as possible, we begin by reviewing the Kazhdan--Lusztig combinatorics of parabolic Coxeter systems for signed permutation groups, borrowing heavily from the language and conventions of \cite{ChrisDyckPaper, bowman2023orientedtemperleyliebalgebrascombinatorial, ChrisHSP}.

    \subsection{Cosets, weights and tile-partitions}
    We let $W(C_n)$ denote the group of signed permutations of $\{1,\dots,n\}$  generated by  the 
    set of permutations $\{s_{0'},s_1,s_2,\dots , s_{n-1}\}$ with Coxeter relations encoded in the leftmost diagram of \cref{coxeterlabelD2}.
    We let  ${W(D_n)}$ denote the subgroup of even signed permutations of $\{1,\dots,n\}$  generated by  the 
    set of permutations $\{s_{0},s_1,s_2,\dots , s_{n-1}\}$ where $s_0=s_{0'}s_1s_{0'}\in {W(C_n)}$, with Coxeter relations encoded in the rightmost diagram of \cref{coxeterlabelD2}. 
We let ${W(A_{n-1})}$  denote the subgroup (of ${W(D_n)}\leq {W(C_n)}$)   generated by the reflections $\{s_1, s_2,\dots, s_{n-1}\}$.
     
 \begin{figure}[ht!]
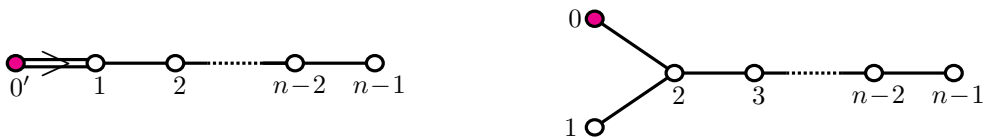

 \begin{minipage}{6cm}

 \end{minipage}
 \caption{ The Dynkin diagrams of the groups ${W(C_n)}$ and $ {W(D_n)}$ with the labelling to be used throughout this paper, with the node not belonging to the parabolic of type $A_{n-1}$ highlighted in pink.}
\label{coxeterlabelD2}
\end{figure}

 We define a \textsf{weight} to be a horizontal strip with $n$ vertices at the half-integer $x$-coordinates $\{\frac{1}{2},\frac{3}{2},...,n-\frac{1}{2}\}$ each labelled by either $\up$ or $\down $. 
 We let $\{s_i \mid 1 \leq i \leq n-1\}$ act by permuting the labels at $i-\frac{1}{2}$ and $i+\frac{1}{2}$ and fixing every other node;
we let $s_{0'}$ flip the label at the 1st node, at $x=\frac{1}{2}$, through its horizontal axis (that is $s_{0'}(\vee) = \wedge$).  
In this manner, the generator $s_{0}=s_{0'}s_1s_{0'}$ can be seen to act by 
first flipping the labels at $x=\frac{1}{2}$ and $x=\frac{3}{2}$ through the horizontal axis and then permuting their order 
(that is $s_{0}(\up\up) = \down \down$,  $s_{0}(\down\down) = \up\up$, $s_{0}(\vee\wedge) = \vee\wedge$ and $s_{0}( \wedge\vee) =  \wedge\vee$).

 For a parabolic Coxeter system $(W,P)$ we denote by $^PW$ the set of minimal length right coset representatives of $P\setminus W$. The action on weights described above allows us to label the entire set of cosets of $P\setminus W$  for $P=W(A_{n-1})$ and $ W=W(C_n)$
  via the set of all weights.
  Similarly, we can encode a set of minimal length coset representatives $P\setminus W$  for $P={W(A_{n-1})}$ and $W={W(D_n)} $ 
  via the subset of all weights 
   with an {\em even} number of vertices labelled $\up$.
  Specifically, we set the identity weight,  $\emptyset$, to be the weight consisting of  $n$ $\vee$'s. 
    The other ($2^{n-1}-1$) cosets can be obtained by permuting the labels of the identity weight.
With this identification in place, we shall, for the remainder of this paper, set $(W,P)=(W(D_n),W(A_{n-1}))$ with generators as above.

    We can also visualise the elements of $^PW$ as tilings within a restricted admissible region: $$\mathscr{A}_{(W,P)}\coloneqq\{ [r,c] \mid r,c \leq  n \text{ and }r-c\geq 0 \},$$ where we draw these tilings ``Russian Style" with rows (fixed values of $r$) pointing northwest and columns (fixed values of $c$) pointing northeast. 
Our convention will be to denote each tiling by the numbers of tiles appearing in each fixed row.
We say that two tiles
  are {\sf neighbouring} if they meet  at an edge.
 Given a pair of neighbouring tiles $X$ and $Y$, we write $Y < X$ if $X$ appears above $Y$ (i.e. the $y$-coordinate of $X$ is strictly larger than that of $Y$). This extends to a partial order on all tiles in $\mathscr{A}_{(W,P)}$ and we say $Y$ {\sf supports} $X$ if $Y<X$ in this ordering. 
 We say that a collection of tiles, $\la\subseteq \mathscr{A}_{(W,P)}$ is a
  {\sf tile-partition} if every tile in $\la$ is
   supported and we denote the set of tile-partitions $\mptn$.  
   We define the {\sf length} of a tile-partition, $\ell(\la)\in\mptn$, to be the total number of tiles $[r,c]\in \la$.  
For $\lambda  ,\mu  \in \mptn$, we define the {\sf Bruhat order} on tile-partitions by
 $\lambda  \leq \mu $ if  
$$\{[r,c] \mid [r,c] \in \lambda \}
\subseteq \{[r,c] \mid [r,c] \in \mu \}.$$

  We additionally draw a tile-partition, $\la$, by colouring the tiles of $\la$ corresponding to the generators of $W$. Each tile, $[r,c]$, inherits a coloured label from the Dynkin diagram of $W$, determined by the $x$-coordinate of the tile
   (further refined by the $x$-coordinate for $s_0$ and $s_1$), which we call the {\sf content} of $[r,c]$, ${\sf ct}([r,c])$. Specifically the tile $ [r,c] $ has colour:
$$  \begin{cases}
s_0 & \text{if $r-c=0$ and $r=2k+1$} \\
s_1 & \text{if $r-c=0$ and $r=2k$} \\
s_i & \text{if $r-c=i-1>1$.} \\
\end{cases}
$$

 \begin{figure}[ht!]
 \centering


\caption{
On the left we depict the 
 identity weight $\emptyset$ along the bottom of the diagram, 
  the weight of $\mu=(1,2,3,4,5,3,3)$ along the top of the diagram. On the right we depict the tile-partition $\mu$ which corresponds to the coset labelled by the reduced word
  $\color{gray}s_0
    \color{red}s_2  
      \color{blue}s_3
          \color{cyan}s_4
          \color{magenta}s_5
  \color{lime!80!black}s_6
      \color{yellow}s_7  
    \color{darkgreen}s_1
        \color{red}s_2  
      \color{blue}s_3
          \color{cyan}s_4
          \color{magenta}s_5
  \color{lime!80!black}s_6
      \color{gray}s_0
    \color{red}s_2  
      \color{blue}s_3
          \color{cyan}s_4
          \color{magenta}s_5
            \color{darkgreen}s_1
        \color{red}s_2  
         \color{gray}s_0$. We also number the tiles of $\mu$ as to correspond to the standard tableau $\stt_\mu$.
}
\label{typeAtiling-long}
\end{figure}

For $\la\in\mptn$, we define a standard tableau of shape $\la$ to be a bijection $\stt$ from the set of tiles of $\la$ to the set $\{1, 2, \dots, \ell(\la)\}$ such that for each $1\leq k \leq \ell(\la)$, the set of tiles $\stt^{-1}(\{1, . . . , k\})$ is a tile-partition. In practice, this is a filling of the tiles of $\la$ by the number $1$ to $\ell(\la)$ such that the numbers increase along rows and columns. We denote by ${\rm Std}(\la)$ the set of all tableaux of shape $\la$ and we let $\stt_\la \in {\rm Std}(\la)$ denote the tableau in which we fill in the tiles of $\la$ first according to increasing $y$-coordinate and then to increasing $x$-coordinate.

Given $\lambda \in \mptn $, we define the set ${\rm Add}(\lambda)$ to be the set of all tiles $[r,c]\notin \lambda$ such that $\lambda \cup [r,c]\in \mptn $. Similarly, we define the set $ {\rm Rem}(\lambda)$ to be the set of all tiles $[r,c]\in \lambda$ such that $\lambda \setminus [r,c] \in \mptn $. We note that a tile-partition $\lambda$ has at most one addable or removable tile of each content.

\begin{rmk}
Above we have detailed how to label a coset by a weight diagram or a tile-partition. We now explain how to go between the two.  
Given a weight diagram, start by reading the labels of a weight diagram from right to left. Starting at the rightmost corner of the admissible region, trace out a north-western step for each $\down$ and a south-western step for each $\up$. After $n-1$ steps, we reach the left-hand ``jagged" wall of the admissible region, having traced out the northern perimeter of the tile-partition. In particular, the identity coset corresponds to the weight diagram labelled by $n$ $\down$'s, tracing the perimeter of the empty tile-partition $\emptyset$. 
The leftmost label is not used in drawing the shape of the tile-partition, but can be used to quickly colour the top left tile. If it is $\up$, then the top left tile will be coloured $s_0$, if $\down$, then it will be coloured $s_1$.
This process can be reversed in the obvious manner.
\end{rmk}

  \begin{rmk}\label{bijection remark}There is a natural bijection between ${^P}W$ and 
$\mptn$ (see   \cite[Appendix]{MR3363009}) under which the length functions coincide. 
\end{rmk}

Throughout the paper, we will identify minimal coset representatives both with their weight diagrams and tile-partitions.

\subsection{Temperley--Lieb diagrams of type $D$}
A {\sf Temperley--Lieb  diagram of type $D_n$} is a rectangular
	frame with  $n$ vertices along   the top and $n$  along the bottom 
	which are paired off by non-crossing strands. 
	Any strand which can be deformed to the leftmost edge of the diagram while not intersecting another strand can carry a dot, $\begin{tikzpicture}\draw[very thick, fill=black](0,0) circle (1.7pt);\end{tikzpicture}$, and we refer to any strand carrying a dot as being {\sf decorated}.
   We refer to a strand connecting a top and bottom vertex as a {\sf propagating strand}
 (or {\sf propagating ray}).
We refer to any strand connecting two top vertices as a  {\sf cup} and any strand connecting two bottom vertices as a {\sf cap}.   
 In \cite{MR1618912}, Green proves that multiplication in the type $D$ Temperley--Lieb algebra, ${\rm TL(D_n)}$, can be understood as
vertical concatenation of Temperley--Lieb diagrams subject to rules for removing closed loops.
In this paper, we will only consider the multiplication of diagrams which do not result in the formation of closed loops.  
Therefore, we will only require the generators (not the relations) of this algebra, which we now recall.

    We define ${\sf e}_i$ for $i\geq 1$ to be the undecorated Temperley--Lieb diagram with  
a single pair of arcs connecting the $i^{{\text th}}$ and $(i+1)^{{\text th}}$ northern (respectively  southern) vertices,  
and with 
$(n-2)$ vertical strands.  
We set ${\sf e}_{0}$ to be the Temperley--Lieb diagram
which has a single pair of arcs connecting the $1^{{\text st}}$ and $2^{{\text nd}}$ northern (respectively  southern) vertices, both of which carry a single decoration, and 
with 
$(n-2)$ vertical undecorated strands.
Our notation has been chosen to match the Coxeter graph in \cref{coxeterlabelD2}.
Examples of these elements are depicted in \cref{gens1231231}.

  \begin{figure}[ht!]
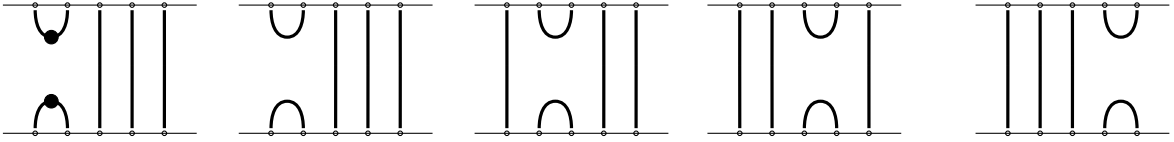

   
	\caption{
 The generators   ${\sf e}_{0}$, ${\sf e}_1$,  ${\sf e}_2$,  ${\sf e}_3$,  ${\sf e}_4$        
 of the   Temperley--Lieb algebra of  type $D_5$.
	}
	\label{gens1231231}
\end{figure} 

 \begin{defn}\label{tiling}
 For each   $\mu\in\mptn$,  we define a Temperley--Lieb diagram ${\sf e}_\mu$ created by filling the boxes of the tile-partition $\mu$ with the corresponding Temperley--Lieb generators and then isotopically deforming the strands. We note that the bottom edge of any such diagram ${\sf e}_\mu$ consists solely of decorated (non-concentric) arcs on the left and some number of vertical strands (the leftmost of which may or may not be decorated). An example is depicted in \cref{on top}.  We write $[\mu]$ for the non-deformed tiled diagram (the leftmost diagram in \cref{on top}). 
 \end{defn}
  
\begin{figure}[ht!]
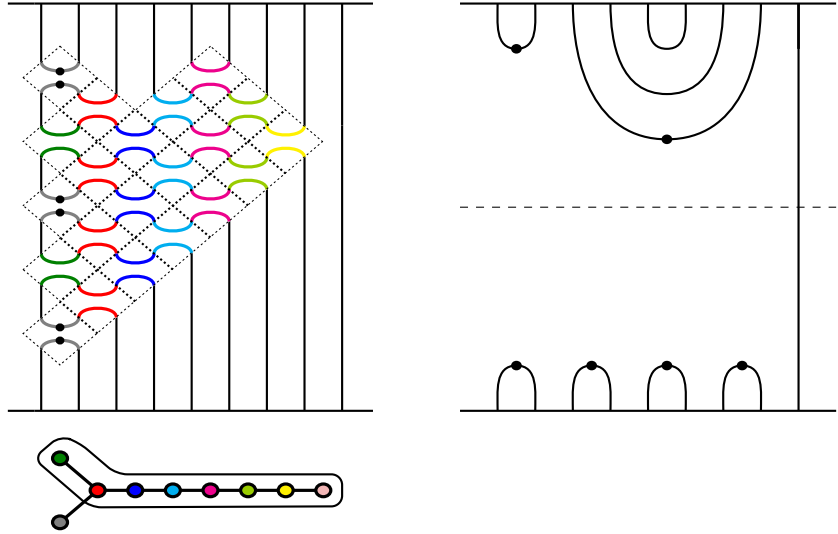

\centering
$$

$$
    \caption{On the left, we depict the tiled Temperley--Lieb diagram $[\mu]$ of type $D_9 $ $\mu=(1,2,3,4,5,3^2)$. By isotopically deforming the strands in $[\mu]$, one can obtain the Temperley--Lieb diagram, ${\sf e}_{\mu}$ on the right. }
\label{on top}
\end{figure}

\begin{defn}
\label{localness}
 Given a strand, $S$,
  in ${\sf e}_\mu$ and a tile $[r,c] \in \mu$ we say that 
  $S \cap[r,c]$ is
  \begin{itemize}[leftmargin=*]
\item  {\sf  empty} if $S$ does not intersect $[r,c]$ in the diagram $[\mu]$
\item {\sf  a local cup} if the   intersection of $S$ 
with $[r,c]$ in the diagram $[\mu]$ connects the two northern edges of the tile $[r,c]$. 
 \item {\sf  a local cap} if the   intersection of $S$ 
with $[r,c]$ in the diagram $[\mu]$ connects the two southern edges of the tile $[r,c]$. 
  \item {\sf  a local cup-cap} if the   intersection of $S$ 
 with $[r,c]$ in the diagram $[\mu]$ connects both the two northern edges and the two southern edges of the tile $[r,c]$. 
 \end{itemize}
 \end{defn}

 \begin{rmk}
 \label{even}
 A local cup-cap only occurs on strands that are decorated and within tiles $[r,c]$ with ${\sf ct}([r,c])=s_{2k}$ for some $k\in \mathbb{Z}_{>0}$. 
 \end{rmk}
 
 \begin{eg}
 \label{examplesss}
In the diagram of the tile-partition $(1,2,3)$ below, ${\color{cyan}S} \cap [3,1]$ is a local cup, ${\color{cyan}S} \cap [3,3]$ is a local cap and ${\color{cyan}S} \cap [3,2]$ is a local cup-cap.

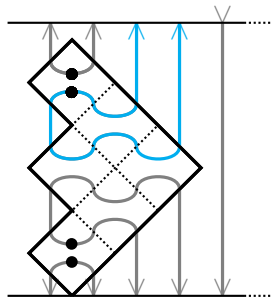
\center

  \end{eg}
 
\vspace{0.1cm}

 \begin{defn}
We let $\underline{\mu}$ denote the top half of the diagram of  ${\sf e}_\mu$ obtained by deleting the southern edge of the frame and all southern arcs (but not the propagating strands).  We refer to $\underline{\mu}$ as the {\sf cup} diagram for $\mu$.  
 We let  $\overline{\mu}$ denote the diagram obtained from  $\underline{\mu}$ by flipping it through the horizontal axis and refer to this as the {\sf cap} diagram associated with $\mu$.
 \end{defn}
 
 An example is depicted in \cref{construction}.

\subsection{Oriented Temperley--Lieb diagrams and Kazhdan--Lusztig polynomials } 
  Let $\nu,\mu\in\mptn$ and $e_\mu$ be a Temperley--Lieb diagram of type $D_n$. We can form a new diagram $\emptyset e_\mu \nu$ by glueing the weight $\emptyset$ under $e_\mu$ and the weight $\nu$ on top of $e_\mu$. 
We say that a strand connecting two weights is {\sf oriented} if one of the weights points into the strand and the other points out of the strand. A strand is {\sf flip--oriented} if its two weights either both point in or both point out of the strand.
We say that $\emptyset e_\mu\nu$ is an {\sf oriented Temperley--Lieb diagram} if each undecorated strand in $\emptyset e_\mu\nu$ is oriented and each decorated strand in $\la e_\mu\nu$ is flip-oriented.

\begin{defn}
We define the {\sf oriented cup diagram} $\underline{\mu} \la$ to be the top half of the diagram   $ \emptyset  {\sf e}_\mu   \la  $ if this is oriented (and otherwise leave this undefined).

For each strand $S$ in $\underline{\mu} \la$, we define $l_S$ to be the $x$-coordinate of the leftmost vertex of $S$ and $r_S$ to be the $x$-coordinate of the rightmost vertex of $S$. Recall that we have defined weights to be at the half-integer coordinates so $l_s, r_s \in \frac{1}{2}+\mathbb{Z}$.
 \end{defn}
 
 Some examples are depicted in \cref{removal} where the highlighted cup $\cp\in\underline{\mu}$ has $l_{\cp}=\frac{7}{2}$, $r_\cp=\frac{13}{2}$.
 
  \begin{defn}
Let $\mu\in \mptn$. Given a cup $\cp  \in \underline{\mu}$ we define its {\sf breadth}, $b(\cp)$, as follows:
 
If $\cp$ is undecorated, then $b(\cp)=\frac{1}{2}(r_\cp-l_\cp+1)$.
  
  If $\cp$ is decorated, then $b(\cp)=\frac{1}{2}(r_\cp+l_\cp)$
 \end{defn}

\begin{defn}
\label{degrees}
   Let $\lambda, \mu$ be weights such that $\underline{\mu}\lambda$ is oriented. We define the {\sf degree} of the oriented cup diagram $\underline{\mu}\lambda$ 
   to be the number of cups whose right vertex is labelled by $\down$ 
in the diagram, that is
\[\deg(\underline{\mu}\lambda)=
		\sharp\left\{  \begin{minipage}{1.2cm}
			\begin{tikzpicture}[scale=0.55]
				\draw[thick](-1,-0.155) node {$\boldsymbol\up$};
				\draw[very thick](0,0.1) node {$\boldsymbol\down$};
				\draw[very thick]  (-1,-0) to [out=-90,in=180] (-0.5,-0.75) to [out=0,in=-90] (0,0);
			\end{tikzpicture}
		\end{minipage}
			\right\}	+
\sharp \left\{	\begin{minipage}{1.2cm}
			\begin{tikzpicture}[scale=0.55]
				\draw[thick](-1,0.1) node {$\boldsymbol\down$};
				\draw[very thick](0,0.1) node {$\boldsymbol\down$};
				\draw[very thick]  (-1,-0) to [out=-90,in=180] (-0.5,-0.75) to [out=0,in=-90] (0,0);
				\draw[very thick, fill=black](-0.5,-0.75) circle (4pt);
			\end{tikzpicture}
		\end{minipage}
		\right\}.	 \]
We also refer to these cups, whose right vertex is labelled by $\down$, as being {\sf clockwise oriented}. Similarly, we refer to cups, whose right vertex is labelled by $\up$, as being {\sf anti-clockwise oriented}. In this sense, one can view a decoration on a cup as an orientation-reversing point.
 \end{defn}

For the purposes of this paper, for $p\geq 0$, we can define the $p$-Kazhdan--Lusztig polynomials of type $(W,P) = (D_n, A_{n-1})$ as follows.  
For $\la,  \mu \in \mptn $ we set
$$
{^p}n_{\la,\mu}(q)= 
\begin{cases}
q^{\deg(\underline{\mu} \la)}		&\text{if $ \underline{\mu} \la $ is oriented}\\
0						&\text{otherwise.}
\end{cases}
$$
We refer to \cite[Theorem 7.3]{bowman2023orientedtemperleyliebalgebrascombinatorial} 
and \cite[Theorem A]{ChrisHSP}
for a justification of this definition and to \cite{MR2813567,MR2918294} 
for the origins of this combinatorics.

\begin{rmk}
\label{flipper}
We say we {\sf flip} a cup if we reflect both the vertices of the cup through the $x$-axis so as to change the orientation of the cup from anti-clockwise to clockwise or vice versa.
 It is clear that for a fixed $\mu\in \mptn $, the diagram $\underline{\mu}\lambda$ is oriented of degree $d$ if and only if the weight $\lambda$ is obtained from the weight $\mu$ by flipping $d$ anti-clockwise cups in $\underline{\mu}$. There are precisely 2 ways of doing this; either flipping undecorated or decorated anti-clockwise cups. Pictorially, this looks like
$$
(i)\begin{minipage}{1.2cm} 
\end{minipage}.
$$
 Moreover, the degree of $\underline{\mu}\la$ is precisely the total number of such swapped pairs. 
See \cref{flipit} for an example of a cup diagram of degree 2.
\end{rmk}

 \begin{figure}[ht!]
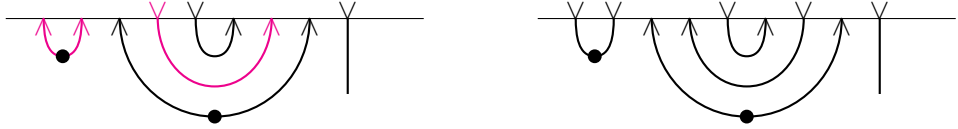

 $$
 %
$$
\caption{
  The degree of $\underline{\nu}\sigma$ on the right is two, which is precisely the number of (highlighted) flipped cups from $\underline{\sigma}\sigma$ on the left.
  }
\label{flipit}
\end{figure}

Let $\mu,\lambda \in \mptn$ be such that $\la$ is obtained from $\mu$ by flipping a single cup $\cp \in \underline{\mu}$. Then we say $ \lambda=\mu - \color{cyan}{p}$. An example is below in \cref{removal}. We also write $\mu = \la +\cp$.
 
  \begin{figure}[ht!]
 $$
 %
$$
\caption{
  The flipping of the vertices of the cup $\cp$ from $\underline{\mu}$, on the left, creates the weight $\la$ on the right. Hence we say $\la=\mu-\cp$, where we have drawn $\underline{\la}$ using \cref{drawcups}
  }
\label{removal}
\end{figure}

The following proposition was proven in \cite{bowman2023orientedtemperleyliebalgebrascombinatorial}
but this is because our treatment is non-chronological; this proposition can also be found as a   definition 
 in  \cite{MR2813567,TypeDKhov}.

  \begin{prop} 
  \label{drawcups}
 For each weight $\lambda$ we can construct the {\sf cup diagram} $\underline{\lambda}$ 
 using the following algorithm.
 \begin{enumerate}[leftmargin=*,label=(\roman*)]
\item Find a pair of vertices labelled $\down$ $\up$ in order from left to right that are neighbours in the sense that there are only vertices already joined by cups in between. Join these new vertices together with an undecorated cup.
\item Repeat this process until there are no more such $\down$ $\up$ pairs. 
\item Then, starting from the left, connect neighbouring vertices labelled with $\up$ $\up$ with a decorated cup. 
\item Finally, draw propagating rays down at all the remaining $\up$ and $\down$ vertices. If there exists a (necessarily unique) $\up$ vertex, then decorate the corresponding propagating ray. 
\end{enumerate}
\end{prop}

  \begin{figure}[ht!]
  $$   
$$
\caption{The construction of the cup diagram $\underline{\la}$ for $\la$ as in \cref{drawcups}.
	 See also \cite[Proposition 7.1]{bowman2023orientedtemperleyliebalgebrascombinatorial}.}
\label{construction}
	\end{figure}

  \subsection{ Cup flipping on tile-partitions.}
  \label{cups gone}

We have seen that for  $\la,\mu\in \mptn $, the cup diagram $\underline{\mu}\lambda$ is oriented of degree $d$ if and only if the weight $\lambda$ is obtained from the weight $\mu$ by flipping $d$ cups in $\underline{\mu}$. We wish to consider what this corresponds to in terms of the tiled Temperley--Lieb diagrams.
We have seen that this can be done in two possible ways:

\begin{enumerate}
\item We can reorient an anti-clockwise  $\down\up$ cup , $S$, to obtain a $\up\down$ cup;
here the tile-partition $\la$ is obtained by deleting the set of all $[r,c]\in[\mu]$ such that 
$S\cap [r,c]$ is a local cup or a local cap. This removes a path of the shape below.

 $$\begin{minipage}{3.55cm}

	\end{minipage}$$

\item Or, we can reorient a decorated $\up\up$ cup, $S'$ to obtain a decorated $\down\down$ cup;
here $\la$ is obtained by deleting the set of all $[r,c]\in[\mu]$ such that 
$S'\cap [r,c]$ is either a local cup, a local cap, or a local cup-cap, 
and additionally one extra tile directly below each cup-cap. 
 This is depicted as a union of pink and blue tiles below. 
In more detail, each blue zigzag shape consists of four tiles: one local cup, one local cap, one local cup-cap, and a tile whose intersection with $S$ is empty.
\end{enumerate}

     $$  \begin{minipage}{1.3cm}
   %
\end{minipage}$$

In keeping with the terminology from \cite{ChrisDyckPaper,bowman2024quiverpresentationsschurweylduality} we refer to the tiles in $[\mu]$ whose removal corresponds to the flipping of weights of the cup $\cp \in \underline{\mu}$ as the \textsf{Dyck path generated by $\cp$}, denoted $\CP \subseteq [\mu]$, to make it clear we are talking about the set of tiles in $[\mu]$. See the pink and blue coloured tiles in \cref{gen Dyck} for an example of paths $\MQ$ and $\CP$ corresponding to decorated and undecorated cups, respectively.

\begin{defn}
Let $\CP \subseteq [\mu]$ be a Dyck path generated by the cup $\cp\in \underline{\mu}$. In the diagram $[\mu]$, starting at $l_\cp$, consider the tile first intersected with when travelling vertically downwards. We define \textsf{${\sf first}(\CP)$} to be the content of the first tile that the strand $\cp$ intersects $\mu$ with, when starting at $l_{\cp}$.
Similarly we define \textsf{${\sf last}(\CP)$} to be the content of the first tile that $\cp$ intersects $\mu$ with, when starting at $r_{\cp}$.
\end{defn}

For example, the Dyck path $\CP$, which is generated by the cup $\cp$, depicted in \cref{examplesss}, has ${\sf first}(\CP)=2$ and ${\sf last}(\CP)=3$.
\begin{rmk}
Note that  \textsf{${\sf last}(\CP)$} will always be the rightmost tile of $\CP$. If $\cp$ is undecorated, then \textsf{${\sf first}(\CP)$} is the leftmost tile of $\CP$ but the same does not hold for $\cp$ decorated.
\end{rmk}

 \begin{defn}
For a Dyck path, $\CP \subseteq [\mu]$ generated by the cup $\cp$, define the \textsf{content} of $\cp$ (resp. $\CP$), denoted {$\rm{ct}(\cp)$ (resp. $\rm{ct}(\CP)$}) for short, as the multi-set of the contents of the tiles in $\CP$.
 \end{defn}

 \begin{rmk}
For an undecorated cup $\cp$, we have that ${\sf ct} (P)= [ {\sf first}(\CP), {\sf last}(\CP)]$.  For a decorated cup $\cp$, ${\sf ct} (\CP)=\{1,1,2,2,\dots,{\sf first}(\CP),{\sf first}(\CP)\}\cup\{{\sf first}(\CP)+1,{\sf first}(\CP)+2,\dots, {\sf last}(\CP)\}$.
 \end{rmk}

\begin{defn}
\label{genuine}
Let $\mu \in \mptn $ and let $[\mu]$ be the corresponding tile diagram. 
Let $S$ be a strand in $\underline{\mu}$.
Given $[r,c] \in [\mu]$ 
if $S\cap [r,c]$ is a local cup we define 
$\langle [r,c]\rangle_\mu$ to be equal to $S$.

If $S\cap [r,c]$ is not a cup then we leave $\langle [r,c]\rangle_\mu$ undefined, even if $S\cap [r,c]$ is non-empty. 
See \cref{gen Dyck} for an example of this.

\begin{figure}[ht!]
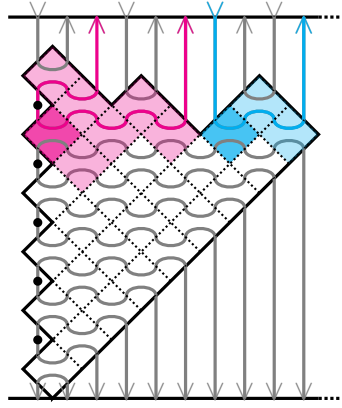


\caption{ The cups $\langle r,c\rangle _\mu$ for $[r,c]  =\color{magenta} [5,5]$
and 
$\color{cyan} [8,2]$ for $\mu=(1,2,3,4,5,6,4,2,2)$. 
We note that, for instance, the strand $\langle 6,5\rangle _\mu$ does not exist as it is a local cup-cap, defined in \cref{localness}.
}
\label{gen Dyck}
\end{figure}

\end{defn}

      \section{Oriented Cup combinatorics}
      \label{Cup combinatorics}
   
In this section, we set up the meta Kazhdan--Lusztig combinatorics of oriented cup diagrams that will govern the presentation and relations we define for the Hecke category of type $(D_n, A_{n-1})$ we will present later. We should note that this paper is an attempt to extend and adapt the work of \cite{ChrisDyckPaper} where $(W,P)=(A_{n+m}, A_n \times A_m)$ to $(W,P)=(D_n, A_{n-1})$ and therefore where possible we will use the same name for definitions that arise in the same manner.

 \begin{defn} 
\label{non-com:def} 
 Let $\mu \in \mptn$ and let $\cp,\mq \in \underline{\mu}$.
 We say that $\cp$ {\sf covers} $\mq$ and we write $\mq$ $\prec \cp$ if and only if 
    $l_{\cp} < l_{\mq}$ and $r_{\cp} > r_{\mq}$. See \cref{covering1,covering2} for examples.
    \end{defn}

    \begin{figure}[ht!]
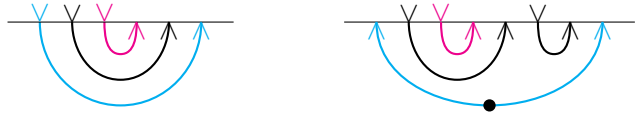

  $$   
$$
		 \caption{Examples of commuting pairs $\cp$ and $\mq$ such that $\mq \prec \cp$.}
 \label{covering1}
\end{figure}

\begin{figure}[ht!]
$$
 %
$$
 \caption{Examples of non-commuting pairs $\cp$ and $\mq$ such that $\mq \prec \cp$.}
 \label{covering2}
\end{figure}

    \begin{defn}
\label{doublenon-com:def} 
 Let $\mu \in \mptn$ and let $\cp,\mq \in \underline{\mu}$.
 We say that $\cp$ {\sf doubly covers} $\mq$ and write $\mq \prec \prec \cp$ if and only if $\cp$ is decorated and $ r_\mq< l_\cp$.  See \cref{doublecover2,doublecover1} for examples.
\end{defn}

    \begin{figure}[ht!]
     $$   %
	$$
		 \caption{Examples of commuting pairs $\cp$ and $\mq$ such that $\mq \prec \prec \cp$.}
 \label{doublecover1}
\end{figure}  
\begin{figure}[ht!]
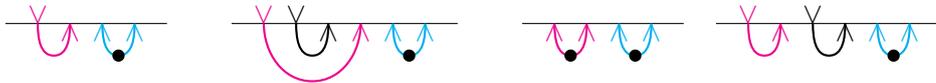


	$$  %
$$ 
    
    \caption{Examples of non-commuting pairs $\cp$ and $\mq$ such that 
    $\mq \prec \prec \cp$.}
    \label{doublecover2}
    \end{figure}

\newpage

  \begin{lem}
  Let $\mu\in \mptn$. For any $\cp \in \underline{\mu}$
   $$ b(\cp)=\sharp \{\cp \cap [r,c] \mid [r,c]\in  [\mu] \text { and }\cp \cap [r,c] \text{ is a local cup or a local cup-cap}\}.	$$
 \end{lem}
 
 \begin{proof}
 There is always one local cup in $\CP$ and an additional local cup in $\CP$ for each cup covered by $\cp$.
 Additionally, if $\cp$ is decorated, there is one zigzag block in $\CP$ for each cup doubly covered by $\cp$, each block adding one local cup and one local cup-cap.
 The lemma then follows.
 This can also be deduced by comparing the images at the beginning of \cref{cups gone} and \cref{gen Dyck}. 
 \end{proof}

Next, we will present a rather cumbersome definition of commuting cups, but for the intuition, one should simply skip straight to \cref{commie} and \cref{commie2}.
       \begin{defn}\label{dist:def}
  Let $\mu\in\mptn$ and $\cp, \mq \in \underline{\mu}$ and suppose without loss of generality that $l_{\mq} < l_\cp$. 
  We say that $\mq$ and $\cp$  commute if 
  either
$(i)$ there exists $\gr $ such that 
$\mq \prec 	\gr		\prec \cp$,
   $(ii)$ there exists $\gr $ such that 
$\mq \prec 	\gr		\prec\prec \cp$ 
or 
$\mq \prec\prec 	\gr		\prec\prec \cp$ 
or $(iii)$ neither  $\mq$  nor     $\cp$ (doubly) covers the other.  
  \end{defn}

 \cref{covering1,doublecover1} above consist of (doubly)  commuting pairs of cups 
 and  \cref{covering2,doublecover2} consist of (doubly) non-commuting pairs of cups.
   
 \begin{lem}
 \label{commie}
Let $\mu\in \mptn $ and let $\cp, \mq\in\underline{\mu}$.  
Then $\cp$ and $\mq$ commute if and only if $\CP\cap \MQ=\emptyset$.
\end{lem} 
 
 \begin{proof}
 This follows directly from the definition and by drawing tiled Temperley--Lieb diagrams as in \cref{gen Dyck}. The conditions (i)-(iii) are equivalent to $\CP\cap \MQ=\emptyset$, where we note that $\CP\cap \MQ=\emptyset$ if and only if there are strands in between $\cp$ and $\mq$ on $[\mu]$.
 \end{proof}
 
\begin{rmk}
\label{commie2}
Let $\mu\in\mptn$ and $\cp,\mq\in \underline{\mu}$. Then $\cp$ and $\mq$ commute if and only if $\cp \in \underline{\mu-\mq}$ and $\mq \in \underline{\mu-\cp}$. This remark will follow naturally from the following lemmas and \cref{general}.
\end{rmk}
 
   \begin{defn} \label{adj:def2}
   Let $\mu, \la \in \mptn$ be such that $\la=\mu-\cp$ for some $\cp\in \underline{\mu}$. Let $\ot\in \underline{\la}$.
 We say that $\cp$ and $\ot$ are {\sf adjacent} if $\ot$ has precisely one node in common with $\cp$.
\end{defn}

\begin{lem} \label{adj lem0}
Let $\la,\mu\in \mptn$ be such that $\la=\mu-\cp$. Suppose $\cp$ is not (doubly) covered in $\underline{\mu}$. Moreover, if there is no undecorated strand to the left of $\cp$ and no decorated strand to the right of $\cp$, then $\cp$ is adjacent to zero cups in $\underline{\la}$.  The following is a complete list of the pairs for $\cp\in \underline{\mu}$ and $\underline{\la}$ that can be pictured on two vertices:

    $$ a) \Biggl( 
    \begin{minipage}{1.4cm}
\end{minipage}\Biggr)$$

 \end{lem}

\begin{lem} \label{adj lem1}
Let $\la,\mu\in \mptn$ be such that $\la=\mu-\cp$. If $\cp$ is not (doubly) covered, then $\cp$ can be adjacent to at most one cup $\ot$. 

The following is a complete list of the adjacent pairs $\cp\in \underline{\mu}$ and $\ot \in \underline{\la}$ that can be pictured on three vertices:
    $$ a)  \Biggl( 
    \begin{minipage}{1.9cm} %
  \end{minipage}\Biggr) 
	 $$

 \end{lem}          
 
\begin{lem} \label{adj lem2}
Let $\la,\mu\in \mptn$ be such that $\la=\mu-\cp$. If $\cp$ is (doubly) covered in $\underline{\mu}$, then $\cp$ is adjacent to two cups, $\gr, \ot \in \underline{\la}$. The following is a complete list of the pairs $\cp\in \underline{\mu}$ and $\gr,\ot \in \underline{\la}$ that can be pictured on four vertices:
 \begin{itemize}[leftmargin=*] 
\item   Firstly, we consider the case where $\cp$ is covered by a non-commuting cup, $\vu\in\underline{\mu}$ (we additionally say that $\cp$ is {\sf covered}):
  $$   
a) \Biggl( 
    \begin{minipage}{2.4cm}  
 \end{minipage}\Biggr).$$ Notice that the two cups $\gr, \ot$ on the right-hand side are commuting pairs and the total number of decorations is the same in $\underline{\mu}$ and $\underline{\la}$ in both cases $a)$ and $b)$.

\item Secondly, we consider the 
case that  $\cp$ is doubly covered by a non-commuting decorated cup, $\vu \in \underline{\mu}$,  (we additionally say that $\cp$ is {\sf doubly covered}):

    $$  c)\Biggl( 
    \begin{minipage}{2.4cm} %
 \end{minipage}\Biggr).$$ 
	 Notice that the two cups $\gr, \ot$ on the right-hand side are non-commuting pairs with $\gr \prec \ot$.
\end{itemize}

 \end{lem}
 
 In \cref{adj lem2} there is a (purple) cup $\vu\in\underline{\mu}$ which, due to the flipping of the vertices of $\cp$, no longer appears in $\underline{\nu}$ and shares one node each with $\gr$ and $\ot$. This is the smallest possible cup (doubly) covering $\cp$. As we will argue below, this cup is uniquely defined even for  more general cases that need $n>4$ vertices to be pictured.
 
 \begin{defn} \label{gen def}
   Let $\mu, \la \in \mptn$ be such that $\la=\mu-\cp$. Let $\cp$ $\in \underline{\mu}$ and $\ot \in \underline{\la}$ be adjacent.
Suppose $\vu\in \underline{\mu}$, $\cp \neq \vu$ is the cup such that $l_\vu=l_\ot$ or $r_\vu=r_\ot$. Then we say that $\vu$ is the \textsf{cup generated by $\cp$ and $\ot$}, denoted $\langle {\color{cyan}p}\cup\ot\rangle_\mu$.
   \end{defn}

We should note that we are intentionally not defining $\langle {\color{cyan}p}\cup {\color{darkgreen}r}\rangle_\mu$ in the case where then $\cp$ is adjacent to a non-commuting pair $\gr, \ot$ of cups and $\gr \prec \ot$ - in this case $l_\vu$=$r_\gr$ or $r_\vu$=$l_\gr$. 
      
\begin{rmk}
\label{general}
Clearly, a cup $\cp\in \underline{\mu}$ can be adjacent to zero, one or two cups in $\underline{\la}$. We will list a complete set of cases below, which follows immediately from \cref{adj:def2}, but can also be deduced from \cite[Lemma 3.9]{TypeDKhov}.
In what follows, we depict each case using the minimal number of vertices. In general, $\underline{\mu}$ may contain additional cups $\ps$ that are unaffected by the removal of $\cp$ (see \cref{general}). In \cref{Contraction Sect}, we discuss a method for reducing these diagrams, explaining why it suffices to consider the minimal cases presented here.
\label{general}
The cases covered in \cref{adj lem2,adj lem1,adj lem0} are the pictured on the minimal number of vertices of $\underline{\mu}$. One can create more cases on a larger number of vertices by adding (possibly many) cups that are (doubly) covered by $\vu$ and/or $\cp$. However each of these cases reduces simply to one of the cases above. For example,
$$ 
 \Biggl( 
    \begin{minipage}{4.3cm}  
 \end{minipage}\Biggr).$$is equivalent to \cref{adj lem2} b) and 
	$$  
	 \Biggl( 
    \begin{minipage}{5.2cm}  %
 \end{minipage}\Biggr)$$is equivalent to \cref{adj lem2} c). One can similarly generalise the remaining cases in \cref{adj lem0,adj lem1,adj lem2}.

In \cref{Contraction Sect} we will explain a method that will allow us to move between weight, cups and Hecke categories of different sizes, removing these smaller (doubly) covered cups, which will allow us to restrict our focus as much as possible to the cup diagrams pictured in \cref{adj lem0,adj lem1,adj lem2}.
\end{rmk}

The final lemma in this section is a simple rephrasing of \cref{flipper}.
\begin{lem}
Let $\la,\mu, \nu \in \mptn$ and suppose $\underline{\mu}\la$ is oriented with $\la= \mu - \sum_{i=1}^{m}\cp^i$, $\cp^i \in \underline{\mu}$. Suppose for some $\ot \in \underline{\la}$ that $\nu= \la - \ot$, then  $\underline{\mu}\nu$ is oriented if and only if $\ot$ is not adjacent to any $\cp^i$.
\end{lem}

\section{The diagrammatic Hecke category}
\label{Hecke}
 
In this section we recall the construction of the main object studied in this paper: the Hecke categories associated to $(W,P)=(D_n, A_{n-1})$.

\subsection{The Hecke categories }
\label{coulda}
Let  $S = \{s_i  \mid 0\leq i \leq n-1\}$ denote the set of simple reflections generating $W$ with the relations encoded in the Coxeter graph shown in \cref{coxeterlabelD2}. 
  We define the {\sf  Soergel generators} to be the framed graphs 
$$
{\sf 1}_{\emptyset } =
\begin{minipage}{1.5cm} 
\end{minipage} 
    $$ for any pair $\csigma, \ctau\in S$ satisfying $\csigma\ctau= \ctau\csigma$. 
    Pictorially, the {\sf duals} of these generators are defined by reflecting the corresponding diagrams through the horizontal axes. Algebraically, this corresponds to swapping subscripts and superscripts. We will also denote this duality by  $\ast$.
   For example, the dual of the fork generator is depicted as
$$
  {\sf fork}^{\csigma\csigma}_{\csigma}=  
  \begin{minipage}{1.8cm}  
 %
\end{minipage}.$$ 
We define the northern and southern reading words of a Soergel generator (or its dual) to be the word in the alphabet $S$ obtained by reading the colours along the northern and southern edges of the frame (reading from left-to-right), respectively.
Given two (dual) Soergel generators $D$ and $D'$, we define $D\otimes D'$ to be the diagram obtained by horizontal concatenation (and we extend this linearly).  
The northern and southern colour sequences of $D\otimes D'$ is then given by concatenating that of $D$ and $D'$ ordered from left to right.  
For any two (dual) Soergel generators, we define their product $D\circ D'$ (or simply $DD'$) to be the vertical concatenation of $D$ above $D'$ if the southern reading word of $D$ is equal to 
 the northern reading word of $D'$; otherwise, the product is defined to be zero. We define a Soergel graph to be any diagram obtained from the Soergel generators and their duals by repeated horizontal and vertical concatenation, considered up to isotopy.

 For each $\csigma\in S$, we define  the associated ``barbell", ``dork" and ``gap'' diagrams to be the elements  
 $${\sf bar}(\csigma)=  
  {\sf spot}_ \csigma^\emptyset
   {\sf spot}^ \csigma_\emptyset
   \qquad
   {\sf dork}^{\csigma\csigma}_{\csigma\csigma}= 
      {\sf fork}^{\csigma\csigma}_{ \csigma}
            {\sf fork}_{\csigma\csigma}^{ \csigma} \qquad {\sf gap}(\csigma) = {\sf spot}^\csigma_\emptyset {\sf spot}^\emptyset_\csigma.$$  Given $\sts ,\stt\in\Std(\la)$ for $\la \in \mptn$, we write 
${\sf braid}^{\sts}_{\stt }$ for the permutation mapping one tableau to the other, with strands coloured according to the contents. For example 
$${\sf braid}^{\sts}_{\stt }=\begin{minipage}{3cm}\begin{tikzpicture}[scale=1.2]
\draw[densely dotted, rounded corners] (-0.25,0) rectangle (2.25,1);
\draw[gray,line width=0.08cm](0,0)--++(90:1);
\draw[blue,line width=0.08cm](1,0) to   (1.5,1);

\draw[red,line width=0.08cm](0.5,0)--(0.5,1);
\draw[cyan,line width=0.08cm](1.5,0)-- (2,1);
\draw[darkgreen ,line width=0.08cm](2,0) to [out=90,in=-90] (1,1);
\end{tikzpicture}\end{minipage} 
\qquad
{\sf braid}^{\stt}_{\stu }=\begin{minipage}{3cm}\begin{tikzpicture}[scale=1.2]
\draw[densely dotted, rounded corners] (-0.25,0) rectangle (2.25,1);
\draw[gray,line width=0.08cm](0,0)--++(90:1);
\draw[red,line width=0.08cm](0.5,0) to  (0.5,1);

\draw[blue,line width=0.08cm](1,0)--(1,1);
\draw[darkgreen,line width=0.08cm](1.5,0)-- (2,1);
\draw[cyan,line width=0.08cm](2,0)--(1.5,1);
\end{tikzpicture}\end{minipage} 
\qquad 
{\sf braid}^{\sts}_{\stu}=\begin{minipage}{3cm}\begin{tikzpicture}[scale=1.2]
\draw[densely dotted, rounded corners] (-0.25,0) rectangle (2.25,1);
\draw[gray,line width=0.08cm](0,0)--++(90:1);
\draw[red,line width=0.08cm](0.5,0) -- (0.5,1);

\draw[blue,line width=0.08cm](1,0)--(1.5,1);
\draw[darkgreen,line width=0.08cm](1.5,0)-- (1,1);
\draw[cyan,line width=0.08cm](2,0)-- (2,1);
\end{tikzpicture}\end{minipage} 
$$
for the three  $\sts,\stt,\stu \in \Std(1,2,1^2)$.
In practice, we will often drop the tableaux and record only the underlying content sequence.

  \begin{rmk} 
The cyclotomic quotients of (anti-spherical) Hecke categories are small categories  
with finite-dimensional morphism spaces given by the light leaves basis of \cite{MR3555156,antiLW}.     
Working with such a category is equivalent to working with a locally unital algebra, as defined in \cite[Section 2.2]{BSBS}, see  \cite[Remark 2.3]{BSBS}.  
   Throughout this paper, we adopt the latter perspective.  
   Readers who prefer a categorical viewpoint may equivalently reformulate all results in terms of categories and representations of categories.  
   \end{rmk}

 The following simplification of the presentation of the Hecke category is obtained using \cite[Theorem 2.1]{ChrisHSP}

 \begin{defn}\label{easydoesit}   \renewcommand{\vvv}{{\underline{w} }} 
\renewcommand{\w}{{\underline{x}}}
\renewcommand{\x}{{\underline{y}}}
\renewcommand{\y}{{\underline{z}}}
\newcommand{\zz}{{\underline{w}}}
 
 We define  $\mathscr{H}_{(D_n, A_{n-1})}   $ to be the locally-unital associative $\Bbbk$-algebra spanned by all  Soergel graphs with multiplication given by $\circ$-concatenation, subject to the following local relations and their vertical and horizontal flips.  First,  for any   $ \csigma, \ctau \in  S$  we have the idempotent relations 
\begin{align*}
{\sf 1}_{\csigma} {\sf 1}_{\ctau}& =\delta_{\csigma,\ctau}{\sf 1}_{\csigma} & {\sf 1}_{\emptyset} {\sf 1}_{\csigma} & =0 & {\sf 1}_{\emptyset}^2& ={\sf 1}_{\emptyset}\\
{\sf 1}_{\emptyset} {\sf spot}_{\csigma}^\emptyset {\sf 1}_{\csigma}& ={\sf spot}_{\csigma}^{\emptyset} & {\sf 1}_{\csigma} {\sf fork}_{\csigma\csigma}^{\csigma} {\sf 1}_{\csigma\csigma}& ={\sf fork}_{\csigma\csigma}^{\csigma} &
{\sf 1}_{\ctau\csigma}  {\sf  braid}_{\csigma\ctau}^{\ctau\csigma}
 {\sf 1}_{\csigma\ctau}  & ={\sf  braid}_{\csigma\ctau}^{\ctau\csigma} 
\end{align*}
where the final relation holds for all ordered pairs  $( \csigma,\ctau)\in S^2$  with $\csigma\ctau=\ctau\csigma$.  
For each  $\csigma \in  S$  we have the fork-spot contraction, the double-fork, and circle-annihilation  
 relations:  
\begin{align*} 
({\sf spot}_\csigma^\emptyset \otimes {\sf 1}_\csigma){\sf fork}^{\csigma\csigma}_{\csigma}
=
{\sf 1}_{\csigma}
 \qquad 
  ({\sf 1}_\csigma\otimes {\sf fork}_{\csigma\csigma}^{ \csigma} )
({\sf fork}^{\csigma\csigma}_{\csigma}\otimes {\sf 1}_{\csigma})
=
{\sf fork}^{\csigma\csigma}_{\csigma}
{\sf fork}^{\csigma}_{ \csigma\csigma}
\qquad  {\sf fork}_{\csigma\csigma}^{\csigma}
 {\sf fork}^{\csigma\csigma}_{\csigma}=0 
\end{align*}
   For   $(\csigma, \ctau ,\crho)\in S^3$ 
with   $\csigma\crho=\crho\csigma,
 \crho\ctau=\ctau\crho , \csigma\ctau=\ctau\csigma$,  we have the commutation relations 
 \begin{align*} 
 {\sf spot}^{ \csigma }_{\emptyset} \otimes {\sf 1}_{\crho} = 
{\sf braid}^{\csigma\crho}_{\crho\csigma}
( {\sf 1}_{\crho} \otimes {\sf spot}^{ \csigma }_{\emptyset} )
 \qquad
(  {\sf fork}^{ \csigma  \csigma}_{ \csigma} \otimes {\sf 1}_{\crho}  ) 
{\sf braid}^{\csigma\crho}_{\crho\csigma}
=
{\sf braid}^{\csigma\csigma\crho}_{\crho\csigma\csigma}
 (     {\sf 1}_{\crho} \otimes {\sf fork}^{ \csigma  \csigma}_{ \csigma} ) 
   \end{align*}
 \begin{align*} 
 {\sf braid}^{\csigma\ctau\crho}_{ \csigma\crho\ctau }
 {\sf braid}^ { \csigma\crho\ctau } _{  \crho\csigma\ctau } 
 {\sf braid}^ {  \crho\csigma\ctau }  _{  \crho \ctau\csigma }  
=
 {\sf braid}^{\csigma\ctau\crho}_{\ctau \csigma\crho  }
 {\sf braid}^ {\ctau \csigma\crho  } _{  \ctau  \crho  \csigma} 
 {\sf braid}^ {   \ctau  \crho  \csigma }  _{  \crho \ctau\csigma }.
     \end{align*}
    For $\csigma,\ctau \in S$ with $\csigma \ctau \csigma= \ctau \csigma \ctau$, we have the one and two colour Demazure relations,
\begin{align*}
     {\sf bar}(\csigma)\otimes {\sf 1}_\csigma
     +
 {\sf 1}_\csigma\otimes           {\sf bar}(\csigma) 
& =2 {\sf gap}(\csigma)
\\      {\sf bar}(\ctau)\otimes {\sf 1}_\csigma
     -
 {\sf 1}_\csigma\otimes           {\sf bar}(\ctau) 
& =		 {\sf 1}_\csigma\otimes           {\sf bar}(\csigma) 	- {\sf gap}(\csigma)
  \end{align*}
     and the null-braid relation
\begin{align*} 
 {\sf 1}_{\csigma\ctau\csigma} + 
 ({\sf 1}_\csigma \otimes {\sf spot}^\ctau _\emptyset \otimes  {\sf 1}_\csigma )
 {\sf dork}^{\csigma\csigma}_{\csigma\csigma}
 ( {\sf 1}_\csigma\otimes 
  {\sf spot}_\ctau ^\emptyset \otimes  {\sf 1}_\csigma )
 =0.
\end{align*}
   Further,  we need the interchange law and the monoidal unit relation 
$$
 \big( {\sf D  }_1 \otimes   {\sf D}_2   \big)\circ  
\big(  {\sf D}_3  \otimes {\sf D }_4 \big)
=
 ({\sf D}_1 \circ   {\sf D_3}) \otimes ({\sf D}_2 \circ  {\sf D}_4)
\qquad {\sf 1}_{\emptyset} \otimes {\sf D}_1={\sf D}_1={\sf D}_1 \otimes {\sf 1}_{\emptyset}
$$
for all diagrams ${\sf D}_1,{\sf D}_2,{\sf D}_3,{\sf D}_4$.  
 Finally, we need the   non-local cyclotomic   relations  
 \begin{align*} 
{\sf 1}_\csigma \otimes D=0
  \qquad 
  {\sf bar}(\ctau) \otimes D=0
\end{align*}
 for all $ \csigma \in S $, with $\csigma\neq s_0$, $\ctau \in S$  and all diagrams,  $D$.

 \end{defn}

 \begin{defn}
   We   define the idempotent truncation  $${\sf 1}_{n}= \sum_{\la\in \mptn } {\sf 1}_{\stt_\la}\qquad 
    \mathcal{H}_{(D_n, A_{n-1})}= {\sf 1}_{n} \mathscr{H}_{(D_n, A_{n-1})} {\sf 1}_{n}$$ 
 
 \end{defn}
 
\begin{rmk} In \cite[Theorem 7.2]{ChrisHSP} the algebra $   \mathcal{H}_{(D_n, A_{n-1})}$ is shown to be a basic algebra. \end{rmk}

   \begin{rmk} 
The algebras $\mathscr{H}_{(D_n, A_{n-1})}$ and $\mathcal{H}_{(D_n, A_{n-1})}$ can be equipped with 
 a $\mathbb Z$-grading which is compatible with the duality~$\ast$.  The degrees  of  the generators with respect to this grading are defined  as follows:
 $$
 {\sf deg}({\sf 1}_\emptyset)=0
 \quad
  {\sf deg}({\sf 1}_\al)=0
  \quad
  {\sf deg} ({\sf spot}^\emptyset_\al)=1
    \quad
  {\sf deg} ({\sf fork}^\al_{\al\al})=-1
    \quad
  {\sf deg} ({\sf braid}^{\al\bet}_{\bet\al} )=0
 $$
for $\al,\bet \in S $   such that $\csigma\ctau=\ctau\csigma$.
 \end{rmk}

   \section{ Generators }
\label{gens}

In this section, as in \cite[Section 5]{ChrisDyckPaper}, we use the meta Kazhdan--Lusztig combinatorics of \cref{Cup combinatorics} to provide a new set of generators for the basic algebra of the Hecke category, $\mathcal{H}_{(D_n, A_{n-1})}$.  These generators all lie in degree $0$ or $1$ and, later on, they will provide us with a quadratic presentation for $\mathcal{H}_{(D_n, A_{n-1})}$. 
 
 \subsection{Light leaves combinatorics via tableaux} 
 
We will rephrase the usual combinatorics of the light leaves basis for the Hecke category in terms of oriented Temperley--Lieb diagrams $\emptyset e_\mu \la$. In \cite{bowman2023orientedtemperleyliebalgebrascombinatorial} they construct a graded algebra structure on the space of all oriented Temperley--Lieb diagrams. In particular, in type $(D_n, A_{n-1})$ this is constructed by overlaying oriented Temperley--Lieb diagrams on a tile-partitions as in \cref{tiling}. This overlaying assigns each tile of the tableau not only a number between $1$ and $\ell(\mu)$, defined by the tableau $\stt_\mu$, but also with one of four possible orientations obtained from the weight, $\la$, determined by the strands passing through the northern and southern edges of the tile. Hence we will view products of these oriented Temperley--Lieb diagrams as products of \lq oriented tableau' $\stt_\mu^\la$\color{black}.

\begin{defn}\label{ortab}
Let $\la, \mu \in \mptn $ such that  $\underline{\mu}\la$ is oriented. Draw the tiled Temperley--Lieb diagram $[\mu]$ as in \cref{flipit}. Glueing $\emptyset$ and $\la$ on the bottom and top of the diagram, respectively defines one of four possible orientations on each tile of $[\mu]$. 
    Throughout, we assume decorations to be `orientation-reversing' (in the sense of \cref{degrees}).
We define the orientation label of a tile 
 as follows:
  $$ 
 $$We then define the oriented tableau $\stt_\mu^\la$ to be the map which assigns to each tile $[r,c]\in \mu$ a pair $(k, x)$ where $k=\stt_\mu ([r,c])$ and $x\in \{ 1, s , f ,sf\}$ is the orientation label of the tile $[r,c]$.
For an example of a labelled tile-partition, see \cref{twopics}.

Throughout, we will refer to $\mu$ as the {\sf shape} of the oriented tableau, $\stt_\mu^\la$, and to $\la$ as the {\sf weight} of $\stt_\mu^\la$.
\end{defn}

 \subsection{Soergel diagrams from  oriented Temperley--Lieb diagrams }
We now revisit the classical definition of the light leaves basis starting from oriented Temperley--Lieb diagrams.  This material is covered in detail in \cite{ChrisHSP}.  
Parts of this section are very similar to \cite[Section 5]{ChrisDyckPaper}, with the main difference being the change in the shape of oriented tableaux, so while the material will be familiar to the reader with prior knowledge of the combinatorics of \cite{ChrisDyckPaper}, in order to be self-contained we will restate the results from that section within the framing of $\mathcal{H}_{(D_n,A{n-1})}$.

 \begin{defn}
 We define up and down operators on diagrams as follows.
 Let $D$ be any Soergel graph with northern colour sequence $\sts\in {\rm Std}(\la)$ for some $\la\in  \mptn $, with $\ell(\la)=m$.
  \begin{itemize}[leftmargin=*]
\item    
Suppose that
${\color{magenta}[r,c]} \in {\rm Add}(\la)$ with ${\rm ct}([r,c])=\csigma \in W$.  We define  $$
\qquad {\sf U}^1_\csigma(D)=\begin{minipage}{1.85cm}\begin{tikzpicture}[scale=1.5]
\draw[densely dotted, rounded corners] (-0.5,0) rectangle (0.75,0.75) node [midway] {$ D$} ;
 \end{tikzpicture}\end{minipage}
 \begin{minipage}{0.75cm}\begin{tikzpicture}[scale=1.5]
\draw[densely dotted, rounded corners] (-0.25,0) rectangle (0.25,0.75);
\draw[magenta,line width=0.08cm](0,0)--++(90:0.75);
 \end{tikzpicture}\end{minipage}
 \qquad 
 \qquad 
 {\sf U}^0_\csigma(D)=
\begin{minipage}{1.99cm}\begin{tikzpicture}[scale=1.5]
\draw[densely dotted, rounded corners] (-0.5,0) rectangle (0.75,0.75) node [midway] {$D$} ;
 \end{tikzpicture}\end{minipage}
\begin{minipage}{0.75cm}\begin{tikzpicture}[scale=1.5]
\draw[densely dotted, rounded corners] (-0.25,0) rectangle (0.25,0.75);
\draw[magenta,line width=0.08cm](0,0)--++(90:0.3525) coordinate (hi);
\draw[fill=magenta,magenta] (hi) circle (3pt);
 \end{tikzpicture}\end{minipage}.
$$
 
\item   Now suppose that  
${\color{magenta}[r,c]} \in {\rm Rem}(\la)$ with ${\sf ct}([r,c])=\csigma \in S$.
We let  $\stt \in  \Std (\la-{\color{magenta}[r,c]})$ be defined as follows: 
if $\sts({\color{magenta}[r,c]})=k$, then we let  
 $\stt^{-1}(j)=\sts^{-1}(j)$ for $ 1\leq j <k$ and 
  $\stt^{-1}(j-1)=\sts^{-1}(j)$ for $ k<j \leq m $. 
We  let $\stt \otimes \csigma  \in \Std(\la)$ be defined by $
(\stt \otimes \csigma)({\color{magenta}[r,c]})=m$ and $(\stt\otimes \csigma)([x,y])=\stt([x,y])$ otherwise. 
We define 
$$
 {\sf D}_\csigma^0(D)= 
\begin{minipage}{1.85cm}
\end{minipage}\quad.$$

\end{itemize}
\end{defn}

 \begin{defn} Let $\la, \mu \in  \mptn $ be such that $\underline{\mu}\la$ is oriented. Recall the oriented tableau $\stt_\mu^\la$ from \cref{ortab}. Suppose that $\stt_\mu^\la[r_k,c_k] = (k,x_k)$ for each $1\leq k\leq \ell(\mu)$. We construct the Soergel graph $D_\mu^\la$ inductively as follows. 
Set $D_0$ to be the empty diagram.  
Now let $1\leq k\leq \ell(\mu)$ with $\csigma = r_k-c_k$.
Set 
  $$D_k= \left\{ \begin{array}{ll} {\sf U}^1_{\csigma }(D_{k-1}) & \mbox{if $x_k = 1$;}\\[6pt] 
  
  {\sf U}^0_{\csigma }(D_{k-1}) & \mbox{if $x_k = s$;}\\[6pt] 
 {\sf D}^0_{\csigma }(D_{k-1}) & \mbox{if $x_k = f$;}\\[6pt] 
  {\sf D}^1_{\csigma }(D_{k-1}) & \mbox{if $x_k = sf$}. \end{array}\right.
  $$
 Now $D_{\ell(\mu)}$ has southern colour sequence $\stt_\mu$ and northern colour sequence some $\sts\in {\rm Std}(\la)$. We then define $D_\mu^\la = {\sf braid}_\sts^{\stt_\la} \circ D_{\ell(\mu)}$. 
 We also define $D^\mu_\la = (D^\la_\mu)^\ast$.

\end{defn}

\begin{eg}
 In \cref{twopics} we provide an example of  the labelling of an
 oriented Temperley--Lieb diagram and the corresponding up-down construction of a light leaves basis element.

\end{eg}

\begin{figure}[ht!]
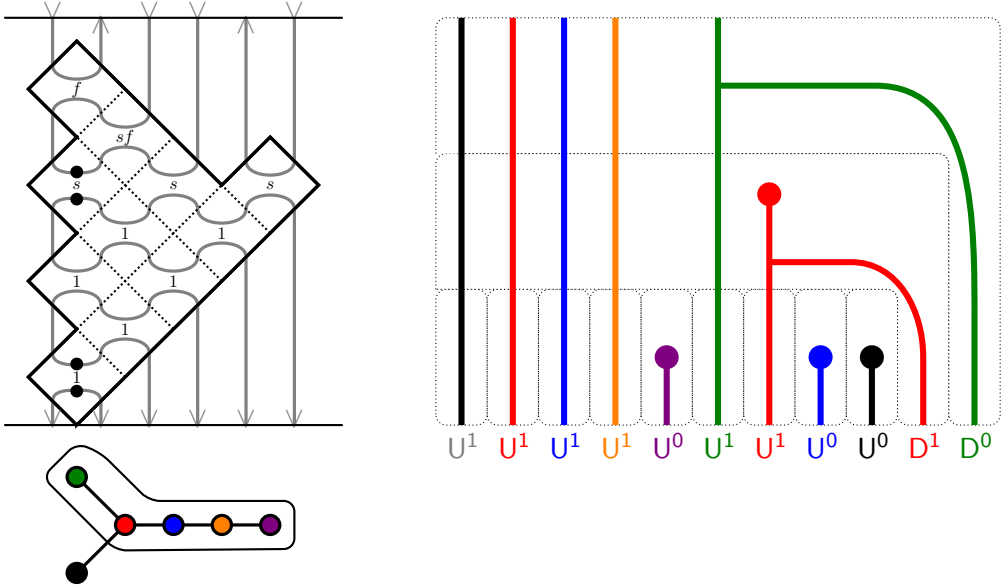

$$
%
$$
\caption{ On the left, we depict the labelling of the oriented  Temperley--Lieb diagram of weight $\la=(1,2,1,1,1)$ and shape $\mu=(1,2,3,4,1)$.
On the right we depict the unique $D^\la_\mu$ for  $t_{(1,2,3,4,1)}\in \Std((1,2,3,4,1))$. }
\label{twopics}
\end{figure}

 \begin{thm}\label{cellular basis}
 The  algebra   
$\mathcal{H}_{(D_n,A{n-1})}  $  is a graded cellular
algebra with a graded cellular basis given by
\begin{equation}\label{basis}
\{D_\la^\mu D^\la_\nu \mid  \la,\mu,\nu\in \mptn  \,\, \mbox{with}\,\,		\underline{\mu}\la,\underline\nu\la \text{ oriented} \}
\end{equation}
with 
$${\rm deg}( D_\la^\mu D^\la_\nu) = {\rm deg}(\underline{\mu}\la) + {\rm deg}(\underline{\nu}\la),$$
with respect to the involution $*$ and the partial order on $\mptn $ given by inclusion.
\end{thm}
\begin{proof}
This is a combinatorial rephrasing of the light leaves basis, constructed in full generality for arbitrary Coxeter systems in  \cite{MR3555156,antiLW}.
\end{proof}

Note, in particular that the degree $0$ basis elements are given by $D_\la^\la = {\sf 1}_\la$, $\la\in \mptn$, and the degree $1$ basis elements are given by $D_\mu^\la$ and $D_\la^\mu$ for $\la, \mu \in \mptn $ such that $\la = \mu - \cp$, for some $\cp\in \underline{\mu}$.  In what follows we will show that these degree 0 and degree 1 elements generate all of $\mathcal{H}_{(D_n, A_{n-1})}$.

\subsection{Multiplying generators on the oriented tableaux}
\label{Go forth and multiply}
We have seen how to determine a Soergel graph $D_\mu^\la$ from the oriented tableau $\stt_\mu^\la$. We now wish to be able to compute the multiplication of two such Soergel graphs directly from the oriented tableaux. To do this it becomes easier to consider pairs of tableaux of the same shape. This can be achieved by adding a fifth possible orientation for tiles, $0$. When constructing the corresponding Soergel graph, whenever we encounter a tile with $0$-orientation, we simply tensor with the empty Soergel graph; that is we leave the graph unchanged (see \cref{zeroorientation,zeroorientation2} for two examples). 

For example, let $\mu, \nu \in \mptn$ be such that $\nu=\mu -\cp$ for some $\cp \in \underline{\mu}$. Additionally suppose, for the moment, that $\mq \in \underline{\mu}$ commutes with $\cp$ and hence $\mq \in \underline{\nu}$ too. We would like to be able to visualise the product $$D^{\mu - \cp -\mq}_{\mu - \cp} D^{\mu - \cp}_\mu$$solely using oriented tableaux. Instead of considering the oriented tableau $\stt_{\mu - \cp}^{\mu - \cp - \mq}$ as a labelling of the tiles of $\mu - \cp$, we can visualise it as a labelling of the tiles of $\mu$ with all tiles belonging to $\CP$ having $0$-orientation. 
The orientation of all other tiles remains unchanged. See the left-hand tableaux in \cref{abba} for an example of this.
We then are able to multiply the elements $D^{\mu - \cp -\mq}_{\mu - \cp}$ and $D^{\mu - \cp}_\mu$ just by  \lq stacking' the two oriented tableaux, without any need to apply a braid generator in between the two diagrams.
 
 \begin{figure}[ht!]
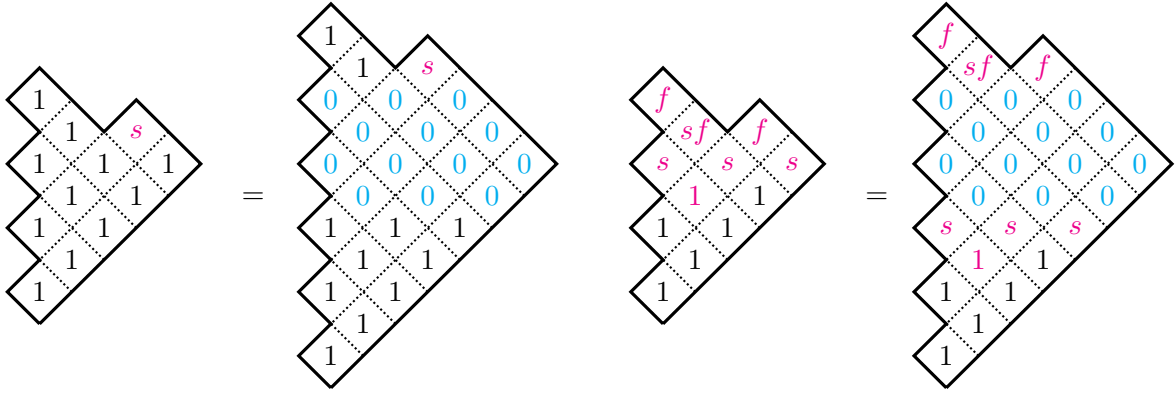

$$
\begin{minipage}{3cm}

  \end{minipage}$$
\caption{Redrawing tableaux of shape $\la=(1,2,3,4,2)$ as tableaux of shape $\mu=(1,2,3,4,5,6,4)$ where $(1,2,3,4,2)=(1,2,3,4,5,6,4)-\cp$, where $\CP$ is depicted with blue zeroes. Note that in both equalities $\mq\prec\prec\cp$. In the first equality $\mq,\cp\in \underline{\mu}$ do commute but in the second equality $\mq,\cp\in \underline{\mu}$ don't commute, so when drawing the tableau $\stt_\mu^{\la-\mq}$ the $s$ and $1$-oriented tiles of $\MQ$ fall two places to be below $\CP$.}
\label{abba}
\end{figure}

\begin{figure}[ht!]
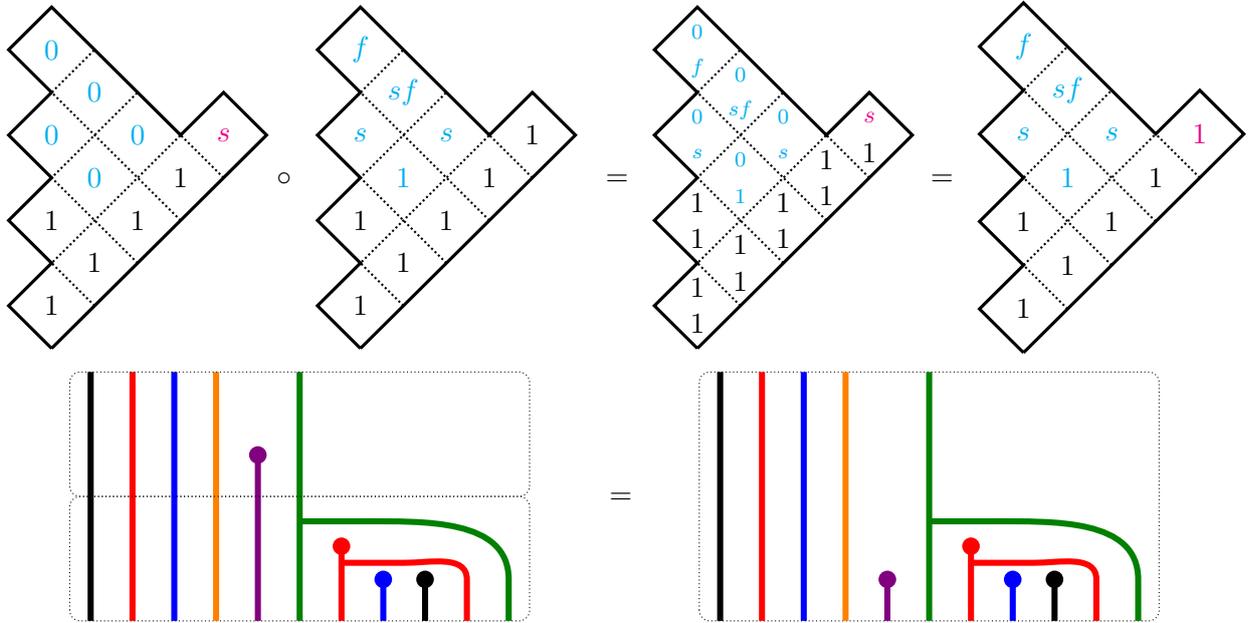

$$
\begin{minipage}{3.55cm}
\end{minipage}
$$
\caption{A product of commuting diagrams on tableaux.}
\label{zeroorientation}
\end{figure}
Now,  let  $\cp,\mq \in \underline{\mu}$ and $\mq\in \underline{\nu}$ be such that $\cp$ and $\mq$ do not commute in $\underline{\mu}$ (this necessitates that either $\mq \prec \cp$ or $\mq \prec \prec \cp$). 
We proceed as above, rewriting the tiles in $\CP$ so as to have a $0$-orientation, and then we let each tile in $\MQ$ with an $s$-orientation or $1$-orientation fall down one place (from $[r,c]$ to $[r-1,c-1]$) if $\cp$ and $\mq$ are non-commuting (as in \cref{abba2}) or two places (from $[r,c]$ to $[r-2,c-2]$) if $\cp$ and $\mq$ are doubly non-commuting (as on the right-hand side of \cref{abba}), so that they fall below $\CP$. We leave all other tiles unchanged.  
When multiplying two elements in this way, we will represent the product by splitting each tile in half with the label of the top half corresponding to the left element and the label of the bottom half corresponding to the right element in the multiplication.  
Examples of this are shown in the first equalities of \cref{zeroorientation,zeroorientation2}.

\begin{figure}[ht!]
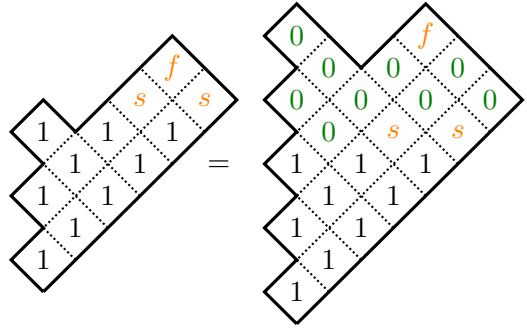

$$
\begin{minipage}{2.5cm}
  \end{minipage}$$
\caption{We redraw a tableau of shape $(1,2,3,2,2,2)$ and weight $(1,2,3,2,1)=(1,2,3,2,2,2)-\ot$ as a tableau of shape $(1,2,3,4,5,3,3)=(1,2,3,2,2,2)+\gr$. The tiles ${\color{darkgreen}R}$ (resp. ${\color{orange}T}$) are depicted with green (resp. orange) zeroes. Note the s-oriented tiles only fall one place in this case as $\ot\prec\gr$ and $\ot,\gr$ do not commute.}
\label{abba2}
\end{figure}

When considering the dual Soergel graphs $D_\la^\mu = (D_\mu^\la)^*$, we will represent them via the same oriented tableau as $D_\mu^\la$ except that we will replace all $s$-orientation (respectively $f$-orientation, or $sf$-orientation) by the symbol $s^*$ (respectively $f^*$, or $sf^*$). An example of this is given in \cref{zeroorientation2}.

\begin{prop}
\label{morepics}
The relations depicted in \cref{spotpic1,spotpic2} hold.

\begin{figure}[ht!]
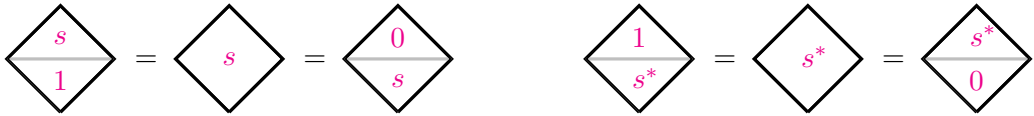


$$
\begin{minipage}{1.6cm}
 
\end{minipage}$$
 
\caption{The (dual) spot-idempotent relation  ${\sf spot}_\csigma^\emptyset  
{\sf 1}_\csigma=  {\sf spot}_\csigma^\emptyset={\sf 1}_\emptyset{\sf spot}_\csigma^\emptyset$ and its dual depicted   on tableaux. 
}
\label{spotpic2}
\end{figure}

\begin{figure}[ht!]
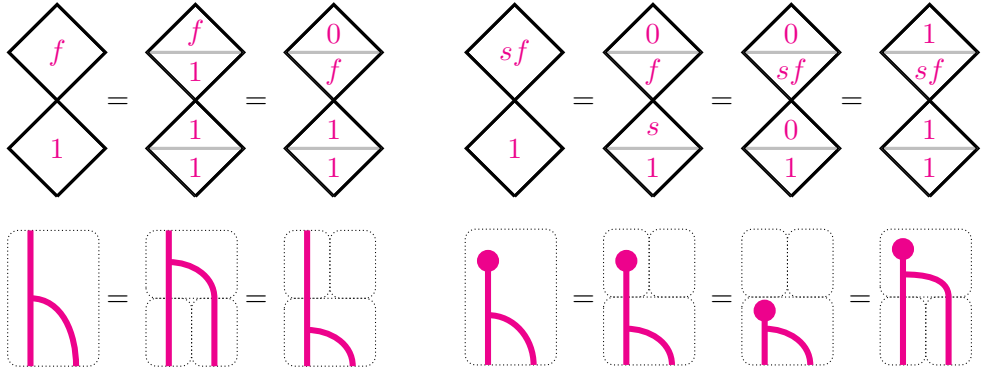

$$
\begin{minipage}{1.2cm}
%
\end{minipage} 
  $$

\caption{A few fork (and fork-spot) relations on tableaux.  One can take the dual of these relations in the same manner as \cref{spotpic2}.}
\label{spotpic1}
\end{figure}

\end{prop}

\begin{proof}
These are all restatements of the relations given in the monoidal presentation of the Hecke category in \cref{Hecke}.
\end{proof}

\begin{rmk}
\label{blank space baby}
In any Soergel diagram we can always isotope strands around blank space (if required). As a consequence we obtain many more relations from \cref{morepics} where the resulting diagram Soergel diagram contains blank space.
For example, the relation given in \cref{spotpic3} is simply a consequence of the final equality of \cref{spotpic1}.

\begin{figure}[ht!]
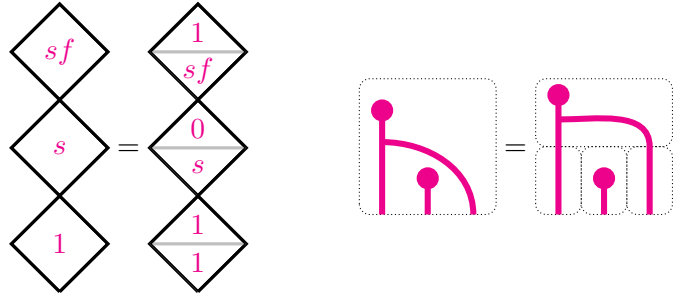

$$\begin{minipage}{1.2cm}
\end{minipage}
$$
\caption{The final relation of \cref{spotpic1} applied with a tile with empty northern reading word in between the tiles of the relation.}
\label{spotpic3}
\end{figure}
\end{rmk}

 \subsection{Generators for the Hecke category}
     We are now able to prove that the algebra $\mathcal{H}_{(D_n A_{n-1})}$ is generated in degrees 0 and 1.
\begin{prop}\label{generatorsarewhatweasay}

The algebra $\mathcal{H}_{(D_n A_{n-1})}$ is generated by the elements 
$$\{D^\la_\mu,
D_\la^\mu \mid 	
\la, \mu \in \mptn  \text{ with $\la = \mu - \cp$ for some $\cp \in \underline{\mu}$} 	\}\cup\{ D_\mu^\mu = {\sf 1}_\mu \mid \mu \in \mptn 	\}.
$$
\end{prop}

 \begin{proof}
  It is enough to show that for  all $\la, \mu\in \mptn $ such that $\DP$ is  oriented, $D_\mu^\la$  can be written as a product of these elements. We will proceed by induction on $k = {\rm deg}(\underline{\mu}\la)$. If $k=0$ or $1$, there is nothing to prove.
       If $k \geq 2$, assume that $\DP$ is oriented where $\la= \mu -\sum_{i=1}^{k} \cp^i$; 
  we let $\mq =\cp^j$, for some $1\leq j\leq k$, 
   be such that there does not exist $\cp^i$ with $\mq \prec \cp^i$ 
    or $\mq \prec \prec \cp^i$ for $i\neq j$. Such a cup will always exist; indeed we can take $\mq$ to be the cup satisfying $r_\mq={\rm max}\{r_{\cp^i}|1\leq i \leq k\}$.
  Then we claim that $$D_\mu^\la = D_{\mu - \mq}^\la D_{\mu}^{\mu - \mq}$$ and the result will then immediately follow from the claim by induction.
To verify the claim, note that the oriented tableau, $\stt_{\mu -\mq}^\la$, viewed as a tableau of shape $\mu$, as explained in the last subsection, is obtained from $\stt_\mu^\la$ by setting the orientation of all tiles of $\MQ$ to $0$. 
If there is a $\cp^i\neq \mq$ that does not commute with $\mq$, then each of the s-orientations or 1-orientations on the tiles of $\CP^i$ fall down one or two tiles, to be below $\MQ$.
By assumption on $\mq$, there are no tiles in $\stt_\mu^\la$ lying below $\MQ$ labelled anything other than $1$, so this is perfectly fine.
Using the relations given in \cref{morepics}, we see that $D_{\mu - \mq}^\la D_{\mu}^{\mu - \mq} = D_\mu^\la$ as required. 
The dual element $D_\la^\mu = (D^\la_\mu)^*$ can then be written as the reverse product of the dual degree 1 elements. 
 \end{proof}

\begin{figure}[ht!]
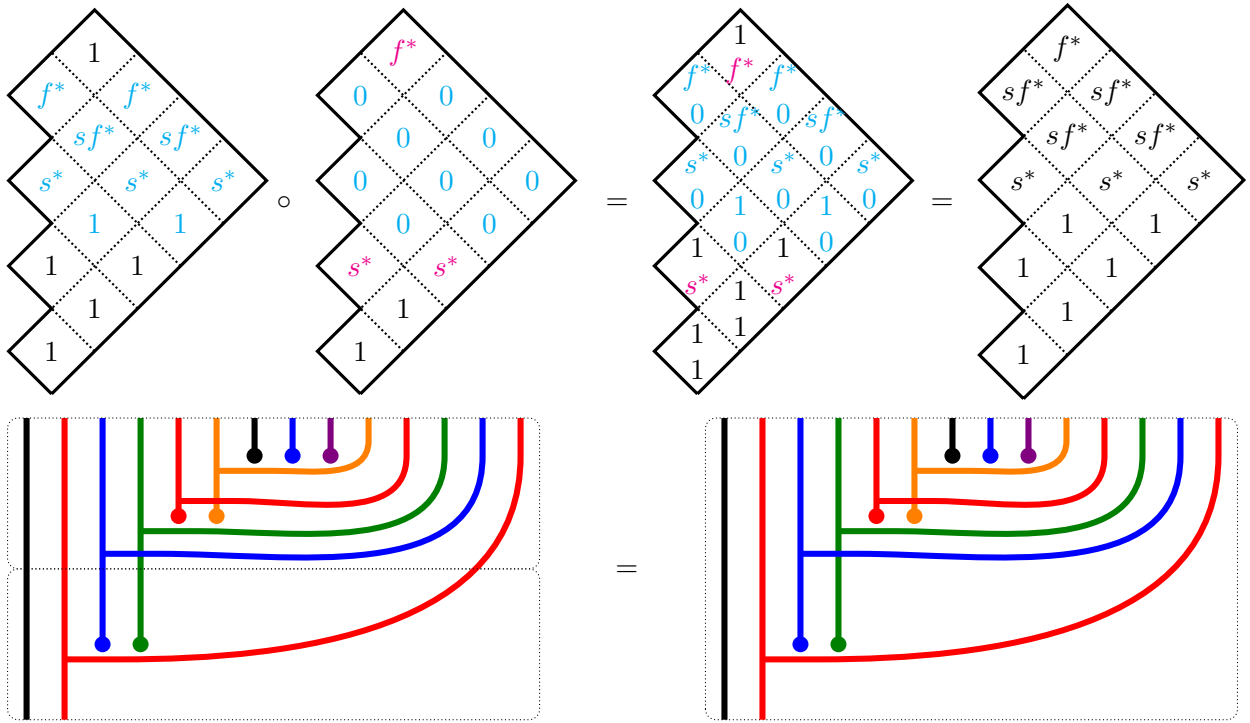

$$
\begin{minipage}{3.55cm}
\end{minipage}
$$
\caption{A product of doubly non-commuting diagrams on tableaux.}
\label{zeroorientation2}
\end{figure}

 \subsection{The presentation of $\mathcal{H}_{(D_n, A_{n-1})}$}
  
At this point we have enough to state the main theorem of the paper, namely that  $\mathcal{H}_{(D_n, A_{n-1}) }$ is a quadratic algebra over $\ZZ$ with an explicit combinatorial presentation given below. The theorem will be proved in \cref{proof}.

 \begin{thm}\label{presentation}
Let $\Bbbk$ be an integral domain containing $i\in\Bbbk$ such that $i^2=-1$. Then the algebra $\mathcal{H}_{(D_n, A_{n-1}) }$ is the  associative $\Bbbk$-algebra generated by the elements 
\begin{equation}\label{geners}
\{D^\la_\mu,
D_\la^\mu \mid 	
\text{$\la, \mu\in \mptn $ with $\la = \mu - {\color{cyan}p}$ for some ${\color{cyan}p}\in \underline{\mu}$} 
	\}\cup\{ {\sf 1}_\mu \mid \mu \in \mptn  \}		
	\end{equation}
	subject to the following relations and their duals. 

	\smallskip\noindent
{\bf The idempotent   
relations:} 
For all $\la,\mu \in \mptn $, we have that 
\begin{equation}\label{rel1}
{\sf 1}_\mu{\sf 1}_\la =\delta_{\la,\mu}{\sf 1}_\la \qquad 
\qquad {\sf 1}_\la D^\la_\mu {\sf 1}_\mu = D^\la_\mu.
\end{equation}

\smallskip\noindent
	{\bf The 
	self-dual relation: } 
	Let  ${\color{cyan}p}\in \underline{\mu}$ and $\la = \mu - {\color{cyan}p}$. If $\cp$ is doubly covered and hence the pair $(\ot, \gr)$ adjacent to $\cp$ is non-commuting with $\gr \prec \ot$, then we have
	\begin{equation}
D_\mu^{\la} D_{\la}^\mu
= (-1)^{b({\color{cyan}p})-1}\Bigg(
2
\!\! \sum_{   \begin{subarray}{c} \mq \in (\underline{\mu} \cap \underline{\la}) \\ \cp \prec\mq,\cp \prec\prec\mq \end{subarray}}
\!\!
(-1) ^{b({\color{magenta}q}) } D^{\la}_{\la- {\color{magenta}q} } D^{\la - {\color{magenta}q}  }_{\la} + 
2  (-1)^{b(\ot)} D^{\la}_{\la- \ot } D^{\la -\ot }_{\la} + 
 (-1)^{b(\gr)} D^{\la}_{\la- \gr } D^{\la - \gr }_{\la}\Bigg)
 \label{selfdualrelD}
\end{equation}
and otherwise we have 
\begin{equation}
D_\mu^{\la} D_{\la}^\mu
= (-1)^{b({\color{cyan}p})-1}\Bigg(
2
\!\! \sum_{   \begin{subarray}{c} \mq \in (\underline{\mu} \cap \underline{\la}) \\ \cp \prec\mq,\cp \prec\prec\mq \end{subarray}}
\!\!
(-1) ^{b({\color{magenta}q}) } D^{\la}_{\la- {\color{magenta}q} } D^{\la - {\color{magenta}q}  }_{\la} + 
\!\!\sum_{  \begin{subarray}{c} \ot \in \underline{\la} \\ \ot \, \text{adj.}\, {\color{cyan}p} \end{subarray}}
\!\!
 (-1)^{b(\ot)} D^{\la}_{\la- \ot} D^{\la - \ot }_{\la}\Bigg), 
 \label{selfdualrelA}
\end{equation}  
where throughout we refer to the set $\underline{\mu} \cap \underline{\la}$ to be the cups in $\underline{\mu}$ that commute with $\cp$ (and hence are in $\underline{\la}$ also) and abbreviate ``adjacent to" simply as ``adj."

\smallskip\noindent
{\bf The commuting relations:} 
Let ${\color{cyan}p},{\color{magenta}q}\in \underline{\mu}$ which commute. Then we have 
\begin{equation}\label{commuting}
D^{\mu -{\color{cyan}p}-{\color{magenta}q}}_{\mu-{\color{cyan}p}}D^{\mu-{\color{cyan}p}}_\mu = D^{\mu-{\color{cyan}p}-{\color{magenta}q}}_{\mu -{\color{magenta}q}}D^{\mu -{\color{magenta}q}}_\mu \qquad
D^{\mu -{\color{cyan}p}}_\mu D^\mu_{\mu - {\color{magenta}q}} = D^{\mu - {\color{cyan}p}}_{\mu - {\color{cyan}p} - {\color{magenta}q}}D^{\mu - {\color{cyan}p} - {\color{magenta}q}}_{\mu - {\color{magenta}q}}.
\end{equation}

\smallskip\noindent
{\bf The non-commuting relation:}   
Let $\cp,\mq \in \underline{\mu}$ with $\mq \prec \cp $ be non-commuting cups. 
Then $\mq$ is adjacent to a pair of commuting non-concentric cups, which 
we label by  ${\color{darkgreen}q}^1$ and ${\color{orange}q}^2$ from left-to-right in $\underline{\mu - \mq}$ (with ${\color{darkgreen}q}^1$ possibly decorated). 
Then we have:
\begin{equation}\label{noncommutingcup1}
D^{\mu -\cp}_\mu D^\mu_{\mu-\mq} = 
 D^{\mu - \cp}_{\mu - \cp - {\color{darkgreen}q}^1}D^{\mu - \cp - {\color{darkgreen}q}^1}_{\mu - \mq} =
 D^{\mu - \cp}_{\mu - \cp - {\color{orange}q}^2}D^{\mu - \cp - {\color{orange}q}^2}_{\mu - \mq}.
\end{equation}

\smallskip\noindent
{\bf The doubly non-commuting relation:}  
Let $\cp,\mq \in \underline{\mu}$ with $\mq \prec \prec \cp $ be doubly non-commuting. 
Then $\mq$ is adjacent to a pair of concentric non-commuting cups, 
 ${\color{darkgreen}q}^1$ and ${\color{orange}q}^2$ with ${\color{darkgreen}q}^1  \prec {\color{orange}q}^2$ in $\underline{\mu - \mq}$ (with ${\color{orange}q}^2$ possibly decorated).
Then we have: 
\begin{equation}\label{noncommutingcup2}
D^{\mu -\cp}_\mu D^\mu_{\mu-\mq} = D^{\mu -\cp }_{\mu - \cp - {\color{orange}q}^2}D^{\mu -\cp- {\color{orange}q}^2}_{\mu -\mq}   
\end{equation}

\smallskip\noindent
{\bf The adjacency relation}  
Given $\la=\mu-\cp$,  suppose that  ${\color{cyan}p}\in \underline{\mu}$ and $\ot\in \underline{\la}$  are  adjacent.
Then we have:
\begin{equation}\label{adjacentcup}
D^{\mu - {\color{cyan}p} - \ot}_{\mu - {\color{cyan}p}}D^{\mu - {\color{cyan}p}}_\mu = 
\left\{ 
\begin{array}{ll} (-1)^{b(\langle {\color{cyan}p}\cup \ot\rangle_\mu) - b(\ot)} 
D^{\mu - {\color{cyan}p} - \ot}_{\mu - \langle {\color{cyan}p}\cup \ot\rangle_\mu}
D^{\mu - \langle {\color{cyan}p}\cup \ot\rangle_\mu}_\mu &
 \mbox{if $\langle {\color{cyan}p}\cup \ot\rangle_\mu$ exists}
  \\ 0 
& \mbox{otherwise} \end{array} \right.
\end{equation}

  \end{thm}

  \begin{rmk}
  \label{small example list}
We now spend a moment drawing the cup diagrams related to some of the relations to gain some intuition that will be useful during the proof later on.
Beginning with non-commuting cups $\mq,\cp\in \underline{\mu}$ with $\mq\prec\cp$, if we picture the cups on a minimal number of vertices, we have that $\underline{\mu}$ is equal to;
  $$   %
$$
where $\cp$ is undecorated or decorated respectively.
Then the cup diagram \underline{$\mu - \mq$} is given by
$$\left.   %
\right._{\displaystyle .}$$
Notice by flipping both the cups ${\color{orange}q}^1,{\color{darkgreen}q}^2 \in \underline{\mu-\mq}$ we obtain the cup diagram \underline{$\mu-\cp$} in both cases:
	$$   %
$$
	As discussed in \cref{general}, there may be (many) undecorated cups added to the diagrams below that have no effect on the relations. For example, if we draw the same non-commuting cups $\cp,\mq$ with $\mq\prec\cp$ adding (randomly chosen) additional undecorated cups that do not cover $\mq$, we have that $\underline{\mu}$ is equal to:
	  $$      %
$$
where $\cp$ is undecorated or decorated respectively.
Then the cup diagram \underline{$\mu - \mq$} is given by:
 $$      %
$$
Finally, if we again flip the cups ${\color{orange}q}^1,{\color{darkgreen}q}^2 \in \underline{\mu-\mq}$ we obtain the cup diagram \underline{$\mu-\cp$} in both cases:
 $$      %
$$	
	We will go into more detail on these additional undecorated cups and how to account for them in the next section.
	
	Next, we will detail similarly the doubly non-commuting cups $\mq,\cp\in \underline{\mu}$, where $\mq\prec\prec\cp$. Then pictured on a minimal number of vertices $\underline{\mu}$ is as follows:
	  $$   %
$$
in which $\mq$ is undecorated or decorated, respectively.
Then the cup diagram \underline{$\mu - \mq$} is given by:
$$ \left.  %
\right._{\displaystyle .}$$
	Note that these are precisely the non-commuting cup diagrams from above (this is partly the reason for calling these un-nested cups doubly non-commuting).
Flipping the cup ${\color{orange}q}^1,{\color{darkgreen}q}^2\in \underline{\mu-\mq}$, as above, in both cases we obtain the cup diagram \underline{$\mu-\cp$}:
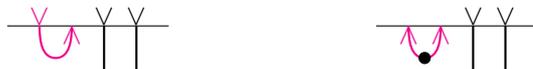
	$$   %
.$$
	
  \end{rmk}

 \section{Contraction }
 \label{Contraction Sect}

Before proving \cref{presentation} we will first make a slight detour to construct contraction maps. These will allow us to compare weights, cups and Hecke categories of different sizes. This will allow us to limit our sights to \lq \lq small cases", where the cup diagrams have the minimal number of vertices, like in \cref{small example list}, during the proof of \cref{presentation} by providing an easy way to generalise these smaller cases to all of $\mathcal{H}_{(D_n, A_{n-1})}$.

\begin{defn}
Define $ \mathscr{P}^k_{n}$ to be the set of  tile-partitions in $\mptn$ that have a removable tile of content $k$. 
We call such tile-partitions {\sf contractible at $k$}.
 \end{defn}

\begin{rmk}
In terms of cup diagrams, we have that
 $\la\in\mathscr{P}^k_{n}$ for $ k\neq 0$ if and only if there exists $\cp \in \underline{\mu}$ with $(l_\cp,r_\cp)=(k-\frac{1}{2},k+\frac{1}{2})$.
Similarly, $\la\in \mathscr{P}^0_{n}$ if and only if there exists a decorated $\mq \in \underline{\la}$ with $(l_\mq,r_\mq)=(\frac{1}{2},\frac{3}{2})$.

\end{rmk}

\begin{defn}
For  $0 \leq k \leq n$ we define the {\sf contraction map} $\Phi_k:  \mathscr{P}^k_{n} 
 \longrightarrow
  \mathscr{P}_{n-2}$ on weights by setting $\Phi_k (\la)$ for $\la \in  \mathscr{P}^k_{n}$ to be the weight obtained from $\la$ by removing these vertices.
  \end{defn}
 
\begin{rmk}

Strictly speaking, we $(i)$ fix the vertices at $x=\frac{1}{2}, \dots, k-\frac{3}{2}$, $(ii)$ then delete the vertices at $x=k-\frac{1}{2},k+\frac{1}{2}$, $(iii)$ then shift the vertices at $x=k+\frac{3}{2},\dots,n-\frac{1}{2}$ to $x-2$ and $(iv)$ finally  orientate the first node (at $x=\frac{1}{2}$) as required to maintain an even number of $\up$ vertices.
See \cref{dilation on cups} below for some examples of this.
\end{rmk}

    \begin{figure}[ht!]
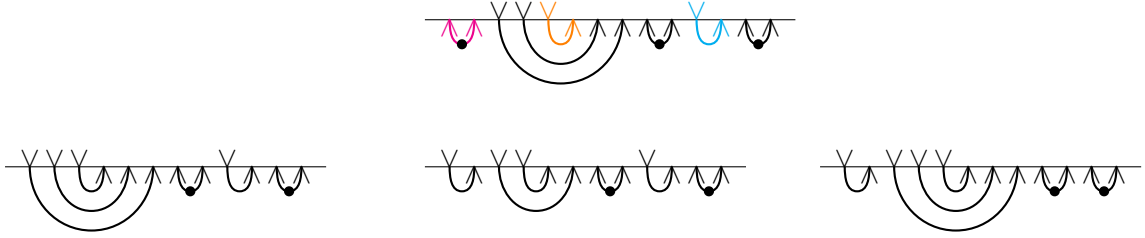

     $$   
$$ 
    
    \caption{$\la=(1,2,3,4,5,6,7,8,9,8,8,8,3)$ is contractible at $k=0,15,11$. We depict below the contraction of $\la$ at $k={\color{magenta}0},{\color{orange}5}$ and ${\color{cyan}11}$ respectively; these are all the $k$ for which $\la$ is contractible.}
    \label{dilation on cups}
    \end{figure} 
 
The following lemmas follow directly from the definitions given above.

\begin{lem}
The map $\Phi_k$ is bijective.
\end{lem}

\begin{lem} 
\label{degree}
 Let $\la, \mu \in  \mathscr{P}^k_{n}$.
We have that $\DP$ is oriented 
if and only if $\underline{\Phi_k(\mu)}\Phi_k(\la)$ is oriented.
If $\la= \mu - \sum_{i=1}^{j}\mq^i$ and for some $1\leq i \leq l$, we have that $(l_{\mq^i},r_{\mq^i})= (k-\frac{1}{2},k+\frac{1}{2})$, then $\DP$ is of degree j and $\underline{\Phi_k(\mu)}\Phi_k(\la)$ is of degree $j-1$. If no such $\mq^i$ exists then both $\DP$ and $\underline{\Phi_k(\mu)}\Phi_k(\la)$ are both of degree j. 
\end{lem}

\begin{lem}
\label{breadth}
If $\la = \mu - \cp$ for some $\cp\in \underline{\mu}$ with $(l_{\cp},r_{\cp})\neq (k-\frac{1}{2},k+\frac{1}{2})$  then we have $\Phi_k(\la) = \Phi_k(\mu) - \cp'$ where $\cp'\in \underline{\Phi_k(\mu)}$   satisfies 
\begin{center}
\begin{tabular}{ll}
 $b(\cp')=b(\cp)-2$ & if $k < l_{\cp}$ and $\cp$ is decorated, \\
$b(\cp')=b(\cp)-1$ & if $l_{\cp}<k <r_{\cp}$,\\
 $b(\cp')=b(\cp)$ & if $k> r_{\cp}$ or $k < l_{\cp}$ and $\cp$ is undecorated.
\end{tabular}
\end{center}
Furthermore, the decorations of 
  $\cp \in  \mu $  and 
  $\cp'\in \underline{\Phi_k(\mu)}$  
are identical unless      $l_{\cp'} = \frac{1}{2}$ and $k\neq0$. 

In all cases we 
write  
$\Phi_k(\cp):=\cp'$.
\end{lem}

 We now extend the contraction map $\Phi_k$ to homomorphisms for the Hecke categories. We will abuse notation and use the same notation for both contraction maps. 

\begin{thm}\label{dilation}
Let $\Bbbk$ be a commutative integral domain and let $i \in \Bbbk$ be a square root of $-1$ and set $1_k =  \sum_{\mu \in  \mathscr{P}^k_{n}} 1_{\stt_\mu}$.
We define the maps $$\Phi_k : 1_k\mathcal{H}_{(D_n, A_{n-1})}1_k \to \mathcal{H}_{ (D_{n-2}, A_{n-3}) }$$on the generators as follows.
Suppose $\la, \mu \in \mathscr{P}^k_{n} $, with $\la = \mu - \cp$, for some $\cp \in \underline{\mu}$ with $(l_{\cp},r_{\cp})\neq (k-\frac{1}{2},k+\frac{1}{2})$ , we define $\Phi_k({\sf 1}_\mu)= {\sf 1}_{{\Phi_k(\mu)}}$
and 
$$
\Phi_k(D^\la_\mu)= \left\{ \begin{array}{ll} -D^{\Phi_k(\la)}_{\Phi_k(\mu)} & \mbox{ if $\cp$ is decorated, $k=0,1$ and $k<l_{\cp}$} \\
i \cdot 
D^{\Phi_k(\la)}_{\Phi_k(\mu)}	& \mbox{ if $l_{\cp}<k <r_{\cp}$ and $\CP\cap[\mu]$ forms a local cap with a tile of content $k$   
}\\	
(-i) \cdot 
D^{\Phi_k(\la)}_{\Phi_k(\mu)}	& \mbox{ if $l_{\cp}<k < r_{\cp}$ and $\CP\cap[\mu]$ forms a local cup with a tile of content $k$}\\
D^{\Phi_k(\la)}_{\Phi_k(\mu)} & \mbox{ else.}
 \end{array} \right.
$$Furthermore, we define $\Phi_k(D_\la^\mu) = \Phi_k((D_\mu^\la)^*) = (\Phi_k(D_\mu^\la))^*$. Then $\Phi_k$ extends to an isomorphism of graded $\Bbbk$-algebras. 
\end{thm}

\begin{proof}
The map $\Phi_k$ is the inverse of the dilation map $\varphi_k$, which is defined on the monoidal (spot, fork, braid and idempotent) generators of $\mathcal{H}_{(D_n,A_{n-1})}$  in \cite[Section 5.3]{ChrisHSP} (note that the notation $\ctau = s_{k}$ is used). This dilation map $\varphi_k$ is proven to be an isomorphism of graded $\Bbbk$-algebras onto $1_k \mathcal{H}_{(D_{n+2}, A_{n+1})}1_k$.
In particular if $D$ is a light leaves basis element then for $k\neq0,1$, $\varphi_k(D)$ is obtained (up to sign) by changing each monoidal generator with colour $s_k$ to become a horizontal concatenation of three different coloured monoidal generator as defined in \cite[Section 5.3]{ChrisHSP} and the collection of a sign.
For $k=0,1$, this concatenation (up to sign) is instead made up of five monoidal generators (of four different colours. See in \cite[Figure 55]{ChrisHSP} for a precise example of what a fork and spot generator map to). 
The aforementioned signs are a function of the length of this concatenation and hence dilating at $k=0,1$ carries with it a different factor of $i$ than for $k\neq 0,1$.
Rewriting the generators $D^\la_\mu$  and $D^{\varphi_k(\la)}_{\varphi_k(\mu)}$ on labelled oriented tableaux as in \cref{undec0} the result follows.
See \cref{undec0} for an example of a $k=1$ dilation case pictured on oriented tableaux and \cref{undec2} for the corresponding light leaves basis element. All the other cases can be pictured in very similar manner.
 \end{proof}
  
  \begin{figure}[ht!]
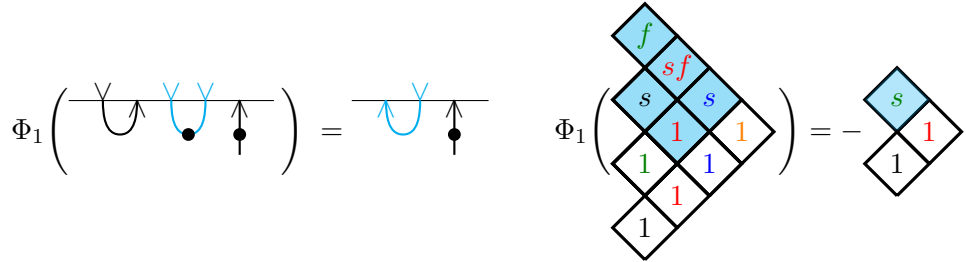

 $$
\Phi_1\Biggl( 
  \begin{minipage}{2.8cm}
 
\end{minipage}
$$
 \caption{
 On the right we have applied $\Phi_1$ to the oriented tableaux $D^\la_\mu$ where $\mu=(1,2)$ and $\la=(1,1)$ and have pictured the corresponding cup diagram on the left. 
 }
  \label{undec0}
\end{figure}

  \begin{figure}[ht!]
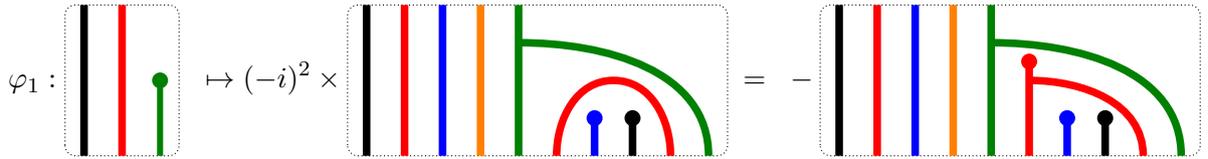

   $$ \varphi_1: \begin{minipage}{1.75cm}
%
\end{minipage}$$
 \caption{Here we show the map $\varphi_1$, the inverse of $\Phi_1$, applied to the light leaves basis element corresponding to the oriented tableau pictured in \cref{undec0}.
 }
  \label{undec2}
\end{figure}
 
\begin{defn}
Let $\la, \mu \in \mptn$ with $\la = \mu-\cp$ for some $\cp \in \underline{\mu}$. We say that $\la\underline{\mu}$ is \textsf{incontractible} if there does not exist   $k\in \ZZ_{\geq0}$ such that $\la, \mu \in\mathscr{P}^k_{n}$.
\end{defn}

\begin{rmk}
It follows from the definition that $\la\underline{\mu}$ being incontractible is equivalent to $\CP=\color{cyan}[r,c]$ being the unique tile in  ${\rm Rem}(\mu)$. Hence, we must have $b(\cp) = 1$.
\end{rmk}

 \section{The presentation and relations for    $\mathcal{H}_{(D_n, A_{n-1}) }$}
 \label{proof}

In this section, we prove the relations stated in \cref{presentation} hold and that they are sufficient to provide a quadratic presentation of $\mathcal{H}_{(D_n, A_{n-1}) }$.
First, we will prove some preparatory lemmas to make the proof slightly quicker. 
The first lemma is a few oriented tableaux relations that follow immediately from the corresponding Soergel diagrams, which will be helpful in calculating multiplication of oriented tableaux in the manner of \cref{morepics}.

\begin{lem}
\label{some pics}
The relations pictured in \cref{spotpic6} hold:
\end{lem}

\begin{figure}[ht!]
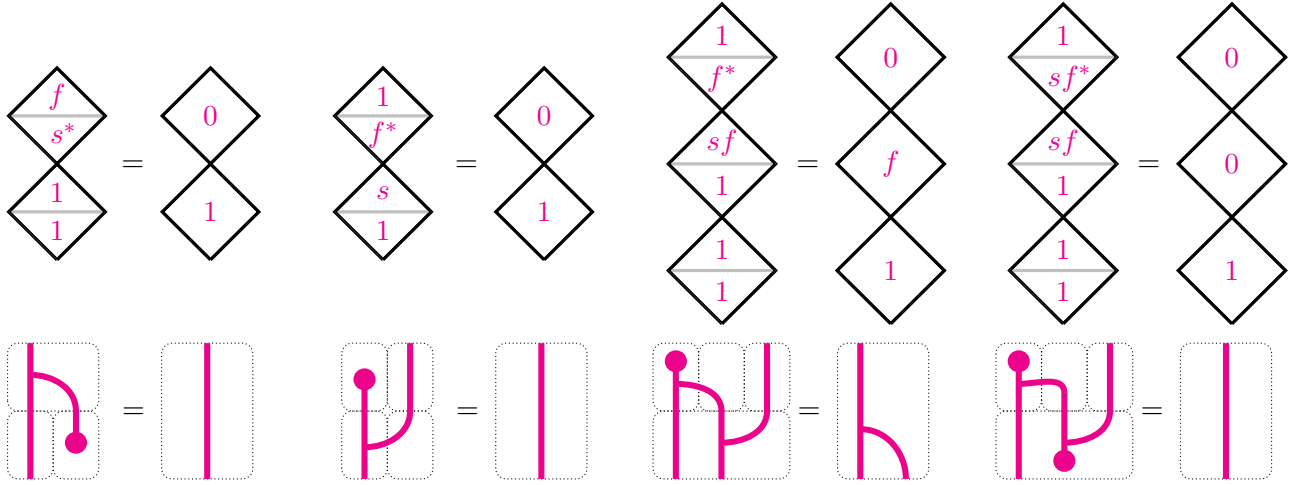

$$
 \begin{minipage}{1.4cm}

    \end{minipage}
    $$
  \caption{A few useful fork-spot relations on oriented tableaux. These relations also hold with tiles corresponding to Soergel diagrams that contain blank space in between, in the sense of \cref{blank space baby}.}
  \label{spotpic6}
  \end{figure}

We will next give a quick definition that we will utilise to prove some further useful lemmas.

\begin{defn}
Given $\la \in \mptn $ with $[x,y] \in \la$ such that $\stt_\la([x,y])=k$ and  ${\sf ct}([x,y])=\ctau$, set 
$$
  {\sf gap}(\stt_\la-[x,y]) = {\sf 1}_{\stt_\la{\downarrow}_{\{1,\dots,k-1\}}} \otimes {\sf spot}^{\ctau}_\emptyset
   {\sf spot}_{\ctau}^\emptyset
   \otimes {\sf 1}_{\stt_\la{\downarrow}_{\{k+1,\dots,\ell(\la)\}}}.  
  $$
\end{defn} 

\begin{rmk}
The oriented tableau for ${\sf gap}(\stt_\la-[x,y])$ is created by placing a {\sf gap-labelled tile} at $\color{cyan}[x,y]$ on the oriented tableau $\stt_\la^{\la}$.
See \cref{gapper} for an example
\end{rmk}

\begin{figure}[ht!]
 \begin{minipage}{3cm}
  \end{minipage}
\caption{The oriented tableau for ${\sf gap}(\stt_\la-{\color{cyan}[4,2]})$, where $\la=(1,2,3,4,2)$.}
\label{gapper}
\end{figure}

The next two lemmas are more oriented tableaux relations that follow from relations detailed in \cref{Hecke}.

\begin{lem}\label{nullbraidontiles}
The following relations hold:
$$ (i)\quad
\begin{minipage}{2.55cm}
%
  \end{minipage}
 $$

\end{lem}

\begin{proof}
%
%
The first and second relations follow by applying the null-braid relation followed by fork-spot contraction as follows: 
$$
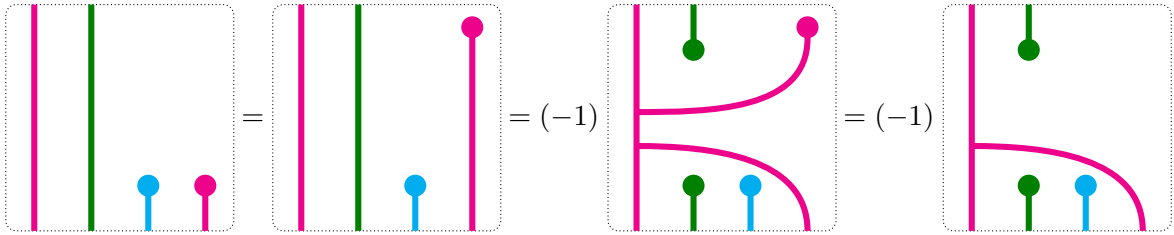
\begin{minipage}{3cm}%
\end{minipage} 
$$  
The third relation is also obtained by simply applying the null-braid relation as follows: 
 $$  
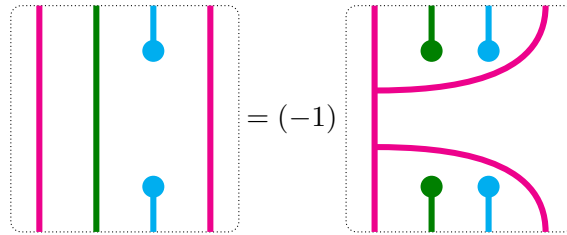
\begin{minipage}{3cm}%
\end{minipage}  .
$$
The final relation is a corollary of the third:
 $$  
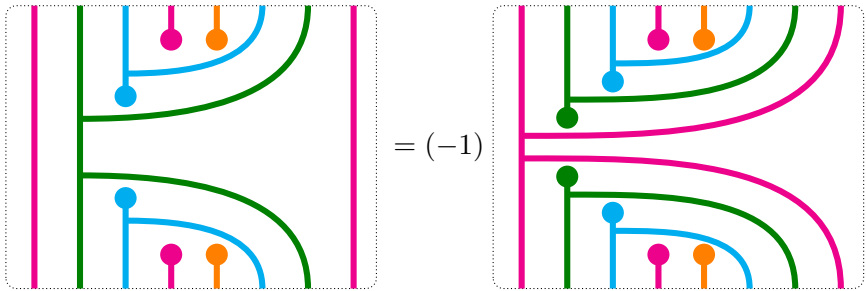
\begin{minipage}{5cm}%
\end{minipage}  
$$

 \end{proof}

\begin{lem}\label{adjacencywayback}
The following relations hold at the left-hand wall of an oriented tableau.
$$ (i)\quad
\begin{minipage}{2.55cm}
%
  \end{minipage}
 $$
\end{lem}

\begin{proof}
The first equation follows by applying the null-braid relation followed by the fork-spot contraction as follows: 
$$
\begin{minipage}{3.5cm}%
\end{minipage}. 
$$  
The second identity also follows by an identical application of the null-braid relation as in \cref{nullbraidontiles}:
$$
\begin{minipage}{4cm}%
\end{minipage} 
$$  
Finally, the first equality of the last relation is equivalent to \cite[Proposition 4.8]{ChrisHSP} pictured on an oriented tableau. The second equality follows as a corollary of the first, by noting that a gap factors through zero in the middle of a Soergel diagram.
\end{proof}

Next, we will utilise \cref{nullbraidontiles,adjacencywayback} to adapt a result from \cite{ChrisHSP}.
\begin{lem}\label{gap}
Let  $\la \in \mptn $ with $[x,y] \in \la$. Then 
$${\sf gap}(\stt_\la- [x,y]) = 
\begin{cases}
(-1) ^{b \langle x,y  \rangle_{\la} } D^{\la}_{\la- \langle x,y \rangle_{\la} } D^{\la - \langle x,y \rangle_{\la}  }_{\la}		&\text{if $\langle x,y \rangle_{\la}$ is well-defined}\\
0						&\text{otherwise.}
\end{cases}$$
\end{lem}

\begin{proof}
Suppose $S$ is the strand connecting the two northern edges of $[x,y]$, then if $\langle x,y \rangle_{\la}$ is not well-defined, either $S$ is a propagating ray or $S\cap[x,y]$ is a local cup-cap. If $S\cap[x,y]$ is a local cup-cap, the following two points hold:
\begin{enumerate}
\item $[x,y] \notin\textsf{Rem}(\lambda$) and
\item $x-y=i$ for some $i$ even.
\end{enumerate}
Point $(1)$ follows from the fact that the $S$ must also intersect the box $[x,y+1]\in \la$ (note that $[x,y+1]\cap S$ is a cap). Point (2) is simply a restatement of \cref{even}.
 Applying \cref{nullbraidontiles} (iii) $\frac{i-2}{2}$ times to the oriented tableaux corresponding to ${\sf gap}(\stt_\la- [x,y])$ results in the oriented tableaux corresponding to ${\sf gap}(\stt_\la- [x-\frac{i-2}{2},2])$; then applying \cref{adjacencywayback} (iii) to this oriented tableau gives the required result in the case that $\langle x,y \rangle_{\la}$ is not well-defined.

Now assume that $\langle x,y \rangle_{\la}$ is a well-defined cup. Apply \cref{nullbraidontiles} (iii) to the oriented tableau of ${\sf gap}(\stt_\la- [x,y])$ $k$ times, where $k$ is the number of local cups in $\langle x,y \rangle_{\la}\cap \la$. Then there are only two cases: $(i)$ there exists two gap tiles at the top of the oriented tableau, one to the east and one to the west of $[x,y]$ and this is precisely the oriented tableau of $D^{\la}_{\la- \langle x,y \rangle_{\la} } D^{\la - \langle x,y \rangle_{\la}  }_{\la}$ or 
$(ii)$ there exists a gap labelled tile $[x',y']$, with $\textsf{ct}([x',y'])=s_0,\text{ or }{\color{darkgreen}s_1}$ at the left-hand wall. 
In case $(ii)$ apply \cref{adjacencywayback} (i) once and \cref{nullbraidontiles} (iv) $b(\cp)-k-1$ times to the tableau, and hence obtain the oriented tableau of $D^{\la}_{\la- \langle x,y \rangle_{\la} } D^{\la - \langle x,y \rangle_{\la}  }_{\la}$.
\end{proof}    
    
\begin{cor}\cite[Proposition 4.18]{ChrisHSP} \label{BBgaprel}    
Let  $\la \in \mptn $,  ${\color{cyan}[r,c]} \in {\rm Add}(\la)$ with $\textsf{ct}([r,c])=\ctau$. Then we have that
\begin{align}\label{notChook2} 
{\sf 1}_{\la}  \otimes {\sf bar}(\ctau  )  
  = 
\sum_{  \begin{subarray}{c} [r,y] \in \la  \\ [r-1,y] \in \la \\ y \leq c \end{subarray}}
(-1) ^{b \langle x,y  \rangle_{\la} } 
D^{\la}_{\la- \langle x,y \rangle_{\la} } D^{\la - \langle x,y \rangle_{\la}  }_{\la},
  \end{align}
  where the sum is taken over the $[x,y] \in \la$ with $[x,y]=[r,y]$ with $y<c$ or  $[x,y]=[r-1,y]$ with $y\leq c$ and $\langle x,y \rangle_\la$ well-defined.
\end{cor}
 
\begin{proof}
In \cite[Proposition 4.18/Lemma 4.20]{ChrisHSP} we have that
 \begin{align}\label{notChook} 
{\sf 1}_{  \la}  \otimes {\sf bar}(\ctau  )  
  = 
-\!\!
\sum_{[x,y]}
  {\sf gap}(\stt_\la- [x,y])  
  \end{align}
    where the sum goes over the $[x,y] \in \la$ with $[x,y]=[r,y]$ with $y<c$ or  $[x,y]=[r-1,y]$ with $y\leq c$.
    For each gap term in the summation, apply \cref{gap} and the result then follows.
\end{proof}

 \subsection{Proof of \cref{presentation}}

By \cref{generatorsarewhatweasay} it suffices to check that  (\ref{rel1}) to (\ref{adjacentcup}) form a complete list of relations. 
We start by proving that all the relations hold before showing, at the end, that (\ref{rel1}) to (\ref{adjacentcup}) form a complete list of relations.
The idempotent relations are immediate.
We now check the other relations.

\smallskip
\noindent\textbf{Proof of the self-dual relation.} 
First, the left-hand sides of both \eqref{selfdualrelD} and \eqref{selfdualrelA} are equivalent, if $b(\cp) = 1$, to the case considered in \cref{BBgaprel}. 
Therefore, if $b(\cp) = 1$, to calculate the right-hand sides of \eqref{selfdualrelD} and \eqref{selfdualrelA} we need only consider which terms on the right-hand sides of \eqref{notChook2} correspond to (well-defined and hence non-zero) cups, $\langle x,y\rangle_\la$.
 The following claim will be useful in this regard; indeed, it will account for all terms which appear with multiplicity two on the right-hand side of \eqref{selfdualrelD} and \eqref{selfdualrelA}.

\noindent{\bf Claim: } For $1\leq  y\leq c-2$ we claim that 
$$ \langle r-1, y+1\rangle_\la = \langle r , y\rangle_\la.$$
To verify this claim, we first note that the tile $[r,y+1] \in \la$ by assumption and lies north-east of $[r-1,y+1]\in \la$ and north-west of $[r,y]\in\la$; therefore the local cups  within the tiles
$[r-1,y+1]$ and  $[r,y]$ belong to the same strand, $S_y$. 
 Let's consider the strand $S_y$ for $1\leq y \leq 
 c-2$.
  If $S_y \cap [r-1, y+1]$ is a local cup-cap then $S_y \cap [r, y]$ will also be a local cup-cap; in which case $ \langle r-1, y+1\rangle_\la$ and $\langle r , y\rangle_\la$ are undefined. 
   If $S_y \cap [r-1, y+1]$ is a local cup then $S_y \cap [r, y]$ will also be a local cup and hence using \cref{genuine}, we have that $ \langle r-1, y+1\rangle_\la = \langle r , y\rangle_\la=S_y:=\mq \in \underline{\mu}$. This proves the claim above.

We next group together the terms in  \eqref{selfdualrelD} and \eqref{selfdualrelA} according to 
  two cases:
 $(i)$  
$\{\mq\in \underline{\mu}\cap\underline{\la}\}$ such that either $\cp\prec \mq$ or $\cp\prec\prec \mq$ 
 and $(ii)$ $\{\ot\in \underline{\la}\}$ such that $\ot\in \underline{\la}$ is adjacent to $\cp$.
  We begin by rewriting the adjacent terms in case $(ii)$ in terms of the tile combinatorics of \cref{notChook2}.

  In \eqref{selfdualrelA}, the adjacent terms from case $(ii)$ are labelled by $\ot_1,\ot_2$ a commuting pair of cups 
which are both adjacent to $\cp$.      In this case, $\ot_1 = \langle r-1,c\rangle _\la$ 
    and 
    $\ot_2= \langle r ,c-1\rangle _\la$ (without loss of generality). 
 In \eqref{selfdualrelD},  $\cp$ is doubly covered and the terms from case $(ii)$ are labelled by 
   $\gr\prec \ot$, which are both adjacent to $\cp$. 
    In this case $\gr=  \langle r ,c-1\rangle _\la$ and 
    $ \ot= S_{c-2}=\langle r,c-2\rangle _\la = \langle r-1,c-1\rangle _\la$. 
    Therefore the claim implies that $S_{c-2}$ appears with multiplicity $2$ on the right-hand side of \cref{selfdualrelD} as required.
    
Finally, we consider case $(i)$.   
 The $\mq$ in $(i)$ are of the form $\mq= S_y$ for some $1\leq y \leq c-2$ 
 (with a strict inequality for $\cp$ doubly covered). By the claim this also implies that $\mq=S_y$ appears twice on the right-hand side of \eqref{selfdualrelD} and \eqref{selfdualrelA} as required. By collecting the cups from cases $(i)$ and $(ii)$ together, it is clear that \cref{selfdualrelD,selfdualrelA} hold for $b(\cp)=1$.

We now consider  the case that $b(\cp)>1$. We will use \cref{dilation} to reduce such cases down to cups of breadth 1.
By assumption $b(\cp)\geq 2$, so there exists some $k\in \mathbb{Z}_{\geq 0}$ such that $(\Phi_k(\la), \Phi_k(\mu))\vcentcolon=(\la', \mu')$,
 where $\la' = \mu'-\cp'$ and $b(\cp')< b(\cp)$. Using \cref{breadth}, this means that either
 $(1)$ $l_{\cp}<k<r_{\cp}$
 or 
 $(2)$
   $k<l_{\cp}$ and   $\cp$ is decorated. 
   We will check $(1)$ and $(2)$ first for \eqref{selfdualrelD} and then for \eqref{selfdualrelA}.
  
We begin by checking \eqref{selfdualrelD}. We note this is the case where $\cp$ is doubly covered as in \cref{adj lem2} c) and d).  
Suppose that $\vu\in \underline{\mu}$ is the unique cup  such that $\cp$ and $\vu$  do not commute 
and  $\cp \prec \prec \vu$
 (and hence there does not exist $\hat{\mq}$ such that $\cp  \prec \hat{\mq}$).
 By \cref{adj lem2} we have that $\cp$ is adjacent to a non-commuting pair $\gr \prec \ot$.  
We first consider  case $(2)$, that is    $k<l_{\cp}$.
In this case $\cp$ must be decorated and $\vu,\gr,\ot$ are as pictured in \cref{adj lem2} d). Applying \cref{dilation} to both sides of \eqref{selfdualrelD} we obtain:
\begin{eqnarray*}
D_{\mu'}^{\la'} D_{\la'}^{\mu'}
&=& (-1)^{b(\cp)}\Bigg(
2
    \sum_{    \begin{subarray}{c} \mq \in (\underline{\mu} \cap \underline{\la}) \\\cp \prec\prec\mq \end{subarray}}
(-1) ^{b(\mq)} D^{\Phi_k(\la)}_{\Phi_k(\la- \mq) } D^{\Phi_k(\la - \mq)  }_{\Phi_k(\la)}    \\
&&\qquad\qquad\qquad + 
2  (-1)^{b(\ot)} D^
{\Phi_k(\la)}_{\Phi_k(\la- \ot) } D^{\Phi_k(\la -\ot )}_{\Phi_k(\la)} + 
 (-1)^{b(\gr)} D^{\Phi_k(\la)}_{\Phi_k(\la- \gr) } D^{\Phi_k(\la - \gr )}_{\Phi_k(\la)}\Bigg).
\end{eqnarray*}
For any $\cp\neq\rs\in \underline{\la}$,  we have that $\Phi_k(\la-\rs) = \la' - \rs'$, by \cref{breadth}.
 Moreover, for $k< l_{\cp}$, $\Phi_k$ also gives rise to the bijections
\begin{equation}\label{akjhsdfhjkasdfhjkladsfhljkasf}
\begin{split}
\{\mq' \in\underline{\mu'}\cap\underline{\la'}\mid \cp'\prec \prec\mq'\} &\leftrightarrow
  \{\mq\in\underline{\mu}\cap\underline{\la} \mid \cp\prec \prec\mq\}\\
   \{\ot',\gr'\in\underline{\la}   \text{ adjacent to } \cp'\} &\leftrightarrow \{\ot,\gr\in\underline{\la}    \text{ adjacent to } \cp\}.\end{split}\end{equation}
Since $b(\cp) = b(\cp')+2$, $b(\mq) = b(\mq')+2$ if 
$\cp'\prec \prec \mq'$, $b(\ot) = b(\ot')+2$
 (as $\ot$ is decorated) 
 and $b(\gr) = b(\gr')$ (as $\gr$ is not decorated),
 the relation holds by \cref{dilation} and induction.

Now consider case $(1)$, that is $l_{\cp}<k<r_{\cp}$.
We can again apply \cref{dilation} to both sides of  \eqref{selfdualrelD} in order to obtain 
\begin{eqnarray*}
-D_{\mu'}^{\la'} D_{\la'}^{\mu'}
&=& (-1)^{b(\cp)}
 \Bigg(2
\sum_{    \begin{subarray}{c} \mq \in (\underline{\mu} \cap \underline{\la}) \\\cp \prec\prec\mq \end{subarray}}
(-1) ^{b(\mq)} D^{\Phi_k(\la)}_{\Phi_k(\la- \mq) } D^{\Phi_k(\la - \mq)  }_{\Phi_k(\la)}   \\
&& \qquad\qquad\quad\quad+ 
2  (-1)^{b(\ot)+1} D^
{\Phi_k(\la)}_{\Phi_k(\la- \ot) } D^{\Phi_k(\la -\ot )}_{\Phi_k(\la)} + 
 (-1)^{b(\gr)} D^{\Phi_k(\la)}_{\Phi_k(\la- \gr) } D^{\Phi_k(\la - \gr )}_{\Phi_k(\la)}\Bigg),	 \qquad \quad
\end{eqnarray*}
where the additional power of $(-1)$ in the second term on the right-hand side comes from the fact that 
$l_\ot<k<r_\ot$. For $l_{\cp}<k<r_{\cp}$, the same bijections as \eqref{akjhsdfhjkasdfhjkladsfhljkasf} hold; additionally noticing that $b(\cp) = b(\cp')+1$, $b(\mq) = b(\mq')+2$, $b(\ot) = b(\ot')+1$ and $b(\gr')=b(\gr)$,
we see that \eqref{selfdualrelD} holds by induction.

We now check \eqref{selfdualrelA}. Hence, suppose now that $\cp$ is not doubly covered. 
 In case $(1)$, we have that  $l_{\cp}<k<r_{\cp}$ and so 
 we can apply \cref{dilation} to \eqref{selfdualrelA} as above to obtain 
\begin{eqnarray*}
-D_{\mu'}^{\la'} D_{\la'}^{\mu'}
&=& (-1)^{b(\cp)}\left(
2
\!\! \sum_{   \begin{subarray}{c} \mq\in \underline{\la}\cap \underline{\mu} \\ \cp\prec \mq\end{subarray}}
\!\!
(-1) ^{b(\mq)+1}(-1) D^{\Phi_k(\la)}_{\Phi_k(\la- \mq) } D^{\Phi_k(\la - \mq)  }_{\Phi_k(\la)} \right. \\
&& + 
2
\left. \!\! \sum_{   \begin{subarray}{c} \mq\in \underline{\la}\cap \underline{\mu} \\\cp\prec\prec \mq \end{subarray}}
\!\!
(-1) ^{b(\mq) } D^{\Phi_k(\la)}_{\Phi_k(\la- \mq) } D^{\Phi_k(\la - \mq)  }_{\Phi_k(\la)}   + 
  \!\!\sum_{  \begin{subarray}{c} \ot\in \underline{\la} \\ \ot \, \text{adj.}\ \cp \end{subarray}}
\!\!
 (-1)^{b(\ot)} D^{\Phi_k(\la)}_{\Phi_k(\la- \ot) } D^{\Phi_k(\la - \ot)  }_{\Phi_k(\la)}\right),
\end{eqnarray*}
where the additional power of $(-1)$ in the first summand comes from the fact that $ l_{\mq}<k< r_{\mq}$.
 The second summand counts the cups $\mq\in\underline{\mu}\cap\underline{\la}$ that doubly cover $\cp$ (and so  $k<  r_{\cp} <l_{\mq}$) and so the  signs are preserved. 
 Similarly,  the third summand consists of
  adjacent terms for which the signs are again preserved.

For any $\cp\neq\rs\in \underline{\la}$,  we have that $\Phi_k(\la-\rs) = \la' - \rs'$, by \cref{breadth}.
 Moreover, for $l_\cp<k< r_{\cp}$, $\Phi_k$ gives rise to bijections as in
\eqref{akjhsdfhjkasdfhjkladsfhljkasf} and the additional bijections
\begin{align}
\{\mq'\mid \cp'  \prec\mq'\} \leftrightarrow
  \{\mq \mid \cp  \prec\mq\}
    \qquad 
   \{ \ot'   \text{ adjacent to } \cp'\} \leftrightarrow \{ \ot   \text{ adjacent to } \cp\}.
   \end{align}
Now, we have that $b(\cp) = b(\cp')+1$, $b(\mq) = b(\mq')+1$ if $\cp'\prec \mq'$; $b(\mq) = b(\mq')+2$ if $\cp'\prec\prec \mq'$ and $b(\ot) = b(\ot')$ if $\ot'$ adjacent to $\cp'$. Therefore, applying \cref{dilation} implies that the relation holds by induction.

Finally, if $k<l_{\cp}$ then as above $\cp$ must be decorated and is as pictured in \cref{adj lem1} d). Here $\cp$ is not (doubly) covered by any cup and is only adjacent to a single undecorated cup. 
Applying \cref{dilation} we get:
\begin{eqnarray*}
D_{\mu'}^{\la'} D_{\la'}^{\mu'}
&=& (-1)^{b(\cp)}(
 (-1)^{b(\ot)} D^{\Phi_k(\la)}_{\Phi_k(\la- \ot) } D^{\Phi_k(\la - \ot)  }_{\Phi_k(\la)}).
\end{eqnarray*}
Since $b(\cp')+2=b(\cp)$ and $b(\ot)=b(\ot')$ the relation follows by \cref{dilation} and induction as in the other subcases above.

\medskip
\noindent 
\textbf{Proof of the commuting relations.}
The pair of cups  $\cp$ and $\mq$  commute  if and only if the Dyck paths  they generate,
 $\color{cyan}P$ and $\color{magenta}Q$, have empty intersection. 
 The relations then follow by applying spot idempotent and fork idempotent relations as illustrated in \cref{zeroorientation}.

\medskip
\noindent 
\textbf{Proof of the non-commuting relations. }
In \eqref{noncommutingcup1}, we can assume that $\mq\prec \cp$ without loss of generality. 
 Comparing the Dyck paths generated by $\cp$ and $\mq$, we have that $  {\sf first}(\MQ)-{\sf first}(\CP)\geq 1$, 
  $  {\sf last}(\MQ)-{\sf last}(\CP)\geq 1$ and ${\sf ct}(\MQ)\subset {\sf ct}(\CP)$. 
  Hence, to calculate the product $D^{\mu -\cp}_\mu D^\mu_{\mu-\mq}$ using the corresponding oriented tableaux, as in \cref{Non-com pic}, we need only consider tiles whose content belongs to ${\sf ct}(\MQ)$.

  \begin{figure}[ht!]
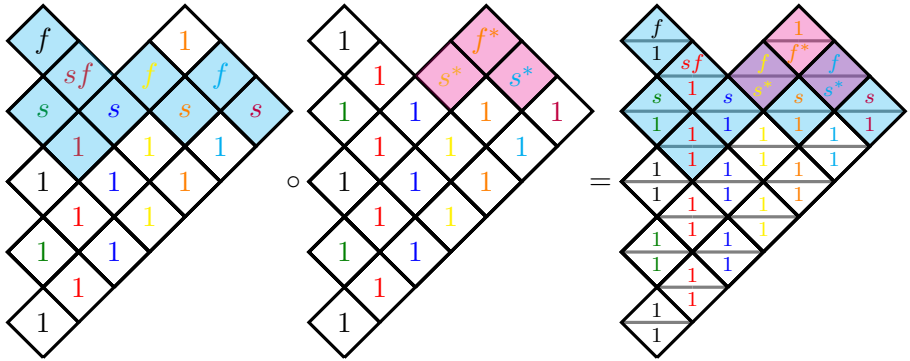

$$
 \begin{minipage}{3.6cm}
 
\end{minipage}
$$
\caption{An example of a product of oriented tableaux for non-commuting $\CP$ and $\MQ$.}
\label{Non-com pic}
\end{figure}

For a tile of content $k$, with ${\sf first}(\MQ)\leq k \leq {\sf last}(\MQ)$, we apply 
the first two spot-fork relations in \cref{some pics} 
 to the product tableau and hence obtain trivial $1$s and $0$s 
in these positions (see for example, the left of \cref{Non-com pic2}). 
  By deleting the 0-oriented tiles we obtain the unique oriented tableau of 
shape $\mu-\mq$ and weight $\mu-\mq- {\color{darkgreen}q}^1 -{\color{orange}q}^2$ (an example is depicted on the left of \cref{Non-com pic2}), and the result follows. 
\color{black}

\begin{figure}[ht!]
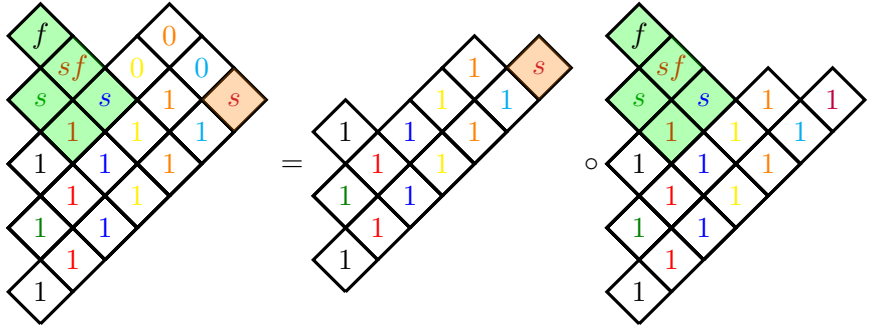


$$
 \begin{minipage}{3.5cm}
 
\end{minipage}
$$

 \caption{On the left we have applied the spot-fork relations of \cref{spotpic1} to the tableau from the right of 
 \cref{Non-com pic}. On the right we have rewritten this oriented tableau as the product of two commuting tableaux.}
  \label{Non-com pic2}
\end{figure}

\medskip
\noindent 
\textbf{Proof of the doubly non-commuting relations. }
 In \eqref{noncommutingcup2}, we have that $\mq \prec \prec \cp$ without loss of generality. Comparing the Dyck paths generated by $\cp$ and $\mq$, we have that ${\sf last}(\MQ)<{\sf first}(\CP) $ and  ${\sf ct}(\MQ)\subset {\sf ct}(\CP)$. Hence, to calculate the product $\stt= D^{\mu -\cp}_\mu D^\mu_{\mu-\mq}$ using the corresponding oriented tableaux, as in \cref{Doubly Non-com pic1}, we initially need only consider tiles of content $k$, where $k\leq{\sf last}(\MQ)$.
 

\begin{figure}[ht!]
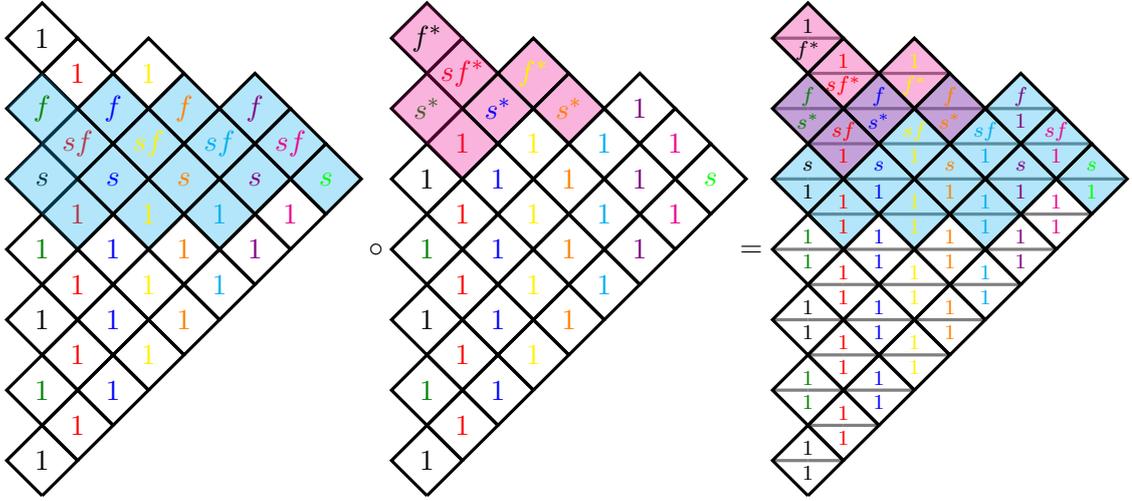

$$
 \begin{minipage}{4.7cm}
 
\end{minipage}
$$
 \caption{An example of a product of tableaux for doubly non-commuting $\CP$ and $\MQ$.}
  \label{Doubly Non-com pic1}
\end{figure}

{\bf Step 1. } 
First we consider tiles of content $k$, where ${\sf first}(\MQ) <k \leq {\sf last}(\MQ)$. For $k$ odd we apply the first relation from \cref{some pics} and the spot relation pictured in \cref{spotpic2}. For $k$ even we apply the third relation from \cref{some pics}. The resulting oriented tableau is obtained from our original tableau by replacing all tiles from $\MQ$ in this interval as to have $0$-orientation and replacing all $sf$-oriented tiles (by their nature in $\CP$) in this interval so as to have $f$-orientation (see the left-hand of \cref{Doubly Non-com pic2} for an example). The remaining tiles in this interval have fixed orientation.

{\bf Step 2. } Next we deal with the tiles of content $k$, where $0\leq k\leq {\sf first}(\MQ)$. 
If $\cp$ is undecorated we leave these tiles unchanged.
If $\cp$ is decorated, then we apply the final relation from \cref{some pics} to tiles of non-zero even content in this interval;
to tiles with odd content or content equal to zero, apply the first two relations in \cref{some pics} in combination (see also \cref{blank space baby}). 
The resulting oriented tableau is obtained from the tableau obtained as output of Step 1 
 by replacing all tiles from $\MQ$ in this interval so as to have $0$-orientation and replacing all the tiles of $\CP$ in the interval as to have $1$-orientation (see the left-hand of \cref{Doubly Non-com pic2} for an example). 
Deleting the $0$-oriented tiles we obtain an oriented tableau, $\stt$, of shape $\mu-\mq$ and weight $\mu-\cp-\mq$.

We now apply the rules of \cref{Go forth and multiply} to the product $\stt'=D^{\mu -\cp }_{\mu - \mq - {\color{orange}q}^2}D^{\mu -\mq- {\color{orange}q}^2}_{\mu -\mq}$, from the right-hand side of \eqref{noncommutingcup2}, in order to obtain 
the oriented tableau $\stt$  of shape $\mu-\mq$ (obtained as output of Step 2 above).

\begin{figure}[ht!]
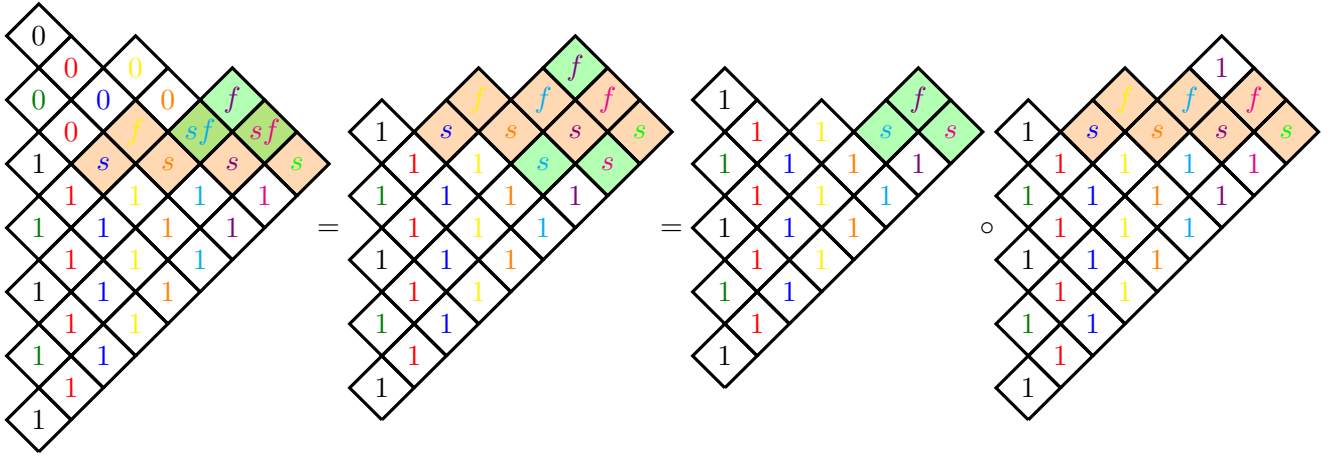

$$
 \begin{minipage}{4cm}
 
\end{minipage}
$$

 \caption{On the left we have applied the spot-fork relations of \cref{some pics} to the tableau from the right of 
 \cref{Doubly Non-com pic1}. The first equality is an application of \cref{spotpic1} and the final equality comes from the method of multiplying non-commuting tableaux.}
  \label{Doubly Non-com pic2}
\end{figure}

We have that ${\color{darkgreen}q}^1$ and  ${\color{orange}q}^2$ are a non-commuting pair with ${\color{darkgreen}q}^1\prec {\color{orange}q}^2$. 
  Hence (just as in the non-commuting proof above), to calculate $\stt'$ using the corresponding oriented tableaux (of shape $\mu-\mq$), pictured in \cref{Doubly Non-com pic2}, we need only consider tiles whose content belongs to ${\sf ct}({\color{darkgreen}Q}^1)$.
 
To tiles of even content $k$ of $\stt'$, for ${\sf first}({\color{darkgreen}Q}^1)\leq k \leq {\sf last}({\color{darkgreen}Q}^1)$ we
 apply the first equality in the second relation of \cref{spotpic1}.
This creates $sf$-oriented tiles that match up precisely with the $sf$-oriented tiles of $\stt$ (which arose from $sf$-oriented tiles in $\CP$).
 An example is pictured in the second equality of \cref{Doubly Non-com pic2}.
We obtain, $\stt$, the unique oriented tableau of 
shape $\mu-\mq$ and weight $\mu-\cp-\mq$ (as depicted on the left of \cref{Doubly Non-com pic2}), and \cref{noncommutingcup2}  follows.

\medskip
\noindent 
\textbf{Proof of the adjacency relation.}
Throughout let $\la=\mu-\cp$ as in \eqref{adjacentcup}. Suppose $\cp\in \underline{\mu}$ and $\ot \in \underline{\la}$ are two adjacent cups.
We will first prove that if $\langle \cp\cup \ot\rangle_\mu$ exists, then the adjacency relation holds as in \eqref{adjacentcup}. We will do this by initially proving the relation holds for cup diagrams such that $\langle \cp\cup \ot\rangle_\mu$ is of minimal breadth; we will then apply \cref{dilation} to show that \eqref{adjacentcup} holds for $\langle \cp\cup \ot\rangle_\mu$ of any breadth. We split the proof into two separate parts; first when $\cp$ is covered and then when $\cp$ is doubly covered.

\medskip
\noindent\textbf{The case where $\cp$ is covered. }  
First, we consider when $\cp$ is covered. In this case $\cp$ is adjacent to two cups in $\underline{\la}$, which we label $\ot$ and $\gr$ from left to right. We will deal here only with $\ot$; the case with $\gr$ is essentially the same. From \cref{adj lem2} we have that $\ot$ is decorated if and only if $\langle {\color{cyan}p}\cup \ot\rangle_\mu$ is decorated. 
We first suppose $b(\langle \cp\cup \ot\rangle_\mu)=2$, from which it follows that $b(\cp)=b(\ot)=1$.

We start with the left-hand side of \eqref{adjacentcup} and the product $D^{\mu - {\color{cyan}p} - \ot}_{\mu - {\color{cyan}p}}D^{\mu - {\color{cyan}p}}_\mu$. We record this as an oriented tableau of shape $\mu$ and weight $\mu-\cp-\ot$ as on the left-hand side of \cref{adj 2}.
Apply \cref{nullbraidontiles} (ii) to this oriented tableau to create an oriented tableau as in the first equality of \cref{adj 2}.
The dual labels on the oriented tableau correspond to a cup $S^*$ of breadth $1$ in $\mu-\cp-\ot$ which is also adjacent to $\cp$; hence $S^*$ must be equal to the cup $\gr$, as required. The non-dual labels trace out a breadth two cup $S \in \underline{\mu}$ with 
$l_S=l_\ot$; this is the vertex of $\ot$ that is not common with $\cp$; therefore $S=\langle \cp\cup \ot\rangle_\mu$, by definition.
Taken together, this implies that the oriented tableau is that of the product $D^{\mu - \cp- \ot}_{\mu - \langle \cp\cup \ot\rangle_\mu}D^{\mu - \langle \cp\cup \ot\rangle_\mu}_\mu$, where we note that $\mu-\cp-\ot-\gr=\mu-\langle \cp \cup \ot\rangle_\mu$. 
 This combined with the fact that $b(\langle \cp\cup \ot\rangle_\mu)=2$ and $b(\ot)=1$ proves that \cref{adjacentcup} holds.

\begin{figure}
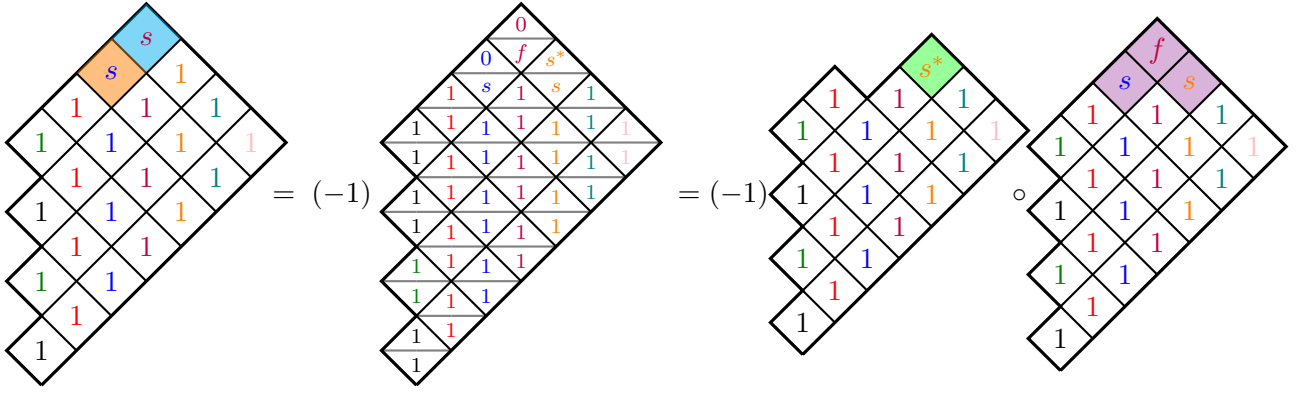
$$
\begin{minipage}{3.2cm}
  \end{minipage}
$$ 
\caption{An example of the adjacency relation on oriented tableaux when $\cp$ is covered and $b(\langle \cp\cup \ot\rangle_\mu)=2$. In the first equality we apply \cref{nullbraidontiles} (ii) and in the second equality we rewrite this resulting tableau as the tableau that correspond to the product $D^{\mu - \cp- \ot}_{\mu - \langle \cp\cup \ot\rangle_\mu}D^{\mu - \langle \cp\cup \ot\rangle_\mu}_\mu$.
}
\label{adj 2}
\end{figure}

Suppose now that $b(\langle \cp\cup \ot\rangle_\mu)>2$. Then there exists a $0\leq k<r_{\langle \cp\cup \ot\rangle_\mu}$ such that $(\Phi_k(\la), \Phi_k(\mu))=(\la', \mu')$,
where $\la' = \mu'-\cp'$ and $b(\langle \cp\cup \ot\rangle_\mu')< b(\langle \cp\cup \ot\rangle_\mu)$.
We apply $\Phi_k$ from \cref{dilation} to both sides of \eqref{adjacentcup} and depending on the exact value of $k$ we obtain one of three equations:

(i) If $0\leq k<l_{\langle \cp\cup \ot\rangle_\mu}$, then we obtain:
$$(\pm) D^{\mu' - \cp' - \ot'}_{\mu' - \cp'}D^{\mu' - \cp'}_{\mu'} = (-1)^{b(\langle \cp\cup \ot\rangle_\mu) - b(\ot)} 
(\pm) D^{\mu' - \cp' - \ot'}_{\mu' - \langle \cp'\cup \ot'\rangle_{\mu'}} D^{\mu' - \langle \cp'\cup \ot'\rangle_{\mu'}}_{\mu'}$$
where the common minus sign only occurs if $k=0,1$ and if $\langle \cp\cup \ot\rangle_\mu$ (and hence $\ot$) is decorated. Noting that $b(\langle \cp\cup \ot\rangle_\mu')\equiv b(\langle \cp\cup \ot\rangle_\mu) \pmod 2$ and that $b(\ot')\equiv b(\ot)\pmod2$ shows that if \eqref{adjacentcup} holds for $\la'=\mu-\cp'$ then \eqref{adjacentcup} holds for $\la=\mu-\cp$ as well.

(ii) If $l_{\langle \cp\cup \ot\rangle_\mu}<k<l_\cp$ or $r_\cp<k<r_{\langle \cp\cup \ot\rangle_\mu}$, then we obtain:
$$(\pm i)D^{\mu' - \cp' - \ot'}_{\mu' - \cp'}D^{\mu' - \cp'}_{\mu'} = (-1)^{b(\langle \cp\cup \ot\rangle_\mu) - b(\ot)} 
(\pm i) D^{\mu' - \cp' - \ot'}_{\mu' - \langle \cp'\cup \ot'\rangle_{\mu'}} D^{\mu' - \langle \cp'\cup \ot'\rangle_{\mu'}}_{\mu'}$$
Note that $b(\langle \cp\cup \ot\rangle_\mu')+1=b(\langle \cp\cup \ot\rangle_\mu)$ and $b(\ot')+1=b(\ot)$; thus if \eqref{adjacentcup} holds for $\la'=\mu-\cp'$ then  \eqref{adjacentcup} holds for $\la=\mu-\cp$ as well.

(iii) Finally, if $l_\cp<k<r_\cp$, then we obtain:
$$(\pm i)D^{\mu' - \cp' - \ot'}_{\mu' - \cp'}D^{\mu' - \cp'}_{\mu'} =(-1)^{b(\langle \cp\cup \ot\rangle_\mu) - b(\ot)} 
(\mp i) D^{\mu' - \cp' - \ot'}_{\mu' - \langle \cp'\cup \ot'\rangle_{\mu'}} D^{\mu' - \langle \cp'\cup \ot'\rangle_{\mu'}}_{\mu'}
$$
Note that $b(\langle \cp\cup \ot\rangle_\mu')+1=b(\langle \cp\cup \ot\rangle_\mu)$ and that $b(\ot')=b(\ot)$; therefore that if \eqref{adjacentcup} holds for $\la'=\mu-\cp'$ then  \eqref{adjacentcup} holds for $\la=\mu-\cp$.

Hence, by repeatedly applying \cref{dilation}, we can reduce the breadth until $b(\langle \cp\cup \ot\rangle_\mu)=2$, and hence deduce that \eqref{adjacentcup} holds for any $\cp$ such that there exists some $\mq \in \underline{\mu}$ such that $\cp \prec \mq$.

\medskip
 \noindent\textbf{The case where $\cp$ is doubly covered. } 
Next we consider when $\cp$ is doubly covered.
By \cref{adj lem2}, we have that $\cp$ is adjacent to two non-commuting cups $\gr\prec\ot$, where $\langle \cp\cup \ot\rangle_\mu$ exists but $\langle \cp\cup \gr\rangle_\mu$ does not. Suppose first that $b(\langle \cp\cup \ot\rangle_\mu)=3$ (and so is minimal) and hence $b(\cp)=1$. 

As above, we start with the left-hand side of \eqref{adjacentcup} and the product $D^{\mu - {\color{cyan}p} - \ot}_{\mu - {\color{cyan}p}}D^{\mu - {\color{cyan}p}}_\mu$. We record this as an oriented tableau of shape $\mu$ and weight $\mu-\cp-\ot$, as on the left-hand side of \cref{adj 7}.
 Apply \cref{adjacencywayback} (ii) to this oriented tableau, as in the first equality of \cref{adj 7}. The dual labels on the oriented tableau correspond to a cup $S^*$ of breadth $1$ on $\mu-\cp-\ot$ which is also adjacent to $\cp$ and is therefore actually the cup $\gr$. 
The non-dual labels trace out a breadth three decorated cup $S$ on $\mu$ with 
$r_S=r_\ot$; this is the vertex of $\ot$ that is not common with $\cp$. So by definition, $S=\langle \cp\cup \ot\rangle_\mu$. 
 Taken together, as above, this implies that the oriented tableau is that of the product $D^{\mu - \cp- \ot}_{\mu - \langle \cp\cup \ot\rangle_\mu}D^{\mu - \langle \cp\cup \ot\rangle_\mu}_\mu$, where again we note that $\mu-\cp-\ot-\gr=\mu-\langle \cp \cup \ot\rangle_\mu$. We have equated the oriented tableaux and it only remains to check the signs. Noting that $b(\langle \cp\cup \ot\rangle_\mu)=3$ and $b(\ot)=2$ proves that \cref{adjacentcup} holds.

\begin{figure}[ht!]$$
\begin{minipage}{3cm}
  \end{minipage}
$$
\caption{An example of the adjacency relation on oriented tableaux when $\cp$ is doubly covered and $b(\langle \cp\cup \ot\rangle_\mu)=3$. In the first equality we apply \cref{nullbraidontiles} (iii) and in the second equality we rewrite this resulting tableau as the tableaux that correspond to the product $D^{\mu - \cp- \ot}_{\mu - \langle \cp\cup \ot\rangle_\mu}D^{\mu - \langle \cp\cup \ot\rangle_\mu}_\mu$.}
\label{adj 7}
\end{figure}

Suppose now that $b(\langle \cp\cup \ot\rangle_\mu)>3$. Then there exists a $0\leq k<r_{\langle \cp\cup \ot\rangle_\mu}$ such that $(\Phi_k(\la), \Phi_k(\mu))=(\la', \mu')$,
where $\la' = \mu'-\cp'$ and $b(\langle \cp\cup \ot\rangle_\mu')< b(\langle \cp\cup \ot\rangle_\mu)$. 
We apply $\Phi_k$ from \cref{dilation} to both sides of \eqref{adjacentcup}. Depending on the exact value of $k$ we obtain one of five equations which we go through case-by-case below:

(i) If $0\leq k<l_\cp$, then we obtain:
$$ (\pm)D^{\mu' - \cp' - \ot'}_{\mu' - \cp'}D^{\mu' - \cp'}_{\mu'}= (-1)^{b(\langle \cp\cup \ot\rangle_\mu) - b(\ot)} 
(\pm)D^{\mu' - \cp' - \ot'}_{\mu' - \langle \cp'\cup \ot'\rangle_{\mu'}} D^{\mu' - \langle \cp'\cup \ot'\rangle_{\mu'}}_{\mu'}
$$ Note that $b(\langle \cp\cup \ot\rangle_\mu')+2= b(\langle \cp\cup \ot\rangle_\mu)$ and that $b(\ot')\equiv b(\ot)\pmod2$; therefore if \eqref{adjacentcup} holds for $\la'=\mu-\cp'$ then \eqref{adjacentcup} holds for $\la=\mu-\cp$ too.

(ii) If $l_\cp=\frac{1}{2}$ and $k=1$, then we obtain:
$$(\pm i)^2 D^{\mu' - \cp' - \ot'}_{\mu' - \cp'}D^{\mu' - \cp'}_{\mu'}= (-1)^{b(\langle \cp\cup \ot\rangle_\mu) - b(\ot)} 
 (-)D^{\mu' - \cp' - \ot'}_{\mu' - \langle \cp'\cup \ot'\rangle_{\mu'}} D^{\mu' - \langle \cp'\cup \ot'\rangle_{\mu'}}_{\mu'}
$$Note that $b(\langle \cp\cup \ot\rangle_\mu')+1= b(\langle \cp\cup \ot\rangle_\mu)$ and $b(\ot')+1=b(\ot)$; thus if \eqref{adjacentcup} holds for $\la'=\mu-\cp'$ then  \eqref{adjacentcup} holds for $\la=\mu-\cp$ too.

(iii) If $l_\cp<k<r_\cp$, $k\neq1$, then we obtain:
$$(\pm i)^2 D^{\mu' - \cp' - \ot'}_{\mu' - \cp'}D^{\mu' - \cp'}_{\mu'}= (-1)^{b(\langle \cp\cup \ot\rangle_\mu) - b(\ot)} 
 D^{\mu' - \cp' - \ot'}_{\mu' - \langle \cp'\cup \ot'\rangle_{\mu'}} D^{\mu' - \langle \cp'\cup \ot'\rangle_{\mu'}}_{\mu'}
$$Note that $b(\langle \cp\cup \ot\rangle_\mu')+2= b(\langle \cp\cup \ot\rangle_\mu)$ and $b(\ot')+1=b(\ot)$; hence if \eqref{adjacentcup} holds for $\la'=\mu-\cp'$ then  \eqref{adjacentcup} holds for $\la=\mu-\cp$ too.

(iv) If $r_\cp<k<l_{\langle \cp\cup \ot\rangle_\mu}$, then we obtain:
$$(\pm i) D^{\mu' - \cp' - \ot'}_{\mu' - \cp'}D^{\mu' - \cp'}_{\mu'}D^{\mu - {\color{cyan}p}}_\mu = (-1)^{b(\langle \cp\cup \ot\rangle_\mu) - b(\ot)} 
 (\mp i) D^{\mu' - \cp' - \ot'}_{\mu' - \langle \cp'\cup \ot'\rangle_{\mu'}} D^{\mu' - \langle \cp'\cup \ot'\rangle_{\mu'}}_{\mu'}
$$Note that $b(\langle \cp\cup \ot\rangle_\mu')+2= b(\langle \cp\cup \ot\rangle_\mu)$ and $b(\ot')+1=b(\ot)$; thus if \eqref{adjacentcup} holds for $\la'=\mu-\cp'$ then  \eqref{adjacentcup} holds for $\la=\mu-\cp$ too.

(v) Finally if $l_{\langle \cp\cup \ot\rangle_\mu}<k<r_{\langle \cp\cup \ot\rangle_\mu}$, then we obtain:
$$(\pm i) D^{\mu' - \cp' - \ot'}_{\mu' - \cp'}D^{\mu' - \cp'}_{\mu'}D^{\mu - {\color{cyan}p}}_\mu =(-1)^{b(\langle \cp\cup \ot\rangle_\mu) - b(\ot)} 
 (\pm i) D^{\mu' - \cp' - \ot'}_{\mu' - \langle \cp'\cup \ot'\rangle_{\mu'}} D^{\mu' - \langle \cp'\cup \ot'\rangle_{\mu'}}_{\mu'}
$$Note that $b(\langle \cp\cup \ot\rangle_\mu')+1= b(\langle \cp\cup \ot\rangle_\mu)$ and $b(\ot')+1=b(\ot)$; therefore if \eqref{adjacentcup} holds for $\la'=\mu-\cp'$ then \eqref{adjacentcup} holds for $\la=\mu-\cp$ too.

Hence, by repeatedly applying \cref{dilation} we can reduce the breadth until $b(\langle \cp\cup \ot\rangle_\mu)=3$, and we can hence deduce that \eqref{adjacentcup} holds for any doubly covered $\cp$ and adjacent $\ot$ for which $\langle \cp\cup \ot\rangle_\mu$ exists.
\medskip

\noindent\textbf{The case where $\langle \cp\cup \ot\rangle_\mu$ does not exist. } 
So far we have established that if $\langle \cp\cup \ot\rangle_\mu$ exists then \eqref{adjacentcup} holds as claimed. We now wish to prove that this is the only occasion that the product  $D^{\mu - \cp- \ot}_{\mu-\cp}D^{\mu-\cp}_{\mu}$ is non-zero. We have previously listed all possible cases of adjacent cups in \cref{adj lem1,adj lem2} and we will go through each of these cases individually in what follows. Initially we will assume that $b(\cp)=b(\ot)=1$ and then afterwards we will apply \cref{dilation} to conclude that every case can be reduced down to this minimal one.

We first begin with the cases listed in {\cref{adj lem1}:

We first consider \cref{adj lem1} a). Since $\ot$ is decorated and $b(\ot)=1$, we must have that $l_\ot=\frac{1}{2}$. In which case the equality follows trivially.

Next, we consider \cref{adj lem1} b). Since $\cp$ is decorated and $b(\cp)=1$ we must have that $l_\cp=\frac{1}{2}$. In which case the equality follows simply by applying \cref{adjacencywayback} (iii).

We now consider case c) of \cref{adj lem1}. We record the product $D^{\mu - {\color{cyan}p} -\ot}_{\mu - {\color{cyan}p}}D^{\mu - {\color{cyan}p}}_\mu$ as an oriented tableau of shape $\mu$ and weight $\mu-\cp-\ot$. 
Denote by $S\in\underline{\mu}$ the propagating ray such that $l_S=l_\ot$.
Apply \cref{nullbraidontiles} (ii) to the oriented tableau to create a gap labelled tile;
then repeatedly apply \cref{nullbraidontiles} (iii), to create 
gap labelled tiles--- this process terminates by hitting the boundary of the tableau, as in \cref{adj 200}. This repeated application of \cref{nullbraidontiles} (iii) creates gap-labelled tiles (and dork-labelled tiles) in all tiles which have non-empty intersection with $S$.
Since the strand $S$ is propagating, the tableau is automatically $0$ by the cyclotomic relations.

\begin{figure}
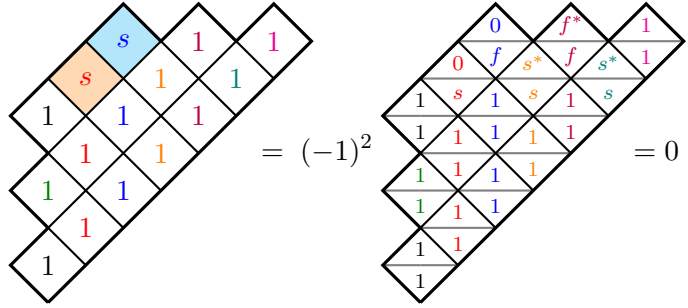
$$
\begin{minipage}{3cm}
  \end{minipage}\;\;=0
$$ 
\caption{An example of an oriented tableau from \cref{adj lem1} c). The resulting tableau is zero by the cyclotomic relations.}
\label{adj 200}
\end{figure}

We now consider the final case, d), of \cref{adj lem1}. We note that we must have $r_\cp=l_\ot=2k-\frac{1}{2}$ for some $k\in \mathbb{Z}$. Apply \cref{nullbraidontiles} (i) to create a gap labelled tile 
and then repeatedly apply \cref{nullbraidontiles} (iii) creating 
gap labelled tiles repeatedly. Since $r_\cp=l_\ot=2k-\frac{1}{2}$ the gap labelled tiles have even content and so this process will terminate with a gap-labelled ${\color{red}s_2}$-tile as in \cref{adj 4}, which is immediately zero by \cref{adjacencywayback} (iii).

\begin{figure}[ht!]
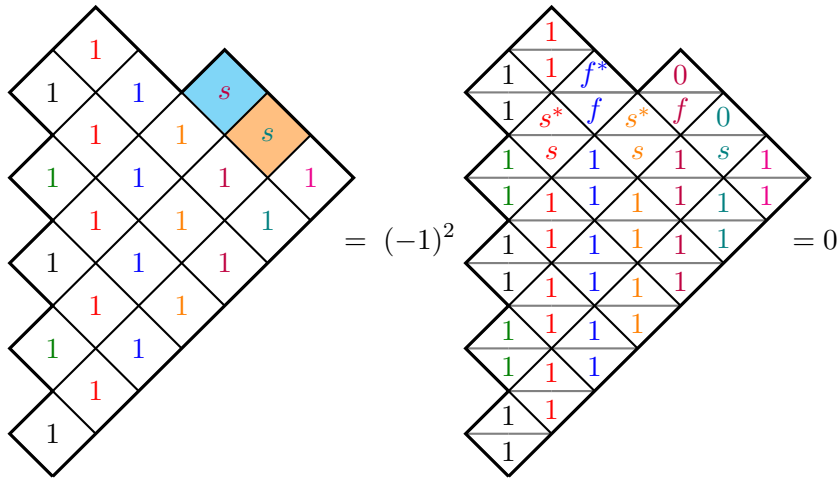
$$
\begin{minipage}{4.1
cm}
  \end{minipage}\;\;=0
$$ 
\caption{An example of an oriented tableau from \cref{adj lem1} d). The resulting tableau is zero by \cref{adjacencywayback} (ii).}
\label{adj 4}
\end{figure}

We now move onto the two cases in \cref{adj lem2}. In both \cref{adj lem2} c) and d), $\cp$ is doubly covered and is adjacent to two non-commuting cups $\gr\prec\ot$, where $\langle \cp\cup \gr\rangle_\mu$ does not exist. 

So, first we consider \cref{adj lem2} c). We note that $r_\cp=l_\gr=2k-\frac{1}{2}$ and therefore applying \cref{nullbraidontiles} (i) and \cref{nullbraidontiles} (iii) repeatedly, creates a tableau with a gap-labelled ${\color{red}s_2}$-tile, just as above, which is immediately zero by \cref{adjacencywayback} (iii).

Finally, we consider the case of \cref{adj lem2} d). In this case $\cp$ is decorated and $b(\cp)=1$, we must have that $l_\cp=\frac{1}{2}$ and the equality follows immediately from \cref{adjacencywayback} (iii).

We have now checked the complete set of minimal breadth cases from  \cref{adj lem1,adj lem2} for which $\langle \cp\cup \ot\rangle_\mu$ does not exist. It remains to show that this is sufficient to prove that the relation holds in general.

If $b(\cp)\geq 2$ or $b(\ot)\geq 2$ then there exists some $k$, with $k<\max \{r_\ot, r_\cp\}$, such that $(\Phi_k(\la), \Phi_k(\mu))=(\la', \mu')$,
where $\la' = \mu'-\cp'$ and $b(\cp')< b(\cp)$ or $b(\ot')< b(\ot)$ (or both). A repeated application of \cref{dilation} to $D^{\mu - \cp- \ot}_{\mu-\cp}D^{\mu-\cp}_{\mu}$ until $b(\cp)=b(\ot)=1$ shows that \eqref{adjacentcup} holds for any adjacent cups $\cp$ and $\ot$.

\bigskip

It remains to show that the above relations form a complete list. This result follows in an identical manner to that at the end of the equivalent parabolic type $(A_n , A_{n-k} \times A_k)$ result (\cite[Theorem 7.3]{ChrisDyckPaper}); all that differs below is the shift of emphasis from the Dyck paths, used in parabolic type $(A_n , A_{n-k} \times A_k)$, towards oriented cups used in this paper to provide the presentation for $\mathcal{H}_{(D_n A_{n-1})}$. 
As such, we will reframe the proof of completeness of relations in the language of this paper below for comprehensiveness, but we emphasise that once we have constructed the basis as in \cref{cellular basis} and noted that the algebra $\mathcal{H}_{(D_n A_{n-1})}$ is generated as in \cref{generatorsarewhatweasay}, all that remains to check is that (\ref{rel1})--(\ref{adjacentcup}) cover every possible product of degree $1$ generators; the result then follows using general properties of the light leaves basis.

\smallskip

\noindent
{\bf Completeness of relations. } 
It is enough to show that using only the relations in (\ref{rel1})--(\ref{adjacentcup}) and their duals, that we can rewrite any product of $k$ degree $1$ generators as a linear combination of light leaves basis elements. We do this by inducting on $k$. For $k=1$ there is nothing to prove. For $k=2$, note that (\ref{rel1})--(\ref{adjacentcup}) and their duals cover precisely all possible (non-zero) products of two degree $1$ generators, rewriting these as linear combinations of basis elements, as required. Now, assume that the result holds for $k$ generators and consider a product of $k+1$ generators. By induction, it is enough to consider a product of the form
$$D^\mu_\la D^\la_\nu D^\nu_{\nu \pm \cp}
$$where $D^\mu_\la D^\la_\nu$ is a basis element of degree $k$ and $\cp \in \underline{\nu}$ or $\cp$ is such that $\underline{\nu+\cp}$ is a cup diagram. 
To show that this product can be rewritten as a linear combination of basis elements, we will also induct on $\ell(\la)+\ell(\nu)$. If $ \ell(\la)+\ell(\nu) = 0$, then $\la = \nu = \emptyset$, $\nu +\cp = (1)$ and, by using (\ref{rel1}), we have that 
$$D^\mu_\emptyset D^\emptyset_\emptyset D^\emptyset_{(1)} = D^\mu_\emptyset  D^\emptyset_{(1)}$$which 
is a basis element, as required. Now assume that $\ell(\la) + \ell(\nu) \geq 1$. If $\la \neq \nu$ then we can write 
$$D^\la_\nu = D^\la_{\nu -\mq} D^{\nu - \mq}_\nu$$for 
some $\mq\in \underline{\nu}$. Now using (\ref{rel1}) to (\ref{adjacentcup}) we can write
$$D^{\nu - \mq}_\nu D^\nu_{\nu \pm \cp} = \sum_{\nu'}c_{\nu'}D^{\nu - \mq}_{\nu'}D^{\nu'}_{\nu \pm \cp}$$for 
some $c_{\nu '}\in \Bbbk$ where $\ell(\nu') \leq \ell(\nu - \mq) < \ell(\nu)$ and 
$\ell(\nu') \leq \ell(\nu \pm \cp)$, with $D^{\nu'}_{\nu \pm \cp}$ and $D^{\nu - \mq}_\nu$ still of degree 1. Then we have
\begin{eqnarray*}
D^\mu_\la D^\la_\nu D^\nu_{\nu \pm \cp} &=& D^\mu_\la D^\la_{\nu -\mq}D^{\nu - \mq}_\nu D^\nu_{\nu \pm \cp} \\
&=& \sum_{\nu '} c_{\nu'} (D^\mu_\la D^\la_{\nu - \mq}D^{\nu - \mq}_{\nu'})D^{\nu'}_{\nu \pm \cp}\\
&=& \sum_{\nu' , \la'} d_{\nu ' , \la'} D^\mu_{\la'}D^{\la'}_{\nu'}D^{\nu'}_{\nu\pm \cp}.
\end{eqnarray*}
where the final equality follows by induction.
 It remains to consider the case where $\la = \nu$. Here we must have $\mu \neq \la$. Observe that 
$$D^\mu_\la D^\la_\la D^\la_{\la +\cp} = D^\mu_\la D^\la_{\la +\cp}$$
by (\ref{rel1}) and this is a basis element. The last case to consider is
$$D^\mu_\la D^\la_\la D^\la_{\la -\cp} = D^\mu_\la D^\la_{\la -\cp}.$$
As $\mu\neq \la$ we have 
$D^\mu_\la = D^\mu_{ \alpha}D^{\alpha}_\la$, where $\la=\alpha-\mq$ for some $\mq\in \underline{\alpha}$ and so
$$D^\mu_\la D^\la_\la D^\la_{\la -\cp} = D^\mu_{\la +\mq}D^{\la +\mq}_\la D^\la_{\la -\cp}.$$
Now, using (\ref{rel1}) to (\ref{adjacentcup}) we have 
$$D^{\alpha}_\la D^\la_{\la -\cp} =D^{\la +\mq}_\la D^\la_{\la -\cp} = \sum_{\nu'}c_{\nu'} D^{\la + \mq}_{\nu'} D^{\nu'}_{\la -\cp}$$
with $\ell(\nu') \leq \ell (\la - \cp)<\ell(\la)=\ell(\nu)$ and so 
\begin{eqnarray*}
D^\mu_{\alpha}D^{\alpha}_\la D^\la_{\la -\cp} = D^\mu_{\la +\mq}D^{\la +\mq}_\la D^\la_{\la -\cp} &=& \sum_{\nu'}c_{\nu'} (D^\mu_{\la + \mq} D^{\la + \mq}_{\nu'}) D^{\nu'}_{\la -\cp}\\
&=& \sum_{\la' , \nu'}d_{\la' , \nu'} D^\mu_{\la '} D^{\la'}_{\nu'}D^{\nu'}_{\la - \cp}
\end{eqnarray*}
using induction as ${\rm deg} D^\mu_{\alpha} D^{\alpha}_{\nu'} = k$, with $\ell(\la')\leq\ell(\nu') \leq \ell(\la -\cp)< \ell(\la)$. Now as $\ell(\la')+\ell(\nu')<\ell(\la)+\ell(\nu)$, we're done by induction.
 \qed

\section{Submodule Structure of Cell Modules} 
\label{structure}

 In this section, we make use of the benefits of \cref{presentation} to calculate the strong Alperin diagrams 
 of  cell modules of  $\mathcal{H}_{(D_n, A_{n-1})}$.
The graded composition factors of these cell  modules were explicitly calculated in 
 \cite{MR2813567,LEJCZYK_STROPPEL_2013,TypeDKhov}. 
 Our focus here is to determine the explicit extensions interlacing these composition factors, expressed in terms of our decorated oriented cup-cap combinatorics.
 In this section, we closely mirror the approach of \cite[Section 8]{ChrisDyckPaper}. By utilising their quadratic presentation of the Hecke category in parabolic type $(A_n , A_{n-k} \times A_k)$, they produce the equivalent results for cell modules in type $A$. Accordingly, where possible, we will choose our notation to coincide.
 
For the remainder of the paper, we will assume that $\Bbbk$ is a field.
 As noted in \cref{cellular basis}, the algebra $\mathcal{H}_{(D_n, A_{n-1})}$ is a basic (positively) graded quasi-hereditary algebra with graded cellular basis given by $$\{D_\la^\mu D^\la_\nu \mid  \la,\mu,\nu\in \mptn  \,\, \mbox{with}\,\, \underline{\mu}\la, \underline\nu\la \text{ oriented} \}.$$ 
For $\la \in \mptn$, we write $\mathcal{H}_{(D_n, A_{n-1})}^{\leq \la} = {\rm span}\{D^\mu_\alpha D^\alpha_\la \, : \, \alpha, \mu  \in \mptn , \, \alpha \leq \la\}$ and $\mathcal{H}_{(D_n, A_{n-1})}^{< \la} = {\rm span}\{D^\mu_\alpha D^\alpha_\la \, : \, \alpha, \mu  \in \mptn , \, \alpha < \la\}$. 
Set $${\rm DP}(\la):=\{ \mu\in \mptn\, : \,  \underline{\mu}\la \text{ oriented}\},$$ the (left) {\sf cell module} $\Delta_n(\la) = \mathcal{H}_{(D_n, A_{n-1})}^{\leq \la} / \mathcal{H}_{(D_n, A_{n-1})}^{< \la} $ has a basis given by 
 $$\{ u_\mu := D^\mu_\la + \mathcal{H}_{(D_n, A_{n-1})}^{<\la} \,  : \, \mu \in {\rm DP}(\la)\}.$$  Each $u_\mu$ generates a submodule of $\Delta_n(\la)$ with a $1$-dimensional simple head, which we denote by $L_n(\mu)$. 
 
The algebra $\mathcal{H}_{(D_n, A_{n-1})}$ is positively graded and so the grading provides a submodule filtration of $\Delta_n(\la)$. We can decompose ${\rm DP}(\la)$ as
 $${\rm DP}(\la) = \bigsqcup_{k \geq 0} {\rm DP}_k(\la) \quad \text{where} \quad  {\rm DP}_k(\la) = \{\mu\in {\rm DP}(\la) \, : \, \underline{\mu}\la \text{ is oriented of degree $k$}\}.$$
 \color{black}
 Note further that the algebra $\mathcal{H}_{(D_n, A_{n-1})}$ is generated in degrees $0$ and $1$. This implies that in order to describe the full submodule structure it is enough to find, for each $\mu \in {\rm DP}_k(\la)$, the set of all $\nu \in {\rm DP}_{k+1}(\la)$ such that 
 $$u_\nu = c  D^\nu_\mu u_\mu$$for 
 some $c\in \Bbbk$ and $\mu = \nu \pm \cp$ for some $\cp\in \underline{\nu}$ or $\cp \in \underline{\mu}$ respectively. This is a necessary condition for the existence of an extension between $L_{n}(\mu)$ and $L_{n}(\nu)$ in $\Delta_{(D_n, A_{n-1})}(\la)$. We will show that it is also sufficient.

Assume $\la= \mu- \sum_{i=1}^{k}\mq^i$. First suppose $\mu= \nu - \cp$. Note that $\underline{\nu}\la$ is oriented if and only if $\cp$ is not adjacent to any of the $\mq^i$ and hence we can write $\la = \nu - \sum_{i=1}^{k}\mq^i - \cp$. We have that
 $$D^{\nu}_\mu D^\mu_\la = D^{\nu}_\la$$ or  $$D^{\mu+\cp}_\mu D^\mu_\la = D^{\mu+\cp}_\la$$ 
 using the light leaves basis. 

Now suppose $\mu= \nu + \cp$. This means that ${\rm deg}(\underline{\mu}\la)+1 = {\rm deg}(\underline{\nu}\la)$; this implies that $\cp$ is (doubly) covered by a (doubly) non-commuting $\mq \in \{\mq^i\}$. We will prove by induction on ${\rm deg}(\underline{\mu}\la)$ that $D^{\mu-\cp}_\mu D^\mu_\la = D^{\mu - \cp}_\la$. 
We will explicitly prove this only for the doubly non-commuting case (the non-commuting case follows in an almost identical manner).

If ${\rm deg}(\underline{\mu}\la) = 1$, then $\mu - \mq=\la$ and the doubly non-commuting relation gives that
 $$D^{\mu -\cp}_\mu D^\mu_{\mu-\mq} = D^{\mu -\cp}_{\mu - \cp - \color{orange}{q}^1} D^{\mu - \cp - {\color{orange}q}^1}_{\mu - \mq } = D^{\mu - \cp}_{\mu-\mq}$$ where ${\color{darkgreen}{q}^2} \prec \color{orange}{q}^1$ are the two concentric non-commuting cups in $\underline{\mu-\cp}$ that upon being flipped creates $\underline{\mu-\mq}$.

Now assume that ${\rm deg}(\underline{\mu}\la)\geq 2$. 
Suppose $\mq$ is not doubly covered by any $\mq'$ with $\mq'\in \{\mq^i\}$. Then we can write $D^\mu_\la = D^\mu_{\mu - \mq} D^{\mu - \mq}_\la$ and we get that 
 $$D^{\mu - \cp}_\mu D^\mu_\la = D^{\mu - \cp}_\mu D^{\mu}_{\mu - \mq}D^{\mu - \mq}_\la = D^{\mu - \cp}_{\mu - \cp - {\color{orange}q}^1} D^{\mu - \cp - {\color{orange}q}^1}_{\mu - \mq} D^{\mu - \mq}_\la = D^{\mu - \cp}_\la$$by 
 the doubly non-commuting relations and the definition of the light leaves basis, where ${\color{orange}q}^1$ is as above.

If not then we have that   $D^{\mu}_\la = D^\mu_{\mu - \mq'}D^{\mu - \mq'}_\la$ for some $\mq' \in \{\mq^i\}$ commuting with $\cp$. Then we have
 $$D^{\mu - \cp}_\mu D^\mu_\la = D^{\mu - \cp}_\mu D^{\mu}_{\mu - \mq'}D^{\mu -\mq'}_\la = D^{\mu - \cp}_{\mu - \cp - \mq'} D^{\mu - \cp - \mq'}_{\mu - \mq'}D^{\mu - \mq'}_\la = D^{\mu - \cp}_{\mu - \cp - \mq'}D^{\mu - \cp - \mq'}_\la = D^{\mu - \cp}_\la$$
where the second equality follows from the commuting relation, the third one follows by induction, as ${\rm deg}(\underline{\mu-\mq'}\la) = {\rm deg}(\underline{\mu}\la) - 1$, and the final equality follows by using the definition of the light leaves basis.

  \begin{rmk}\label{EQ-P}
  We have shown that a non-trivial extension between $L_{n}(\mu)$ and $L_{n}(\nu)$ exists precisely when $\mu = \nu \pm \cp$ (for some $\cp\in \underline{\nu}$ or $\cp \in \underline{\mu}$). This explicitly identifies the degree $1$ generators of the underlying path algebra, validating that \cref{presentation} provides the Ext-quiver and relations presentation of $\mathcal{H}_{(D_n, A_{n-1})}$.
  \end{rmk}

  \begin{rmk}\label{socle} 
  We set $k_\la = \max \{ k\geq 0 \, | \, {\rm DP}_k(\la)\neq \emptyset\}$. One can check that ${\rm DP}_{k_\la}(\la)$ consists of a single element $\mu_\la$.  To construct the cup diagram of $\underline{\mu_\la}$, start with the weight $\la$ and apply the following two steps:
 
\begin{enumerate}[leftmargin=*]
 \item  
Repeatedly find a pair of vertices labelled $\up$ $\down$ in order from left to right that are neighbours in the sense that there are only vertices already joined by cups in between. Join these pairs of vertices together with undecorated cups until there are no more such $\up$ $\down$ pairs. 

\item 
Repeatedly find pairs of $\down$ $\down$ vertices in order from left to right that are neighbours in the sense that there are only vertices joined by an undecorated cup in between. Join these pairs of vertices together with decorated cups.
	
\item We are left with at most one $\down$ node followed immediately by a sequence of $\up$'s. If there is a $\down$ node join it to the neighbouring $\up$ node with an undecorated anticlockwise cup. Then join the sequences of $\up$'s with decorated cups. If there is a final $\up$ or $\down$ node left remaining, draw a vertical decorated or undecorated ray from this node.
\end{enumerate}
One can then find the weight $\mu_\la$ in the obvious way.
 \end{rmk}

 Suppose $ \la = \mu_\la - \sum_{i}\mq^i$. 
Note that \cref{socle} implies $\mu_\la$ is characterised by the following two properties:
 \begin{enumerate}
 \item There is no $\nu\in\mptn$ with $\underline{\nu}\la$ oriented such that $\mu_\la = \nu - \cp$
 \item If $\cp\in\{\mq^i\}$ and $\mq$ is (doubly) covered by $\cp$ then $\mq \in \{\mq ^i\}$.
 \end{enumerate}
 
In particular, this implies that if $\mu \in {\rm DP}_k(\la)$ for $k<k_\la$, then either ($1$) or ($2$) above fails. In each of these cases we have seen that we can find some $h\in \mathcal{H}_{(D_n, A_{n-1})}$ and $\nu\in {\rm DP}_{k+1}(\la)$ such that $u_\nu = hu_\mu$.  Therefore the radical and socle filtration of $\Delta_{n}(\la)$ coincide with its grading filtration and the socle of $\Delta_{n}(\la)$ is simply given by $L_{n}(\mu_\la)$.
Hence we  have the following:

\begin{thm}\label{submodule}
Let $\la \in \mptn$. The strong Alperin diagram of the standard module $\Delta_{n}(\la)$ has vertex set labelled by the set 
 $\{ L_{n}(\mu)\, : \,  \mu\in {\rm DP}(\la)\}$  and edges
$$L_{n}(\mu) \longrightarrow L_{n}(\nu)$$whenever 
$\mu \in {\rm DP}_k(\la)$, $\nu \in {\rm DP}_{k+1}(\la)$ for some $k\geq 0$ and $\mu = \nu \pm \cp$ for some $\cp\in \underline{\nu}$ or $\cp \in \underline{\mu}$ respectively.
Moreover, the radical and socle filtrations both coincide with the grading filtration and $\Delta_{n}(\la)$ has simple socle isomorphic to $L_{n}(\mu_\la)$ (where $\mu_\la$ is constructed in \cref{socle}).
\end{thm}

\cref{submodule} is an example of a strong Alperin diagram in $\mathcal{H}_{(D_7, A_{6})}$.

 \begin{figure}[ht!]
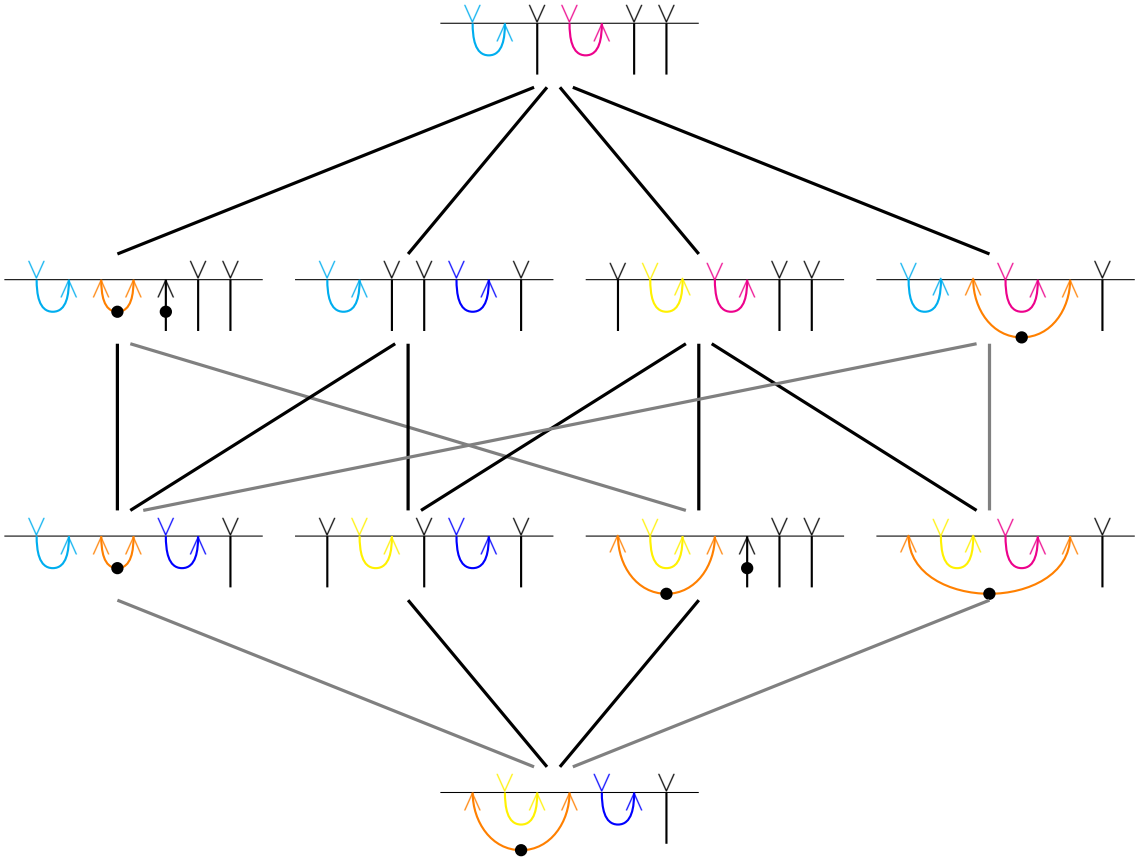

     $$   
.$$ 
    
    \caption{The strong Alperin diagram depicting the submodule structure of $\Delta_7(\la)$, where $\la=(1,2,1,1)\in \mathscr{P}_7$. 
    The simple head at the top of the diagram is indexed by the cup diagram $\underline{\la}\la$.
    The diagrams in the second and third row are the complete set of $\nu\in \mathscr{P}_7$ such that $\underline{\nu}\la$ is of degree $1$ and degree $2$, respectively.
     Finally the simple socle at the bottom is indexed by the unique $\mu_\la \in \mathscr{P}_7$ such that $\underline{\mu_\la}\la$ is of maximal degree ($3$).
    Black lines indicate an extension obtained by adding a cup and grey lines indicate extensions obtained by the removal of a cup.}
    \label{submodule}
    \end{figure}

               \bibliographystyle{amsalpha}   
\bibliography{Ben}

\newcommand{\etalchar}[1]{$^{#1}$}
\providecommand{\bysame}{\leavevmode\hbox to3em{\hrulefill}\thinspace}
\providecommand{\MR}{\relax\ifhmode\unskip\space\fi MR }
\providecommand{\MRhref}[2]{%
  \href{http://www.ams.org/mathscinet-getitem?mr=#1}{#2}
}
\providecommand{\href}[2]{#2}
\begin{thebibliography}{BDD{\etalchar{+}}24}

\bibitem[BDD{\etalchar{+}}24]{bowman2024quiverpresentationsschurweylduality}
Chris Bowman, Maud {De~Visscher}, Alice Dell'Arciprete, Amit Hazi, Rob Muth,
  and Catharina Stroppel, \emph{Quiver presentations and {S}chur--{W}eyl
  duality for {K}hovanov arc algebras},
  \href{https://arxiv.org/abs/2411.15520}{arXiv:2411.15520}, 2024.

\bibitem[BDF{\etalchar{+}}25]{bowman2023orientedtemperleyliebalgebrascombinatorial}
Chris Bowman, Maud {De~Visscher}, Niamh Farrell, Amit Hazi, and Emily Norton,
  \emph{Oriented temperley–lieb algebras and combinatorial kazhdan–lusztig
  theory}, Canadian Journal of Mathematics (2025), 1–43.

\bibitem[BDHN25]{ChrisHSP}
Chris Bowman, Maud {De~Visscher}, Amit Hazi, and Emily Norton, \emph{The
  anti-spherical {H}ecke categories for {H}ermitian symmetric pairs}, Adv.
  Math. \textbf{480} (2025), Paper No. 110501, 90. \MR{4953003}

\bibitem[BDHS23]{ChrisDyckPaper}
Chris Bowman, Maud {De~Visscher}, Amit Hazi, and Catharina Stroppel,
  \emph{Quiver presentations and isomorphisms of {H}ecke categories and
  {K}hovanov arc algebras}, 2023.

\bibitem[BH09]{BOEANDCO}
Brian~D. Boe and Markus Hunziker, \emph{Kostant modules in blocks of category
  {${\mathscr O}_S$}}, Comm. Algebra \textbf{37} (2009), no.~1, 323--356.
  \MR{2482826}

\bibitem[Boe88]{Boe1988-va}
Brian~D. Boe, \emph{Kazhdan-{L}usztig polynomials for {H}ermitian symmetric
  spaces}, Trans. Amer. Math. Soc. \textbf{309} (1988), no.~1, 279--294.
  \MR{957071}

\bibitem[Bre09]{brenti2009parabolic}
Francesco Brenti, \emph{Parabolic {K}azhdan-{L}usztig polynomials for
  {H}ermitian symmetric pairs}, Trans. Amer. Math. Soc. \textbf{361} (2009),
  no.~4, 1703--1729. \MR{2465813}

\bibitem[Bru03]{MR1937204}
Jonathan Brundan, \emph{Kazhdan-{L}usztig polynomials and character formulae
  for the {L}ie superalgebra {$\mathfrak {g} l(m|n)$}}, J. Amer. Math. Soc.
  \textbf{16} (2003), no.~1, 185--231. \MR{1937204}

\bibitem[BS10]{Brundan2010-mb}
Jonathan Brundan and Catharina Stroppel, \emph{Highest weight categories
  arising from {K}hovanov's diagram algebra. {II}. {K}oszulity}, Transform.
  Groups \textbf{15} (2010), no.~1, 1--45. \MR{2600694}

\bibitem[BS11a]{MR2918294}
\bysame, \emph{Highest weight categories arising from {K}hovanov's diagram
  algebra {I}: cellularity}, Mosc. Math. J. \textbf{11} (2011), no.~4,
  685--722, 821--822. \MR{2918294}

\bibitem[BS11b]{Brundan2011-ye}
Jonathan Brundan and Catharina Stroppel, \emph{Highest weight categories
  arising from khovanov's diagram algebra {III}: category $\mathfrak{O}$},
  Represent. Theory \textbf{15} (2011), 170--243.

\bibitem[BS24]{BSBS}
Jonathan Brundan and Catharina Stroppel, \emph{Semi-infinite highest weight
  categories}, Mem. Amer. Math. Soc. \textbf{293} (2024), no.~1459, vii+152.
  \MR{4684337}

\bibitem[CD11]{MR2813567}
Anton Cox and Maud {De~Visscher}, \emph{Diagrammatic {K}azhdan-{L}usztig theory
  for the (walled) {B}rauer algebra}, J. Algebra \textbf{340} (2011), 151--181.
  \MR{2813567}

\bibitem[EHP14]{MR3363009}
T.~Enright, M.~Hunziker, and A.~Pruett, \emph{Diagrams of {H}ermitian type,
  highest weight modules, and syzygies of determinantal varieties}, Symmetry:
  representation theory and its applications, Progr. Math., vol. 257,
  Birkh\"{a}user/Springer, New York, 2014, pp.~121--184. \MR{3363009}

\bibitem[ES86]{MR840583}
Thomas~J. Enright and Brad Shelton, \emph{Decompositions in categories of
  highest weight modules}, J. Algebra \textbf{100} (1986), no.~2, 380--402.
  \MR{840583}

\bibitem[ES16a]{Ehrig_Stroppel_2016}
Michael Ehrig and Catharina Stroppel, \emph{2-row {S}pringer fibres and
  {K}hovanov diagram algebras for type {D}}, Canad. J. Math. \textbf{68}
  (2016), no.~6, 1285--1333. \MR{3563723}

\bibitem[ES16b]{TypeDKhov}
\bysame, \emph{Diagrammatic description for the categories of perverse sheaves
  on isotropic {G}rassmannians}, Selecta Math. (N.S.) \textbf{22} (2016),
  no.~3, 1455--1536. \MR{3518556}

\bibitem[ES17]{ehrig2017category}
\bysame, \emph{On the category of finite-dimensional representations of {${\rm
  OSp}(r|2n)$}: {P}art {I}}, Representation theory---current trends and
  perspectives, EMS Ser. Congr. Rep., Eur. Math. Soc., Z\"urich, 2017,
  pp.~109--170. \MR{3644792}

\bibitem[ES23]{Strop-Eber}
Jens~Niklas Eberhardt and Catharina Stroppel, \emph{Standard extension algebras
  {I}: Perverse sheaves and fukaya calculus},
  \href{https://arxiv.org/abs/2310.09206}{arXiv:2310.09206}, 2023.

\bibitem[EW14]{elias2014hodgetheorysoergelbimodules}
Ben Elias and Geordie Williamson, \emph{The {H}odge theory of {S}oergel
  bimodules}, Ann. of Math. (2) \textbf{180} (2014), no.~3, 1089--1136.
  \MR{3245013}

\bibitem[EW16]{MR3555156}
B.~Elias and G.~Williamson, \emph{Soergel calculus}, Represent. Theory
  \textbf{20} (2016), 295--374. \MR{3555156}

\bibitem[Gre98]{MR1618912}
R.~M. Green, \emph{Generalized {T}emperley-{L}ieb algebras and decorated
  tangles}, J. Knot Theory Ramifications \textbf{7} (1998), no.~2, 155--171.
  \MR{1618912}

\bibitem[HNS24]{heidersdorf2024khovanovalgebrastypeb}
Thorsten Heidersdorf, Jonas Nehme, and Catharina Stroppel, \emph{Khovanov
  algebras of type {B} and tensor powers of the natural
  $\mathrm{OSp}$-representation}, 2024.

\bibitem[Kho00]{Khov}
Mikhail Khovanov, \emph{A categorification of the {J}ones polynomial}, Duke
  Math. J. \textbf{101} (2000), no.~3, 359--426. \MR{1740682}

\bibitem[LS81]{lascoux1981polynomes}
Alain Lascoux and Marcel-Paul Sch\"utzenberger, \emph{Polyn\^omes de {K}azhdan
  \&\ {L}usztig pour les grassmanniennes}, Young tableaux and {S}chur functors
  in algebra and geometry ({T}oru\'n, 1980), Ast\'erisque, vol. 87-88, Soc.
  Math. France, Paris, 1981, pp.~249--266. \MR{646823}

\bibitem[LS13]{LEJCZYK_STROPPEL_2013}
Tobias Lejczyk and Catharina Stroppel, \emph{A graphical description of
  {$(D_n,A_{n-1})$} {K}azhdan-{L}usztig polynomials}, Glasg. Math. J.
  \textbf{55} (2013), no.~2, 313--340. \MR{3040865}

\bibitem[LW22]{antiLW}
Nicolas Libedinsky and Geordie Williamson, \emph{The anti-spherical category},
  Adv. Math. \textbf{405} (2022), Paper No. 108509, 34. \MR{4437613}

\bibitem[Str09]{Strop1}
Catharina Stroppel, \emph{Parabolic category {$\mathscr O$}, perverse sheaves
  on {G}rassmannians, {S}pringer fibres and {K}hovanov homology}, Compos. Math.
  \textbf{145} (2009), no.~4, 954--992. \MR{2521250}

\end{thebibliography}

\end{document}